\documentclass[11pt]{amsart}
\usepackage{graphicx}
\usepackage{tikz}
\usepackage[mathcal]{euscript}
\usepackage{float}
\usepackage{verbatim}
\usepackage{color}

\usetikzlibrary{arrows,shapes,snakes,backgrounds,patterns}

\pgfdeclarelayer{background}
\pgfdeclarelayer{background2}
\pgfdeclarelayer{background2a}
\pgfdeclarelayer{background2b}
\pgfdeclarelayer{background3}
\pgfdeclarelayer{background4}
\pgfdeclarelayer{background5}
\pgfdeclarelayer{background6}
\pgfdeclarelayer{background7}

\pgfsetlayers{background7,background6,background5,background4,background3,background2b,background2a,background2,background,main}

\restylefloat{figure}

\newtheorem{theorem}{Theorem}[section]
\newtheorem{lemma}[theorem]{Lemma}
\newtheorem{corollary}[theorem]{Corollary}
\newtheorem{proposition}[theorem]{Proposition}

\theoremstyle{definition}

\theoremstyle{remark}
\newtheorem{remark}[theorem]{Remark}

\numberwithin{equation}{section}

\newcommand{\Hom}{\operatorname{Hom}}

\begin{document}

\title{A Turaev surface approach to Khovanov homology}

\author{Oliver T. Dasbach}
\address{Department of Mathematics\\
Louisiana State University\\
Baton Rouge, Louisiana}
\email{kasten@math.lsu.edu}
\thanks{The first author was partially supported by {NSF-DMS} 0806539 and {NSF-DMS FRG} 0456275.\\
The second author was partially supported by {NSF-DMS} 0602242 ({VIGRE})}
\author{Adam M. Lowrance}
\address{Department of Mathematics\\
Vassar College\\
Poughkeepsie, New York} 
\email{adlowrance@vassar.edu}

\subjclass{}
\date{}

\begin{abstract}
We introduce Khovanov homology for ribbon graphs and show that the Khovanov homology of a certain ribbon graph embedded on the Turaev surface of a link is isomorphic to the Khovanov homology of the link (after a grading shift). We also present a spanning quasi-tree model for the Khovanov homology of a ribbon graph.
\end{abstract}

\maketitle

\section{Introduction}

Khovanov \cite{Khovanov:homology} introduced a categorification of the Jones polynomial now known as Khovanov homology. The Khovanov homology of a link $L$ with diagram $D$, equivalently denoted $Kh(L)$ or $Kh(D)$, is a bigraded abelian group with homological grading $i$ and polynomial grading $j$. The reduced version of Khovanov homology, denoted either $\widetilde{Kh}(L)$ or $\widetilde{Kh}(D)$ is also a bigraded abelian group. The graded Euler characteristic of reduced Khovanov homology is the Jones polynomial $J_L(q)$ while the graded Euler characteristic of Khovanov homology is $(q+q^{-1}) J_L(q)$.

An {\em oriented ribbon graph} $\mathbb{G}$ is a graph $G$ together with an embedding into an oriented surface $\Sigma$ such that $\Sigma\setminus\mathbb{G}$ is a disjoint union of two-cells. The graph $G$ is called the {\em underlying graph of $\mathbb{G}$}. Two oriented ribbon graphs $\mathbb{G}$ and $\mathbb{G}^\prime$ embedded into surfaces $\Sigma$ and $\Sigma^\prime$ respectively are equivalent if there is an orientation preserving homeomorphism from $\Sigma$ to $\Sigma^\prime$ that restricts to a graph isomorphism from the underlying graph of $\mathbb{G}$ to the underlying graph of $\mathbb{G}^\prime$. The embedding into the surface determines a cyclic ordering of the half edges incident to each vertex, and oriented ribbon graphs are often depicted as graphs drawn in the plane with the cyclic ordering around each vertex given by counterclockwise rotation. The genus of $\mathbb{G}$, denoted $g(\mathbb{G})$, is the genus of the surface $\Sigma$. All ribbon graphs that we consider are oriented and are referred to just as ribbon graphs. 

The all-$A$ ribbon graph of a link diagram (defined in Section \ref{section:khovanov}) is an oriented ribbon graph embedded on the Turaev surface of the link diagram. Dasbach, Futer, Kalfagianni, Lin, and Stoltzfus \cite{DFKLS:GraphsOnSurfaces, DFKLS:DDD} interpret the Jones polynomial via the all-$A$ ribbon graph of a diagram of the link. We define the Khovanov homology $Kh(\mathbb{G})$ and reduced Khovanov homology $\widetilde{Kh}(\mathbb{G})$ of a ribbon graph $\mathbb{G}$. If $\mathbb{G}$ is the all-$A$ ribbon graph of a link diagram $L$, then up to a grading shift, the Khovanov homology of the ribbon graph $\mathbb{G}$ is isomorphic to the Khovanov homology of the link $L$. If $M$ is a bigraded abelian group and $r$ and $s$ are integers, let $M[r]\{s\}$ denote the group $M$ but with the homological grading shifted up by $r$ units and the polynomial grading shifted up by $s$ units. The first main theorem of the paper is:
\begin{theorem}
\label{theorem:maintheorem1}
Let $L$ be a link with diagram $D$ and all-$A$ ribbon graph $\mathbb{D}$. Suppose that $D$ has $n_+$ positive crossings and $n_-$ negative crossings. There are grading preserving isomorphisms
\begin{eqnarray*}
Kh(\mathbb{D})[-n_-]\{n_+-2n_-\} & \cong & Kh(L)~\text{and}\\
\widetilde{Kh}(\mathbb{D})[-n_-]\{n_+-2n_- \} & \cong & \widetilde{Kh}(L).
\end{eqnarray*}
\end{theorem}

Chmutov \cite{Chmutov:GeneralizedDuality} generalized the notion of the all-$A$ ribbon graph to virtual link diagrams (see Section \ref{section:virtual}). However, the all-$A$ ribbon graph $\mathbb{D}$ of a virtual link diagram $D$ is not necessarily orientable. We show that a generalization of Theorem \ref{theorem:maintheorem1} holds whenever $\mathbb{D}$ is orientable.

Champanerkar and Kofman \cite{ChampanerkarKofman:SpanningTrees} and independently Wehrli \cite{Wehrli:SpanningTrees} proved that there exists a complex whose homology is the reduced Khovanov homology of a link $L$ and whose generators are in one-to-one correspondence with the spanning trees of the checkerboard graph of a diagram $D$ of $L$. Such a complex is known as a {\em spanning tree model of Khovanov homology}.
A ribbon graph $\mathbb{G}$ can be represented by a surface $\Sigma_{\mathbb{G}}$ with boundary. The surface $\Sigma_{\mathbb{G}}$ is a two-dimensional regular neighborhood of $\mathbb{G}$  in $\Sigma$, as in Figure \ref{fig:fatgraph}. A {\em spanning quasi-tree $\mathbb{T}$} of a ribbon graph $\mathbb{G}$ is a spanning ribbon subgraph of $\mathbb{G}$ such that the surface $\Sigma_{\mathbb{T}}$ has only one boundary component. The spanning quasi-trees of a ribbon graph can be viewed as generalizations of spanning trees of a graph embedded in the plane in the following sense. If the genus of $\mathbb{G}$ is zero, then the spanning quasi-trees of $\mathbb{G}$ are precisely the spanning trees of the underlying graph of $\mathbb{G}$. However, if the genus of $\mathbb{G}$ is greater than zero, then $\mathbb{G}$ will have spanning quasi-trees of all different genera ranging from zero to the genus of $\mathbb{G}$.
Champanerkar, Kofman, and Stoltzfus \cite{CKS:GraphsOnSurfacesKhovanov}  show that the spanning trees of the checkerboard graph of a link diagram $D$ are in one-to-one correspondence with the spanning quasi-trees of the all-$A$ ribbon graph of $D$. Therefore, the spanning tree model of Khovanov homology for links may be viewed as a spanning quasi-tree model. We show that a spanning quasi-tree model can be obtained directly from our definition of the Khovanov homology of ribbon graphs.
\begin{theorem}
\label{theorem:maintheorem2}
Let $\mathbb{G}$ be a ribbon graph. There exists a complex $\widetilde{C}(\mathbb{G})$ whose generators are in one-to-one correspondence with the spanning quasi-trees of $\mathbb{G}$ and whose homology is $\widetilde{Kh}(\mathbb{G})$.
\end{theorem}
The homological and polynomial gradings of a generator in the spanning quasi-tree complex can be expressed via the quasi-tree activities defined by Champanerkar, Kofman, and Stoltzfus \cite{CKS:Quasi-trees}. Corollary \ref{cor:Jones} gives an expansion of the Jones polynomial as a summation over the spanning quasi-trees where each spanning quasi-tree is assigned a signed monomial in $q$.

We provide several applications of the quasi-tree model for the Khovanov homology of ribbon graphs. The first application compares the homological width of $\widetilde{Kh}(\mathbb{G})$ and the genus of $\mathbb{G}$. Recall that $\widetilde{Kh}(\mathbb{G})$ is bigraded with homological grading $i$ and polynomial grading $j$, and the summand in the $(i,j)$ bigrading is denoted $\widetilde{Kh}^{i,j}(\mathbb{G})$. The diagonal grading $\delta$ is defined by $\delta=j/2-i$, and the summand in diagonal grading $\delta$ is denoted $\widetilde{Kh}^{\delta}(\mathbb{G})$. Let $\delta_{\max}(\mathbb{G})=\max\{\delta~|~\widetilde{Kh}^\delta(\mathbb{G})\neq 0\},$ and $\delta_{\min} = \min\{\delta~|~\widetilde{Kh}^{\delta}(\mathbb{G})\neq 0\}$. The {\em homological width} of $\widetilde{Kh}(\mathbb{G})$, denoted $hw(\widetilde{Kh}(\mathbb{G}))$, is defined as $hw(\widetilde{Kh}(\mathbb{G})) = \delta_{\max}(\mathbb{G}) - \delta_{\min}(\mathbb{G})+1$. Our first application is the following theorem.
\begin{theorem}
\label{theorem:hw}
Let $\mathbb{G}$ be a ribbon graph. The genus of $\mathbb{G}$ gives an upper bound on the homological width of $\widetilde{Kh}(\mathbb{G})$. In particular, we have
$$hw(\widetilde{Kh}(\mathbb{G}))\leq g(\mathbb{G}) + 1.$$
\end{theorem}

The genus of $\mathbb{G}$ is zero precisely if $\mathbb{G}$ is the all-$A$ ribbon graph of some alternating link $L$. Theorem \ref{theorem:hw} implies that the Khovanov homology of $\mathbb{G}$ lies on adjacent diagonals, and Theorem \ref{theorem:maintheorem1} implies that $Kh(\mathbb{G})$ is isomorphic to $Kh(L)$ up to a prescribed grading shift. Lee \cite{Lee:KhovanovAlternating} proved that the Khovanov homology of an alternating link is supported on adjacent diagonals, and thus Theorem \ref{theorem:hw} can be viewed as a generalization of Lee's result. 

Let $V(\mathbb{G})$ and $E(\mathbb{G})$ denote the set of vertices and edges of $\mathbb{G}$ respectively. Also, let $F(\mathbb{G})$ denote the set of faces of $\mathbb{G}$, that is the set of disjoint disks of $\Sigma\setminus\mathbb{G}$. If $S$ is any finite set, then let $|S|$ denote the number of elements of $S$. A {\em loop} of $\mathbb{G}$ is an edge that connects a vertex to itself. If a ribbon graph $\mathbb{G}$ does not have any loops, then the quasi-tree model implies that the reduced Khovanov homology is isomorphic to $\mathbb{Z}$ in its minimum nontrivial polynomial grading. Define $j_{\min}(\mathbb{G})= \min\{ j ~|~\bigoplus_{i\in\mathbb{Z}}Kh^{i,j}(\mathbb{G})\neq 0\}$. 
\begin{theorem}
\label{theorem:loopless}
Suppose that $\mathbb{G}$ is a ribbon graph with no loops. Then $j_{\min}(\mathbb{G})=1-|V(\mathbb{G})|$ and
$$\bigoplus_{i\in\mathbb{Z}}\widetilde{Kh}^{i,1-|V(\mathbb{G})|}(\mathbb{G})\cong\widetilde{Kh}^{0,1-|V(\mathbb{G})|}(\mathbb{G})\cong\mathbb{Z}.$$
\end{theorem}
A ribbon graph where both $\mathbb{G}$ and its dual $\mathbb{G}^*$ (defined in Section \ref{section:ribbongraphs}) have no loops is called {\em adequate}. If a ribbon graph is adequate, then a corollary to Theorem \ref{theorem:loopless} states that the Khovanov homology in the maximum polynomial grading is also isomorphic to $\mathbb{Z}$. This corollary is the generalization to ribbon graphs of results by Khovanov \cite{Khovanov:Patterns} and Abe \cite{Abe:TuraevGenusAdequateKnot}. Define $j_{\max}(\mathbb{G})=\max\{j~|~\bigoplus_{i\in\mathbb{Z}}\widetilde{Kh}^{i,j}(\mathbb{G})\neq 0\}$.
\begin{corollary}
\label{cor:adequate}
Let $\mathbb{G}$ be an adequate ribbon graph. Then $j_{\min}(\mathbb{G})= 1 - |V(\mathbb{G})|$, $j_{\max}(\mathbb{G}) = |E(\mathbb{G})| + |F(\mathbb{G})|-1$, and
$$
\begin{array}{l c l c l}
\bigoplus_{i\in\mathbb{Z}}\widetilde{Kh}^{i,j_{\min}(\mathbb{G})}(\mathbb{G}) & \cong & \widetilde{Kh}^{0,1-|V(\mathbb{G})|}(\mathbb{G}) & \cong & \mathbb{Z}~\text{and}\\
\bigoplus_{i\in\mathbb{Z}}\widetilde{Kh}^{i,j_{\max}(\mathbb{G})}(\mathbb{G}) & \cong & \widetilde{Kh}^{|E(\mathbb{G})|,|E(\mathbb{G})| + |F(\mathbb{G})|-1}(\mathbb{G}) & \cong & \mathbb{Z}.
\end{array}
$$
\end{corollary}
Corollary \ref{cor:adequate} implies that the homological width of an adequate ribbon graph is determined by its genus (see Corollary \ref{cor:hwadequate}).

This paper is organized as follows. In Section \ref{section:ribbongraphs}, we review basic definitions for ribbon graphs. In Section \ref{section:homology}, we define the Khovanov homology of ribbon graphs. In Section \ref{section:khovanov}, we review the construction of Khovanov homology and show if $\mathbb{G}$ is the all-$A$ ribbon graph of a diagram of a link $L$, then $Kh(\mathbb{G})\cong Kh(L)$ up to a grading shift. In Section \ref{section:virtual}, we show how to construct the all-$A$ ribbon graph of a virtual link diagram and prove a generalization of Theorem \ref{theorem:maintheorem1} for virtual links whose all-$A$ ribbon graphs are orientable. In Section \ref{section:Reidemeister}, we define Reidemeister moves for ribbon graphs that generalize the Reidemeister moves for both classical and virtual links. We also show that our Khovanov homology of ribbon graphs is invariant under the ribbon graph Reidemeister moves. In Section \ref{section:quasi}, we construct the spanning quasi-tree model for Khovanov homology of ribbon graphs. We also show that the gradings in this complex can be express via activity words. In Section \ref{section:applications}, we provide several applications of our spanning quasi-tree model. In Section \ref{section:example}, we compute the Khovanov homology of an example ribbon graph.

{\bf Acknowledgment:} The first author thanks Christian Blanchet for fruitful discussions on Khovanov homology during a visit to Paris. We also thank the referee, whose suggestion to add Sections \ref{section:virtual} and \ref{section:Reidemeister} greatly improved the paper.

\section{Ribbon graphs}
\label{section:ribbongraphs}

In this section, we provide some basic definitions for ribbon graphs. A {\em ribbon subgraph $\mathbb{H}$} of a ribbon graph $\mathbb{G}$ is a subgraph $H$ of the underlying graph $G$ of $\mathbb{G}$ such that the cyclic order of the edges in $\mathbb{H}$ around each vertex is inherited from the cyclic order of the edges in $\mathbb{G}$. A ribbon subgraph $\mathbb{H}$ of $\mathbb{G}$ is {\em spanning} if $V(\mathbb{H})= V(\mathbb{G})$.

Recall that a ribbon graph $\mathbb{G}$ can be represented by its two-dimensional regular neighborhood $\Sigma_{\mathbb{G}}$ in $\Sigma$. 
There is a natural identification between the faces of $\mathbb{G}$ and the boundary components of $\Sigma_{\mathbb{G}}$, and we will use the notation $F(\mathbb{G})$ to equivalently denote both sets. 
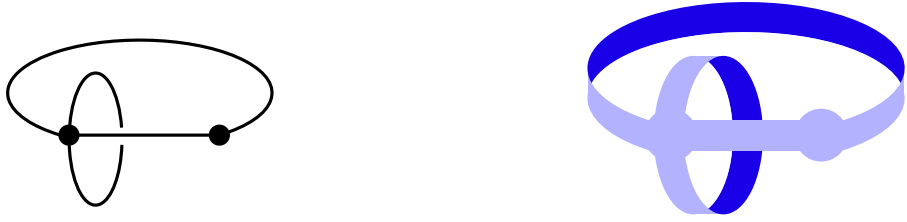
\begin{figure}[h]
$$\begin{tikzpicture}
\fill[blue!30!white] (0,0) circle (10pt);
\fill[blue!30!white] (2,0) circle (10pt);

\begin{pgfonlayer}{background}
\fill[blue!30!white] (0, 0.2) rectangle (2,-0.2);
\end{pgfonlayer}

\begin{pgfonlayer}{background2}
\def \firstellipse {(0.7,0) ellipse (15pt and 30pt)};
\def \secondellipse {(0.3,0) ellipse (15pt and 30pt)};
\fill[blue!90!red] \firstellipse;
\fill[blue!30] \secondellipse;
\begin{scope}
      \clip \firstellipse;
      \fill[white] \secondellipse;
\end{scope}
\end{pgfonlayer}
\begin{pgfonlayer}{background3}
\def \firstrectangle {(0.35,1.05) rectangle (0.65,-1.05)};
\fill[blue!30!] \firstrectangle;
\end{pgfonlayer}

\begin{pgfonlayer}{background4}
\def \firstarc {(2.3, .2) arc (-52:233:60pt and 25pt)};
\def \secondarc {(2.3, -.2) arc (-52:233:60pt and 25pt)};
\fill[blue!90!red] \firstarc;
\fill[blue!30!]\secondarc;
\begin{scope}
      	\clip \firstarc;
      	\fill[white] \secondarc;
\end{scope}
\end{pgfonlayer}
\begin{pgfonlayer}{background5}
\def \firstrectangle {(-1.1,.55) rectangle (3.1,.85)};
\fill[blue!30!] \firstrectangle;
\end{pgfonlayer}

\fill[black] (-6,0) circle (4pt);
\fill[black] (-8,0) circle (4pt);
\draw[very thick] (-8,0) -- (-6,0);
\draw[very thick] (-7.3,0.1) arc (10:355:10pt and 25pt);
\draw[very thick] (-6,0) arc (-53:233:50pt and 20pt);
\end{tikzpicture}$$
\caption{{\bf Left:} A ribbon graph $\mathbb{G}$. {\bf Right:} The surface $\Sigma_{\mathbb{G}}$.}
\label{fig:fatgraph}
\end{figure}
If the boundary components of $\Sigma_{\mathbb{G}}$ are capped off with disks, then one recovers the surface $\Sigma$. In the case where $\mathbb{G}$ is the all-$A$ ribbon graph of a link diagram $D$ (defined in Section \ref{section:khovanov}), then $\Sigma$ is known as {\em the Turaev surface of $D$}. The {\em dual ribbon graph $\mathbb{G}^*$} is constructed as follows. The vertices of $\mathbb{G}^*$ are the centers of the faces of $\mathbb{G}$. There is a one-to-one correspondence between the edges of $\mathbb{G}$ and $\mathbb{G}^*$. Locally, each edge $e$ in $\mathbb{G}$ has two (not necessarily distinct) disks attached to it to form $\Sigma$. For each edge $e$ in $\mathbb{G}$, there is a dual edge $e^*$ in $\mathbb{G}^*$ such that the endpoints of $e^*$ correspond to the two (not necessarily distinct) disks attached along $e$, and such that $e$ and $e^*$ transversely intersect exactly once in $\Sigma$. The cyclic order of the edges around each vertex in $\mathbb{G}^*$ is given by its embedding into $\Sigma$. Chmutov \cite{Chmutov:GeneralizedDuality} and Moffatt \cite{Moffatt:PartialDuality} provide other notions of duality in ribbon graphs. Compare also with \cite{Krushkal:Duality}.

Let $S(\mathbb{G})$ be the set of spanning ribbon subgraphs of the ribbon graph $\mathbb{G}$. Fix a bijection $S(\mathbb{G})\to S(\mathbb{G}^*)$ taking $\mathbb{H}$ to $\widehat{\mathbb{H}}$ as follows. Suppose that the edges of $\mathbb{G}$ are $e_1,\dots,e_n$ and the edges of $\mathbb{G}$ are $e_1^*,\dots,e_n^*$. Given a spanning ribbon subgraph $\mathbb{H}$ of $\mathbb{G}$, define $\widehat{\mathbb{H}}$ to be the spanning ribbon subgraph of $\mathbb{G}^*$ whose edge set is 
$E(\widehat{\mathbb{H}})=\{e_i^*\in E(\mathbb{G}^*)~|~e_i\notin E(\mathbb{H})\}.$ The ribbon graphs $\mathbb{H}$ and $\widehat{\mathbb{H}}$ can be mutually embedded into $\Sigma$ (though these embeddings are not necessarily cellular). Let $\Sigma_{\mathbb{H}}$ and $\Sigma_{\widehat{\mathbb{H}}}$ be two-dimensional regular neighborhoods of $\mathbb{H}$ and $\widehat{\mathbb{H}}$ inside of $\Sigma$. By taking suitably sized neighborhoods of $\mathbb{H}$ and $\widehat{\mathbb{H}}$ one may realize $\Sigma$ as $\Sigma_{\mathbb{H}}\cup \Sigma_{\widehat{\mathbb{H}}}$, which gives a bijection $\Phi$ between the boundary components $F(\mathbb{H})$ of $\mathbb{H}$ and the boundary components $F(\widehat{\mathbb{H}})$ of $\widehat{\mathbb{H}}$. See Figure \ref{figure:duality}.
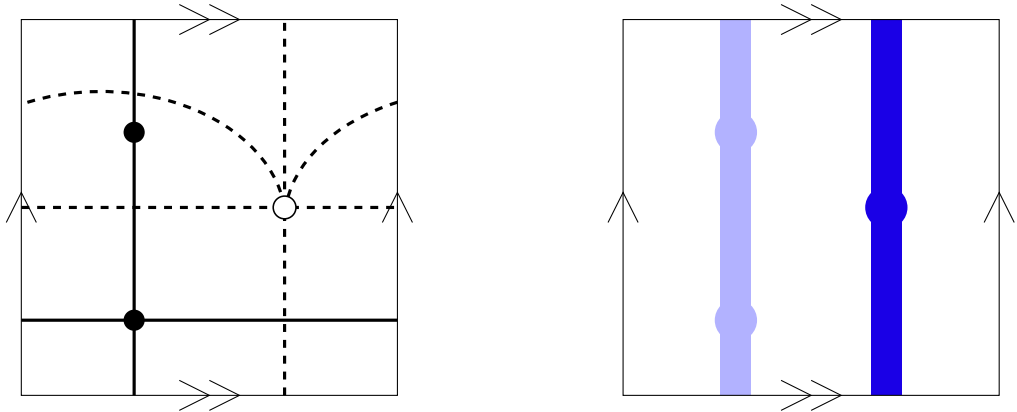
\begin{figure}[h]
$$\begin{tikzpicture}
\draw (0,0) rectangle (5,5);
\draw (-.2, 2.3) -- (0,2.7) -- (.2, 2.3);
\draw (4.8,2.3) -- (5,2.7) -- (5.2,2.3);
\draw (2.1,4.8) -- (2.5,5) -- (2.1, 5.2);
\draw (2.5,4.8) -- (2.9,5) -- (2.5,5.2);
\draw (2.1,-.2) -- (2.5,0) -- (2.1, .2);
\draw (2.5,-.2) -- (2.9,0) -- (2.5,.2);
\fill[black] (1.5,1) circle (4pt);
\fill[black] (1.5,3.5) circle (4pt);
\draw[very thick] (3.5,2.5) circle (4pt);
\fill[white] (3.5,2.5) circle (4pt);
\begin{pgfonlayer}{background}
\draw[very thick,dashed] (0,2.5) -- (5,2.5);
\draw[very thick,dashed] (3.5,0) -- (3.5,5);
\draw[very thick] (0,1) -- (5,1);
\draw[very thick] (1.5,0) -- (1.5,5);
\draw[dashed,very thick] (3.5,2.5) arc (7:115:70pt and 50pt);
\draw[dashed,very thick] (5,3.9) arc (115:175:74pt and 50pt);
\end{pgfonlayer}
\begin{scope}[xshift = 8 cm]
\draw (0,0) rectangle (5,5);
\draw (-.2, 2.3) -- (0,2.7) -- (.2, 2.3);
\draw (4.8,2.3) -- (5,2.7) -- (5.2,2.3);
\draw (2.1,4.8) -- (2.5,5) -- (2.1, 5.2);
\draw (2.5,4.8) -- (2.9,5) -- (2.5,5.2);
\draw (2.1,-.2) -- (2.5,0) -- (2.1, .2);
\draw (2.5,-.2) -- (2.9,0) -- (2.5,.2);
\fill[blue!30!] (1.5,1) circle (8pt);
\fill[blue!30!] (1.5,3.5) circle (8pt);
\fill[blue!90!red](3.5,2.5) circle (8pt);
\begin{pgfonlayer}{background}
\fill[blue!30!](1.3,0) rectangle (1.7,5);
\fill[blue!90!red] (3.3,0) rectangle (3.7,5);
\end{pgfonlayer}
\end{scope}
\end{tikzpicture}$$
\caption{{\bf Left:} The ribbon graph $\mathbb{G}$ is depicted with black vertices and solid edges, while the ribbon graph $\mathbb{G}^*$ is depicted with white vertices and dashed edges. {\bf Right:} Ribbon subgraphs $\mathbb{H}$ and $\widehat{\mathbb{H}}$ of $\mathbb{G}$ and $\mathbb{G}^*$ respectively.}
\label{figure:duality}
\end{figure}

Chmutov \cite{Chmutov:GeneralizedDuality} introduced a representation of a ribbon graph called an arrow presentation. An {\em arrow presentation} is a collection of non-nested circles in the plane together with a collection of oriented, labeled arcs lying on the circles, called {\em marking arrows}, such that each label appears on exactly two marking arrows. Two arrow presentations are equivalent if one can be obtained from the other by reversing all arrows on a circle and reversing the cyclic order of the arrows along it,  by reversing the orientation of all marking arrows that belong to some subset of the labels, or by changing the labeling set. 

A (possibly non-orientable) ribbon graph can be obtained from an arrow presentation by the following process. Consider each circle of the arrow presentation as the boundary of a disk corresponding to a vertex of the ribbon graph. Glue a band to each pair of marking arrows with the same label such that the orientation of the band agrees with the orientation of the marking arrows, as depicted in Figure \ref{fig:band}.
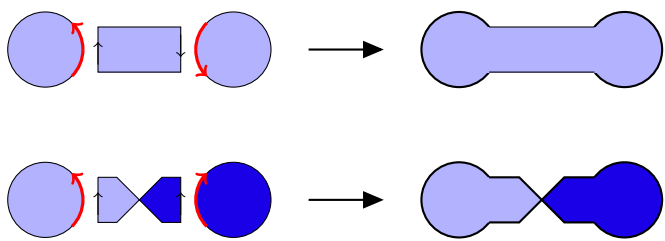
\begin{figure}[h]
$$\begin{tikzpicture}
\begin{pgfonlayer}{background2}
\fill[blue!30!white] (0,0) circle (.5 cm);
\fill[blue!30!white] (2.5,0) circle (.5 cm);
\fill[blue!30!white] (.7,.3) rectangle (1.8,-.3);
\end{pgfonlayer}
\begin{pgfonlayer}{background}
\draw (0,0) circle (.5cm);
\draw (2.5,0) circle (.5cm);
\end{pgfonlayer}
\draw[red,very thick] (0.5,0) arc (0:-45:.5 cm);
\draw[red,very thick, ->] (0.5,0) arc (0:45:.5cm);
\draw[red,very thick] (2,0) arc (180:135:.5cm);
\draw[red,very thick, ->] (2,0) arc (180:225:.5cm);
\draw (.7,.3) rectangle (1.8,-.3);
\draw[->] (.7,-.2) -- (.7,.1);
\draw[->] (1.8,.2) -- (1.8,-.1);

\draw[thick,>=triangle 45,->] (3.5,0) -- (4.5,0);

\begin{pgfonlayer}{background3}
\fill[blue!30!white] (5.5,0) circle (.5 cm);
\fill[blue!30!white] (7.7,0) circle (.5 cm);
\end{pgfonlayer}
\begin{pgfonlayer}{background2}
\draw[thick] (5.5,0) circle (.5cm);
\draw[thick] (7.7,0) circle (.5cm);
\end{pgfonlayer}
\begin{pgfonlayer}{background}
\fill[blue!30!white] (5.8,.3) rectangle (7.4,-.3);
\end{pgfonlayer}
\draw (5.9,.3) -- (7.3,.3);
\draw (5.9,-.3) -- (7.3,-.3);

\begin{scope}[yshift=-2cm]
\begin{pgfonlayer}{background2}
\fill[blue!30!white] (0,0) circle (.5 cm);
\fill[blue!90!red] (2.5,0) circle (.5 cm);
\fill[blue!90!red] (1.8, .3) -- (1.55,.3) -- (1.25,0) -- (1.55,-.3) -- (1.8,-.3) -- (1.8,.3);
\fill[blue!30!white] (.7,.3) -- (.95,.3) -- (1.25,0) -- (.95,-.3) -- (.7,-.3) -- (.7,.3);

\end{pgfonlayer}
\begin{pgfonlayer}{background}
\draw (0,0) circle (.5cm);
\draw (2.5,0) circle (.5cm);
\draw (.7,.3) -- (.95, .3) -- (1.55,-.3) -- (1.8,-.3) -- (1.8,.3) -- (1.55,.3) -- (.95,-.3) -- (.7,-.3) -- (.7,.3);
\end{pgfonlayer}

\draw[thick,>=triangle 45,->] (3.5,0) -- (4.5,0);

\draw[red,very thick] (0.5,0) arc (0:-45:.5 cm);
\draw[red,very thick, ->] (0.5,0) arc (0:45:.5cm);
\draw[red,very thick,->] (2,0) arc (180:135:.5cm);
\draw[red,very thick] (2,0) arc (180:225:.5cm);

\draw[->] (.7,-.2) -- (.7,.1);
\draw[->] (1.8,-.2) -- (1.8,.1);

\begin{pgfonlayer}{background3}
\fill[blue!30!white] (5.5,0) circle (.5 cm);
\fill[blue!90!red] (7.7,0) circle (.5 cm);
\end{pgfonlayer}
\begin{pgfonlayer}{background2}
\draw[thick] (5.5,0) circle (.5cm);
\draw[thick] (7.7,0) circle (.5cm);
\end{pgfonlayer}

\begin{pgfonlayer}{background}
\fill[blue!30!white] (5.8,.3) -- (6.3,.3) -- (6.6,0) -- (6.3,-.3) -- (5.8,-.3);
\fill[blue!90!red] (7.4,.3) -- (6.9,.3) -- (6.6,0) -- (6.9,-.3) -- (7.4,-.3);

\end{pgfonlayer}
\draw[thick] (5.9,.3) -- (6.3,.3) -- (6.9,-.3) -- (7.3,-.3);
\draw[thick] (5.9,-.3) -- (6.3,-.3) -- (6.9,.3) -- (7.3,.3);

\end{scope}
\end{tikzpicture}$$
\caption{Attaching bands to pairs of identically labeled marking arrows.}
\label{fig:band}
\end{figure}

Once bands are attached to every pair of identically labeled marking arrows, the resulting surface  may or may not be orientable. If the resulting surface is orientable, then it is the regular neighborhood of an oriented ribbon graph. We will not consider arrow presentations whose associated ribbon graphs are non-orientable. Figure \ref{fig:arrowpres} shows an arrow presentation for the ribbon graph depicted in Figure \ref{fig:fatgraph}.
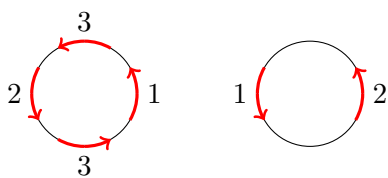
\begin{figure}[h]
$$\begin{tikzpicture}
\begin{pgfonlayer}{background}
\draw (0,0) circle (.7 cm);
\draw (3,0) circle (.7 cm);
\end{pgfonlayer}
\draw[very thick, red, ->] (0.7,0) arc (0:30:.7cm);
\draw[very thick, red] (0.7,0) arc (0:-30:.7cm);
\draw[very thick, red, ->] (0,0.7) arc (90:120:.7cm);
\draw[very thick, red] (0,0.7) arc (90:60:.7cm);
\draw[very thick, red] (-0.7,0) arc (180:150:.7cm);
\draw[very thick, red,->] (-0.7,0) arc (180:210:.7cm);
\draw[very thick, red] (0,-0.7) arc (270:240:.7cm);
\draw[very thick, red,->] (0,-0.7) arc (270:300:.7cm);

\draw[very thick, red] (2.3,0) arc (180:150:.7cm);
\draw[very thick, red, ->] (2.3,0) arc (180:210:.7cm);
\draw[very thick, red] (3.7,0) arc (0:-30:.7cm);
\draw[very thick, red, ->] (3.7,0) arc (0:30:.7cm);

\draw (0.7,0) node[right]{$1$};
\draw (0,0.7) node[above]{$3$};
\draw (-0.7,0) node[left]{$2$};
\draw (0,-0.7) node[below]{$3$};
\draw (2.3,0) node[left]{$1$};
\draw (3.7,0) node[right]{$2$};

\end{tikzpicture}$$
\caption{An arrow presentation for the ribbon graph in Figure \ref{fig:fatgraph}.}
\label{fig:arrowpres}
\end{figure}

\section{Khovanov homology of ribbon graphs}
\label{section:homology}

In this section, we introduce Khovanov homology and reduced Khovanov homology for ribbon graphs. 
We also prove a result about the Khovanov homology of the dual ribbon graph. The construction closely imitates Khovanov's original categorification of the Jones polynomial \cite{Khovanov:homology} (especially as interpreted by Bar-Natan \cite{Bar-Natan:Khovanov} and Viro \cite{Viro:Khovanov}).

\subsection{Khovanov homology for ribbon graphs}

In the cube of resolutions complex for the Khovanov homology of links, the vertices in the hypercube correspond to Kauffman states of the link diagram, while in our construction the vertices in the hypercube correspond to subsets of the edge set of the ribbon graph. In both constructions, the $\mathbb{Z}$-modules associated to the vertices and the maps between those $\mathbb{Z}$-modules are defined analogously.

A bigraded $\mathbb{Z}$-module $M$ is a $\mathbb{Z}$-module that has a direct sum decomposition $M=\bigoplus_{i,j\in\mathbb{Z}} M^{i,j}$, where each summand $M^{i,j}$ is said to have bigrading $(i,j)$. Alternatively, one can think of a bigrading on $M$ as an assignment of a bigrading $(i,j)$ to each element in a chosen basis of $M$. If $M=\bigoplus_{i,j} M^{i,j}$ and $N=\bigoplus_{k,l}N^{k,l}$ are bigraded $\mathbb{Z}$-modules, then both $M\oplus N$ and $M\otimes N$ are bigraded $\mathbb{Z}$-modules where $(M\oplus N)^{m,n} = M^{m,n}\oplus N^{m,n}$ and $(M\otimes N)^{m,n} = \bigoplus_{i+k = m, j+l=n} M^{i,j}\otimes N^{k,l}$. Moreover, if $r$ and $s$ are integers, then define $(M[r]\{s\})^{i,j} = M^{i-r,j-s}$.

Let $\mathbb{G}$ be a ribbon graph with edges $e_1,\dots,e_n$, and let $\{0,1\}^n$ denote the $n$-dimensional hypercube. Denote the vertices and edges of $\{0,1\}^n$ by $\mathcal{V}(n)$ and $\mathcal{E}(n)$ respectively. A vertex $I=(m_1,\dots,m_n)$ in the hypercube is an $n$-tuple of $0$'s and $1$'s. There is a directed edge $\xi\in\mathcal{E}(n)$ from a vertex $I=(m_1,\dots,m_n)$ to a vertex $J=(m^\prime_1,\dots,m^\prime_n)$ if there exists a $k$ with $1\leq k\leq n$ such that $m_k=0$, $m^\prime_k=1$, and if $i\neq k$, then $m_i=m^\prime_i$. Define the height $h(I)$ of a vertex $I=(m_1,\dots,m_n)$ by $h(I) = \sum_{i=1}^n m_i$. The set $S(\mathbb{G})$ of spanning ribbon subgraphs of $\mathbb{G}$ is in one-to-one correspondence with the vertices of the hypercube $\mathcal{V}(\mathbb{G})$. Each vertex $I=(m_1,\dots,m_n)\in\mathcal{V}(\mathbb{G})$ is associated to the spanning ribbon subgraph $\mathbb{G}(I)$ of $\mathbb{G}$ whose edge set is $E(\mathbb{G}(I))=\{e_i|m_i=1\}$.

There are $\mathbb{Z}$-modules associated to each vertex in $\mathcal{V}(n)$ and morphisms associated to each edge in $\mathcal{E}(n)$. Let $V$ be the free $\mathbb{Z}$-module with basis elements $v_+$ and $v_-$, and suppose that $v_+$ has bigrading $(0,1)$ and $v_-$ has bigrading $(0,-1)$. Associate the $\mathbb{Z}$-module $V(\mathbb{G}(I)) = V^{\otimes F(\mathbb{G}(I))}[h(I)]\{h(I)\}$ to each $I \in \mathcal{V}(n)$. One should view this as associating one tensor factor of $V$ to each boundary component of $\Sigma_{\mathbb{H}}$. Define $CKh(\mathbb{G})$ to be the direct sum $\bigoplus_{I\in\mathcal{V}(n)}V(\mathbb{G}(I))$. The $\mathbb{Z}$-module $CKh(\mathbb{G})$ is bigraded, and we write $CKh(\mathbb{G}) = \bigoplus_{i,j}CKh^{i,j}(\mathbb{G})$. The summand $CKh^{i,j}(\mathbb{G})$ is said to have {\em homological grading $i$} and {\em polynomial grading $j$}. It will sometimes be useful to consider all summands of $CKh(\mathbb{G})$ in a particular homological grading without specifying the polynomial grading; therefore, we let $CKh^{i,*}(\mathbb{G}) = \bigoplus_{j\in\mathbb{Z}}CKh^{i,j}(\mathbb{G})$.

Suppose that there is a directed edge $\xi\in\mathcal{E}(n)$ from a vertex $I=(m_1,\dots,m_n)\in\mathcal{V}(n)$ to a vertex $J=(m^\prime_1,\dots, m^\prime_n)\in\mathcal{V}(n)$. The spanning ribbon subgraph $\mathbb{G}(J)$ can be obtained from the spanning ribbon subgraph $\mathbb{G}(I)$ by adding a single edge.
The {\em height} of the edge $\xi$ is defined as $|\xi| = E(\mathbb{G}(I))$, which is the height of the vertex from which the edge originates. Each edge in the hypercube has an associated map $d_\xi:V(\mathbb{G}(I))\to V(\mathbb{G}(J))$. Define $\mathbb{Z}$-linear maps $m:V\to V\otimes V$ and $\Delta:V\otimes V\to V$ by
$$\begin{array}{l c l}
m:V\to V\otimes V & \quad &
     m:\begin{cases}
      v_+\otimes v_-\mapsto v_- &
      v_+\otimes v_+\mapsto v_+ \\
      v_-\otimes v_+\mapsto v_- &
      v_-\otimes v_-\mapsto 0
    \end{cases}\\
    \Delta: V\otimes V\to V & \quad &
   \Delta:\begin{cases}
      v_+ \mapsto v_+\otimes v_- + v_-\otimes v_+ &\\
      v_- \mapsto v_-\otimes v_- &
    \end{cases}.
  \end{array}$$
Adding the edge $e$ to $\mathbb{G}(I)$ either merges two boundary components of $\Sigma_{\mathbb{G}(I)}$ into one boundary component of $\Sigma_{\mathbb{G}(J)}$ or splits one boundary component of $\Sigma_{\mathbb{G}(I)}$ into two boundary components of $\Sigma_{\mathbb{G}(J)}$. Define $d_\xi:V(\mathbb{G}(I))\to V(\mathbb{G}(J))$ to be the identity on the tensor factors corresponding to boundary components that do not change when adding the edge $e$. If adding the edge $e$ to $\mathbb{G}(I)$ merges two boundary components, then define $d_\xi$ to be the map $m:V\to V\otimes V$ on the tensor factors corresponding to merging boundary components, and if adding the edge $e$ to $\mathbb{G}(I)$ splits one boundary component into two, then define $d_\xi$ to be the map $\Delta:V \to V \otimes V$ on the tensor factor corresponding to the splitting boundary component.

Suppose that $I_{00}, I_{10}, I_{01},$ and $I_{11}$ are vertices in $\mathcal{V}(n)$ that agree in all but two coordinates $k$ and $l$, and whose $k$ and $l$ coordinates are given by their subscripts. Let $\xi_{*0},\xi_{0*},\xi_{1*}$, and $\xi_{*1}$ be the edges in the hypercube from $I_{00}$ to $I_{10}$, from $I_{00}$ to $I_{01}$, from $I_{10}$ to $I_{11}$, and from $I_{01}$ to $I_{11}$ respectively. The edge maps around this square commute, that is
$$d_{\xi_{1*}}\circ d_{\xi_{*0}} = d_{\xi_{*1}}\circ d_{\xi_{0*}}.$$
In order to ensure that $d \circ d = 0$, it is necessary that the edge maps around any square anti-commute. An {\em edge assignment on $\mathcal{E}(n)$} is a map $\epsilon:\mathcal{E}(n)\to\{\pm 1\}$ such that each square in the hypercube has an odd number of edges $\xi$ for which $\epsilon(\xi) = -1$. Given such an edge assignment, we have
$$\epsilon(\xi_{1*})d_{\xi_{1*}}\circ \epsilon(\xi_{*0})d_{\xi_{*0}} =-\epsilon(\xi_{*1}) d_{\xi_{*1}}\circ \epsilon(\xi_{0*})d_{\xi_{0*}}.$$ 
Proposition \ref{prop:edge_assign} below states the choice of edge assignment does not change the isomorphism type of the chain complex. If we wish to highlight the choice of the edge assignment, we denote the complex by $(CKh(\mathbb{G}),d_\epsilon)$; however we will often hide this choice and denote the complex by only $(CKh(\mathbb{G}),d)$ or just $CKh(\mathbb{G})$.

Suppose the edges of $\mathbb{G}$ are $e_1,\dots,e_n$ and fix an ordering on the edges where $e_i<e_j$ if and only if $i<j$. An edge assignment on $\mathcal{E}(n)$ can be constructed as follows. Suppose that $\xi$ is a directed edge from vertex $I$ to vertex $J$, where $I$ and $J$ differ only at the $k$th coordinate. Suppose that $I=(m_1,\dots,m_n)$, and define $\epsilon(\xi)=(-1)^l$ where $l=|\{m_i|m_i=1~\text{and}~i<k\}|$.

The differential $d^i:CKh^{i,*}(\mathbb{G})\to CKh^{i+1,*}(\mathbb{G})$ is defined by taking the signed sum of the edge maps $d_\xi$. Define $d^i:=\sum_{|\xi|=i} \epsilon(\xi) d_\xi$. Observe that $d^i$ preserves the polynomial grading, and thus $d^i = \sum_{j\in\mathbb{Z}} d^{i,j}$ where $d^{i,j} = d^i|_{CKh^{i,j}(\mathbb{G})}$. Since the signed edge maps around any square of the hypercube anticommute, it follows that $d\circ d = 0$. The Khovanov homology of the ribbon graph $\mathbb{G}$ is defined to be $$Kh(\mathbb{G}) = \bigoplus_{i,j\in\mathbb{Z}} Kh^{i,j}(\mathbb{G}),$$
where $Kh^{i,j}(\mathbb{G}) = \ker d^{i,j} / \text{im}~d^{i-1,j}$.

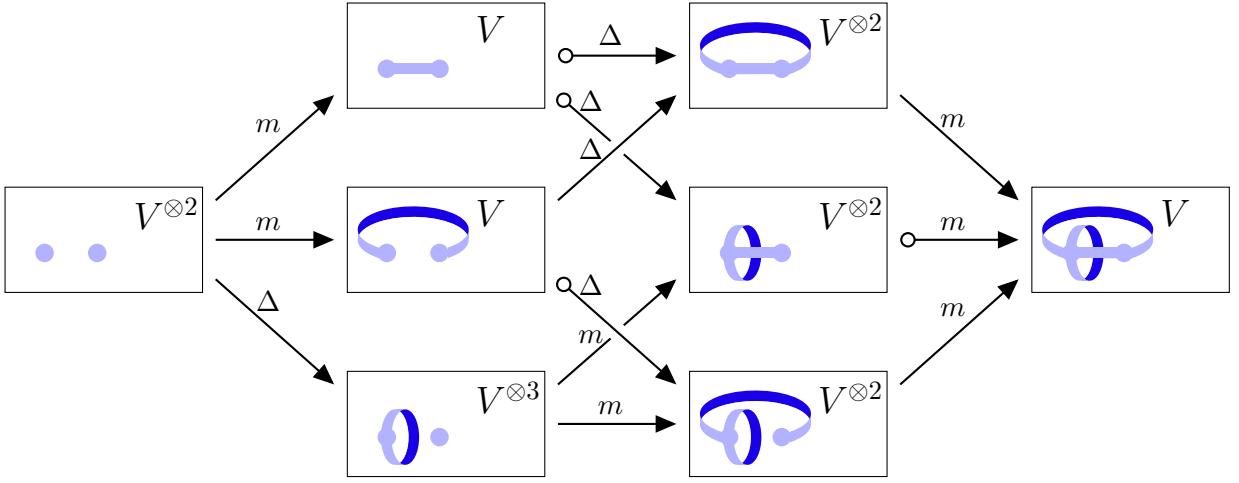
\begin{figure}[h]
$$\begin{tikzpicture}[scale=.35]
\draw (-1.5,-1.5) rectangle (6,2.5);
\fill[blue!30!white] (0,0) circle (10pt);
\fill[blue!30!white] (2,0) circle (10pt);
\draw (3,1.5) node[right]{\Large{$V^{\otimes 2}$}};
\begin{scope}[xshift = 13cm]
\draw (-1.5,-1.5) rectangle (6,2.5);
\fill[blue!30!white] (0,0) circle (10pt);
\fill[blue!30!white] (2,0) circle (10pt);
\begin{pgfonlayer}{background4}
\def \firstarc {(2.3, .2) arc (-52:233:60pt and 25pt)};
\def \secondarc {(2.3, -.2) arc (-52:233:60pt and 25pt)};
\fill[blue!90!red] \firstarc;
\fill[blue!30!]\secondarc;
\begin{scope}
      	\clip \firstarc;
      	\fill[white] \secondarc;
\end{scope}
\fill[white](2.3,-.22) rectangle (-.3,.22);
\end{pgfonlayer}
\begin{pgfonlayer}{background5}
\def \firstrectangle {(-1.1,.55) rectangle (3.1,.85)};
\fill[blue!30!] \firstrectangle;
\end{pgfonlayer}
\draw (3,1.5) node[right]{\Large{$V$}};
\end{scope}
\begin{scope}[xshift=13cm, yshift=7cm]
\draw (-1.5,-1.5) rectangle (6,2.5);
\fill[blue!30!white] (0,0) circle (10pt);
\fill[blue!30!white] (2,0) circle (10pt);

\begin{pgfonlayer}{background}
\fill[blue!30!white] (0, 0.2) rectangle (2,-0.2);
\end{pgfonlayer}
\draw (3,1.5) node[right]{\Large{$V$}};
\end{scope}
\begin{scope}[xshift=13cm, yshift=-7cm]
\draw (-1.5,-1.5) rectangle (6,2.5);
\fill[blue!30!white] (0,0) circle (10pt);
\fill[blue!30!white] (2,0) circle (10pt);

\begin{pgfonlayer}{background2}
\def \firstellipse {(0.7,0) ellipse (15pt and 30pt)};
\def \secondellipse {(0.3,0) ellipse (15pt and 30pt)};
\fill[blue!90!red] \firstellipse;
\fill[blue!30] \secondellipse;
\begin{scope}
      \clip \firstellipse;
      \fill[white] \secondellipse;
\end{scope}
\end{pgfonlayer}
\begin{pgfonlayer}{background3}
\def \firstrectangle {(0.35,1.05) rectangle (0.65,-1.05)};
\fill[blue!30!] \firstrectangle;
\end{pgfonlayer}
\draw (3,1.5) node[right]{\Large{$V^{\otimes{3}}$}};
\end{scope}
\begin{scope}[xshift=26cm]
\draw (-1.5,-1.5) rectangle (6,2.5);
\fill[blue!30!white] (0,0) circle (10pt);
\fill[blue!30!white] (2,0) circle (10pt);

\begin{pgfonlayer}{background}
\fill[blue!30!white] (0, 0.2) rectangle (2,-0.2);
\end{pgfonlayer}

\begin{pgfonlayer}{background2}
\def \firstellipse {(0.7,0) ellipse (15pt and 30pt)};
\def \secondellipse {(0.3,0) ellipse (15pt and 30pt)};
\fill[blue!90!red] \firstellipse;
\fill[blue!30] \secondellipse;
\begin{scope}
      \clip \firstellipse;
      \fill[white] \secondellipse;
\end{scope}
\end{pgfonlayer}
\begin{pgfonlayer}{background3}
\def \firstrectangle {(0.35,1.05) rectangle (0.65,-1.05)};
\fill[blue!30!] \firstrectangle;
\end{pgfonlayer}
\draw (3,1.5) node[right]{\Large{$V^{\otimes{2}}$}};
\end{scope}
\begin{scope}[xshift=26cm, yshift=7cm]
\draw (-1.5,-1.5) rectangle (6,2.5);
\fill[blue!30!white] (0,0) circle (10pt);
\fill[blue!30!white] (2,0) circle (10pt);

\begin{pgfonlayer}{background}
\fill[blue!30!white] (0, 0.2) rectangle (2,-0.2);
\end{pgfonlayer}

\begin{pgfonlayer}{background4}
\def \firstarc {(2.3, .2) arc (-52:233:60pt and 25pt)};
\def \secondarc {(2.3, -.2) arc (-52:233:60pt and 25pt)};
\fill[blue!90!red] \firstarc;
\fill[blue!30!]\secondarc;
\begin{scope}
      	\clip \firstarc;
      	\fill[white] \secondarc;
\end{scope}
\end{pgfonlayer}
\begin{pgfonlayer}{background5}
\def \firstrectangle {(-1.1,.55) rectangle (3.1,.85)};
\fill[blue!30!] \firstrectangle;
\end{pgfonlayer}
\draw (3,1.5) node[right]{\Large{$V^{\otimes{2}}$}};
\end{scope}
\begin{scope}[xshift=26cm, yshift=-7cm]
\draw (-1.5,-1.5) rectangle (6,2.5);
\fill[blue!30!white] (0,0) circle (10pt);
\fill[blue!30!white] (2,0) circle (10pt);

\begin{pgfonlayer}{background2}
\def \firstellipse {(0.7,0) ellipse (15pt and 30pt)};
\def \secondellipse {(0.3,0) ellipse (15pt and 30pt)};
\fill[blue!90!red] \firstellipse;
\fill[blue!30] \secondellipse;
\begin{scope}
      \clip \firstellipse;
      \fill[white] \secondellipse;
\end{scope}
\end{pgfonlayer}
\begin{pgfonlayer}{background3}
\def \firstrectangle {(0.35,1.05) rectangle (0.65,-1.05)};
\fill[blue!30!] \firstrectangle;
\end{pgfonlayer}

\begin{pgfonlayer}{background4}
\def \firstarc {(2.3, .2) arc (-52:233:60pt and 25pt)};
\def \secondarc {(2.3, -.2) arc (-52:233:60pt and 25pt)};
\fill[blue!90!red] \firstarc;
\fill[blue!30!]\secondarc;
\begin{scope}
      	\clip \firstarc;
      	\fill[white] \secondarc;
\end{scope}
\fill[white](2.3,-.22) rectangle (-.3,.22);
\end{pgfonlayer}
\begin{pgfonlayer}{background5}
\def \firstrectangle {(-1.1,.55) rectangle (3.1,.85)};
\fill[blue!30!] \firstrectangle;
\end{pgfonlayer}
\draw (3,1.5) node[right]{\Large{$V^{\otimes{2}}$}};
\end{scope}
\begin{scope}[xshift=39cm]
\draw (-1.5,-1.5) rectangle (6,2.5);
\fill[blue!30!white] (0,0) circle (10pt);
\fill[blue!30!white] (2,0) circle (10pt);

\begin{pgfonlayer}{background}
\fill[blue!30!white] (0, 0.2) rectangle (2,-0.2);
\end{pgfonlayer}

\begin{pgfonlayer}{background2}
\def \firstellipse {(0.7,0) ellipse (15pt and 30pt)};
\def \secondellipse {(0.3,0) ellipse (15pt and 30pt)};
\fill[blue!90!red] \firstellipse;
\fill[blue!30] \secondellipse;
\begin{scope}
      \clip \firstellipse;
      \fill[white] \secondellipse;
\end{scope}
\end{pgfonlayer}
\begin{pgfonlayer}{background3}
\def \firstrectangle {(0.35,1.05) rectangle (0.65,-1.05)};
\fill[blue!30!] \firstrectangle;
\end{pgfonlayer}

\begin{pgfonlayer}{background4}
\def \firstarc {(2.3, .2) arc (-52:233:60pt and 25pt)};
\def \secondarc {(2.3, -.2) arc (-52:233:60pt and 25pt)};
\fill[blue!90!red] \firstarc;
\fill[blue!30!]\secondarc;
\begin{scope}
      	\clip \firstarc;
      	\fill[white] \secondarc;
\end{scope}
\end{pgfonlayer}
\begin{pgfonlayer}{background5}
\def \firstrectangle {(-1.1,.55) rectangle (3.1,.85)};
\fill[blue!30!] \firstrectangle;
\end{pgfonlayer}
\draw (3,1.5) node[right]{\Large{$V$}};
\end{scope}
\begin{scope}[>=triangle 45]
	\draw[->,thick] (6.5,.5) -- (11,.5);
	\draw (8.5,.5) node[above]{$m$};
	\draw[->,thick] (6.5,2) -- (11, 6);
	\draw (8.5, 4.2) node[above]{$m$};
	\draw[->,thick] (6.5,-1) -- (11,-5);
	\draw (8.5, -2.6) node[above]{$\Delta$};
	\begin{scope}[xshift = 13cm]
		\draw[->,thick] (6.5,2) -- (11,6);
		\draw (7.75,3.2) node[above]{$\Delta$};
		\draw[o->,thick](6.5,-1)--(11,-5);
		\draw (7.75,-2) node[above]{$\Delta$};
	\end{scope}
	\begin{scope}[xshift=13cm,yshift = 7cm]
		\draw[o->,thick](6.5,.5)--(11,.5);
		\draw (8.5,.5) node[above]{$\Delta$};
		\draw[o-,thick] (6.5,-1) -- (8.5,-2.75);
		\draw (7.75,-2) node[above]{$\Delta$};
		\draw[->,thick] (9,-3.25)--(11,-5);
	\end{scope}
	\begin{scope}[xshift=13cm, yshift=-7cm]
		\draw[->,thick](6.5,.5) -- (11,.5);
		\draw (8.5,.5) node[above]{$m$};
		\draw[thick] (6.5,2) -- (8.5,3.75);
		\draw (7.75,3.2) node[above]{$m$};
		\draw[->,thick](9,4.25) -- (11,6);
	\end{scope}
	\begin{scope}[xshift = 26cm, yshift=7cm]
		\draw[->,thick](6.5,-1) -- (11,-5);
		\draw (8.5, -2.6) node[above]{$m$};
	\end{scope}
	\begin{scope}[xshift = 26cm]
		\draw[o->,thick] (6.5,.5) -- (11,.5);
		\draw (8.5,.5) node[above]{$m$};
	\end{scope}
	\begin{scope}[xshift = 26cm, yshift = -7cm]
		\draw[->,thick] (6.5,2) -- (11, 6);
		\draw (8.5, 4.2) node[above]{$m$};
	\end{scope}
\end{scope}
\end{tikzpicture}$$
\caption{The set of ribbon subgraphs arranged into a hypercube. Each vertex $\mathbb{G}(I)$ is labeled with $V^{\otimes F(\mathbb{G}(I))}$. The edges of the cube are labeled according either $m$ or $\Delta$ corresponding to a merge or a split of boundary components respectively. If the sign of an edge map is negative, then there is a small circle on the tail of the edge.}
\end{figure}

The construction of the chain complex $CKh(\mathbb{G})$ depends on an edge assignment $\epsilon:\mathcal{E}(n)\to\{\pm 1\}$. However, using a proof adapted from Ozsv\'ath, Rasmussen, and Szab\'o \cite{OzsvathSzabo:OddKhovanov}, one can show that complexes with different edge assignments are isomorphic.

\begin{proposition}
\label{prop:edge_assign}
Let $\epsilon$ and $\epsilon^\prime$ be edge assignments on $\mathcal{E}(n)$. Then $(CKh(\mathbb{G}),d_\epsilon)\cong(CKh(\mathbb{G}),d_{\epsilon^\prime})$.
\end{proposition}
\begin{proof}
The hypercube $\{0,1\}^n$ is a simplicial complex. We consider the edge assignments $\epsilon$ and $\epsilon^\prime$ as $1$-cochains in $\Hom(C_1,\mathbb{F}_2)$ where $C_1$ is the space of $1$-chains and $\mathbb{F}_2$ is the field of two elements. Since both edge assignments assign a $-1$ to  an odd number of edges around each square, it follows that $\epsilon\cdot\epsilon^\prime$ is a $1$-cocycle. Because the hypercube is contractible, the product of the edge assignments $\epsilon\cdot\epsilon^\prime$ is the coboundary of a $0$-cochain, that is there exists $\eta:\mathcal{V}(n)\to\{\pm 1\}$ such that $\eta(I)\eta(J)=\epsilon(\xi)\epsilon^\prime(\xi)$ if $\xi$ is an edge between vertices $I$ and $J$.

Let $\psi:(CKh(\mathbb{G}),d_{\epsilon})\to(CKh(\mathbb{G}),d_{\epsilon^\prime})$ be the map which when restricted to $V(\mathbb{G}(I))$ is multiplication by $\eta(I)$. Then $\psi$ is an isomorphism from $(CKh(\mathbb{G}),d_{\epsilon})\to(CKh(\mathbb{G},d_{\epsilon^\prime})$. 
\end{proof}

\subsection{Reduced homology}
\label{subsec:reduced}

In the construction of $CKh(\mathbb{G})$, one associates a tensor factor of $V$ to each boundary component of $\Sigma_{\mathbb{G}(I)}$. Suppose that there is a marked point on the boundary of a vertex of $\Sigma_{\mathbb{G}}$ that misses the bands attached for each edge. Let $\widetilde{F}(\mathbb{G}(I))$ denote the set of boundary components of $\Sigma_{\mathbb{G}(I)}$ without marked points. Note that $|\widetilde{F}(\mathbb{G}(I))|= |F(\mathbb{G}(I))| - 1$. For each vertex $I\in\mathcal{V}(n)$, one can consider $V(\mathbb{G}(I))$ as $V\otimes V^{\otimes \widetilde{F}(\mathbb{G}(I))} = (\mathbb{Z}v_+ \otimes V^{\otimes \widetilde{F}(\mathbb{G}(I))}) \oplus (\mathbb{Z}v_- \otimes V^{\otimes \widetilde{F}(\mathbb{G}(I))})$, where the $\mathbb{Z}v_+$ and $\mathbb{Z}v_-$ are the two summands of $V$ associated to the boundary component of $\Sigma_{\mathbb{G}(I)}$ that contains the marked point. Define $\widetilde{V}(\mathbb{G}(I)) = (\mathbb{Z}v_-\otimes V^{\otimes \widetilde{F}(\mathbb{G}(I))})\{1\}$ where as before, the $\mathbb{Z}v_-$ corresponds to the boundary component of $\Sigma_{\mathbb{G}(I)}$ that contains the marked point. 

Let $\widetilde{CKh}(\mathbb{G}) = \bigoplus_{I\in\mathcal{V}(n)} \widetilde{V}(\mathbb{G}(I))$, and define $\widetilde{d}:= d|_{\widetilde{CKh}(\mathbb{G})}$. Since the range of $\widetilde{d}$ is a subset of $\widetilde{CKh}(\mathbb{G})$, it follows that $(\widetilde{CKh}(\mathbb{G}), \widetilde{d})$ forms a chain complex. The homology of this chain complex $\widetilde{Kh}(\mathbb{G})$ is called the reduced Khovanov homology of $\mathbb{G}$.

\subsection{Homology of the dual ribbon graph}
\label{subsec:dual}

Throughout this subsection, let $\mathbb{G}$ be a ribbon graph and let $\mathbb{G}^*$ be the dual ribbon graph. In what follows we show that the Khovanov complex of $\mathbb{G}^*$ is isomorphic to the dual complex of the Khovanov complex of $\mathbb{G}$.

If $M$ is a $\mathbb{Z}$-module, then define the dual of $M$ by $M^*=\Hom(M,\mathbb{Z})$, and if $f:M\to N$ is a $\mathbb{Z}$-module homomorphism, then the dual homomorphism $f^*:M^*\to N^*$ is defined by $f^*(\phi) = \phi \circ f$. Let $(C,\partial)$ denote the complex
$$\cdots C^i\xrightarrow{\partial^i}C^{i+1}\to\cdots.$$
The dual complex $(C^*,\partial^*)$ is the complex where $(C^*)^i=(C^{-i})^*$ and $(\partial^*)^i$ is the dual of $\partial^{-i-1}$. When there is a polynomial grading on $C$ that $\partial$ preserves (as is the case with the Khovanov homology defined above), define $C^*$ to have the opposite polynomial grading, i.e. $(C^*)^{i,j}=(C^{-i,-j})^*$.

\begin{proposition}
\label{prop:duality}
Let $\mathbb{G}$ be a ribbon graph with $n$ edges, and let $\mathbb{G}^*$ be the dual ribbon graph. The Khovanov complex of $\mathbb{G}^*$ is isomorphic to the dual of the Khovanov complex of $\mathbb{G}$, that is
\begin{eqnarray*}
CKh(\mathbb{G^*}) & \cong & CKh(\mathbb{G})^*[n]\{n\}~\text{and}\\
\widetilde{CKh}(\mathbb{G}^*) & \cong & \widetilde{CKh}(\mathbb{G})^*[n]\{n\}.
\end{eqnarray*}
\end{proposition}
\begin{proof} We prove the proposition for $CKh(\mathbb{G}^*)$; the result for $\widetilde{CKh}(\mathbb{G}^*)$ is proved similarly.
Let $\{\widehat{0,1}\}^n$ be an $n$-dimensional hypercube with vertex set $\widehat{\mathcal{V}}(n)$ and edge set $\widehat{\mathcal{E}}(n)$. The one-skeleton of the hypercube $\{\widehat{0,1}\}^n$ is the same underlying graph as the one-skeleton of $\{0,1\}^n$ except that the edges in $\widehat{\mathcal{E}}(n)$ are in the opposite direction as the edges in $\mathcal{E}(n)$. If $I=(m_1,\dots,m_n)$ is a vertex in $\mathcal{V}(n)$, define its dual vertex $\widehat{I} = (\widehat{m}_1,\dots,\widehat{m}_n)$ in $\widehat{\mathcal{V}}(n)$ to be the vertex where $m_i+\widehat{m}_i\equiv 1\mod 2$ for $1\leq i \leq n$. The complexes $CKh(\mathbb{G})$ and $CKh(\mathbb{G}^*)$ will use the hypercube $\{0,1\}^n$, while the complex $CKh(\mathbb{G})^*$ will use the dual hypercube $\{\widehat{0,1}\}^n$.

First, we show that for each vertex $I\in\mathcal{V}(n)$, we have a grading preserving isomorphism 
$$\Psi\circ\Phi_*:V(\mathbb{G}^*(I))\xrightarrow{\cong} V(\mathbb{G}(\widehat{I}))^*[n]\{n\}.$$ Next we show that if $\xi$ is an edge in $\mathcal{E}(n)$ from $I$ to $J$ and $\widehat{\xi}$ is the dual edge in $\widehat{\mathcal{E}}(n)$ from $\widehat{I}$ to $\widehat{J}$, then the edge maps $d_\xi:\mathcal{V}(\mathbb{G}^*(I))\to\mathcal{V}(\mathbb{G}^*(J))$ and  $d_\xi^*:\mathcal{V}(\mathbb{G}(\widehat{I}))^*\to\mathcal{V}(\mathbb{G}(\widehat{J}))^*$ commute with $\Psi\circ\Phi_*$ and $(\Psi\circ\Phi_*)^{-1}$.  Finally, we note that an edge assignment for the hypercube $\{0,1\}^n$ induces an edge assignment for the dual hypercube $\{\widehat{0,1}\}^n$, giving us the desired isomorphism of complexes.

Define a basis $\{v_-^*,v_+^*\}$ of $V^*$ by
$$v_-^*(v_-) = 0,\hspace{1cm}v_-^*(v_+)=1,\hspace{1cm}v_+^*(v_-)=1,\hspace{1cm}v_+^*(v_+)=0.$$
Fix an isomorphism $\psi: V^*\to V$ where $\psi(v_-^*)=v_-$ and $\psi(v_+^*)=v_+$, and define an isomorphism $\Psi:V(\mathbb{G}(I))^*\to V(\mathbb{G}(I))$ by $\Psi=\psi\otimes\cdots\otimes\psi$. The map $\Psi$ sends summands in the $(i,j)$-bigrading of $CKh(\mathbb{G})^*$ to the $(-i,-j)$-bigrading of $CKh(\mathbb{G})$. As noted in Section \ref{section:ribbongraphs}, there is a canonical bijection $\Phi$ from the boundary components of $\Sigma_{\mathbb{G}^*(I)}$ to the boundary components of $\Sigma_{\mathbb{G}(\widehat{I})}$ given by the gluing map in $\Sigma = \Sigma_{\mathbb{G}^*(I)}\cup \Sigma_{\mathbb{G}(\widehat{I})}$. The bijection induces an isomorphism $\Phi_*:V(\mathbb{G}^*(I))\to V(\mathbb{G}(\widehat{I}))$ that sends the summand in the $(i,j)$-bigrading of $CKh(\mathbb{G})$ to the $(n-i,n-j)$-bigrading of $CKh(\mathbb{G}^*)$. The composition $\Psi\circ\Phi_*:V(\mathbb{G}^*(I))\xrightarrow{\cong} V(\mathbb{G}(\widehat{I}))^*[n]\{n\}$ is the desired isomorphism.

If $m^*$ and $\Delta^*$ are the dual maps of $m$ and $\Delta$ respectively, then
$$\begin{array}{l c l}
m^*:V^*\to V^*\otimes V^* & \quad &
     m^*:\begin{cases}
          v_+^* \mapsto v_+^*\otimes v_-^* + v_-^*\otimes v_+^* &\\
      v_-^* \mapsto v_-^*\otimes v_-^* &
    \end{cases}\\
    \Delta^*: V^*\otimes V^*\to V^* & \quad &
   \Delta^*:\begin{cases}
     v_+^*\otimes v_-^*\mapsto v_-^* &
      v_+^*\otimes v_+^*\mapsto v_+^* \\
      v_-^*\otimes v_+^*\mapsto v_-^* &
      v_-^*\otimes v_-^*\mapsto 0
    \end{cases}.
  \end{array}$$
Let $d_\xi^*$ be the edge maps in the dual complex defined using $m^*$ and $\Delta^*$.
Since $m^*(v^*_{\pm}\otimes v^*_{\pm}) = (m(v_{\pm}\otimes v_{\pm}))^*$ and $\Delta^*(v^*_{\pm}) = (\Delta(v_{\pm}))^*$, it follows that $\Psi\circ\Phi_*\circ d_\xi = d_\xi^*\circ\Psi\circ\Phi_*$.

An edge assignment $\epsilon:\mathcal{E}(n)\to\{\pm 1\}$ gives an edge assignment $\widehat{\epsilon}:\widehat{\mathcal{E}}(n)\to\{\pm 1\}$ by $\widehat{\epsilon}(\widehat{\xi}) := \epsilon(\xi)$. Therefore, up to the prescribed grading shift, the complexes $(CKh(\mathbb{G}^*),d_\epsilon)$ and $(CKh(\mathbb{G})^*,d^*_{\widehat{\epsilon}})$ are isomorphic. Proposition \ref{prop:edge_assign} states that the choice of edge assignment does not change the isomorphism type of the complex, and the result follows.
\end{proof}

The following corollary follows from Proposition \ref{prop:duality} and the relationship between the homology of a complex and the homology of its dual.
\begin{corollary}
\label{cor:duality}
Let $\mathbb{G}$ be a ribbon graph  with $n$ edges, and let $\mathbb{G}^*$ be the dual ribbon graph. There are isomorphisms
\begin{eqnarray*}
Kh^{i,j}(\mathbb{G}^*) \otimes \mathbb{Q} & \cong & Kh^{n-i,n-j}(\mathbb{G})\otimes \mathbb{Q}\\
\text{Tor}(Kh^{i,j}(\mathbb{G}^*)) & \cong & \text{Tor}(Kh^{n-i+1,n-j+1}(\mathbb{G}))\\
\widetilde{Kh}^{i,j}(\mathbb{G}^*) \otimes \mathbb{Q} & \cong & \widetilde{Kh}^{n-i,n-j}(\mathbb{G})\otimes \mathbb{Q}\\
\text{Tor}(\widetilde{Kh}^{i,j}(\mathbb{G}^*)) & \cong & \text{Tor}(\widetilde{Kh}^{n-i+1,n-j+1}(\mathbb{G})).
\end{eqnarray*}
\end{corollary}


\section{The connection to Khovanov homology of links}
\label{section:khovanov}

Every link diagram $D$ has an associated ribbon graph $\mathbb{D}$, called the all-$A$ ribbon graph, whose construction is given in this section. The Khovanov homology of $\mathbb{D}$ is isomorphic to the Khovanov homology of the associated link (up to a grading shift). The goal of this section is to establish this isomorphism.

Loebl and Moffatt \cite{LoeblMoffatt:Fatgraphs} established a similar result for a different homology assigned to ribbon graphs with signed edges. Specifically, one can obtain Khovanov homology from the Loebl and Moffatt homology of the (necessarily planar) signed checkerboard graph of the link. By considering ribbon graphs with arbitrary genus, we are able to eliminate the dependency on signs.

\subsection{Link diagrams and ribbon graphs}
Let $D$ be a link diagram. Each crossing $c$ in $D$ has an $A$-resolution and a $B$-resolution (also known as a $0$-resolution and $1$-resolution, respectively) as in Figure \ref{figure:smoothing}. 
A diagram with all of its crossings resolved is called a {\it Kauffman state}. If $s$ is a Kauffman state, then define $|s|$ to be the number of components of $s$. Since a Kauffman state does not contain any information about the crossings of $D$, we find it useful to remember the crossings by adding certain colored line segments to each Kauffman state. When a crossing is resolved, we replace the crossing with a line segment called the {\it trace of the crossing} that connects the two strands. Traces corresponding to $A$-resolutions are colored blue, while traces corresponding to $B$-resolutions are colored red. 
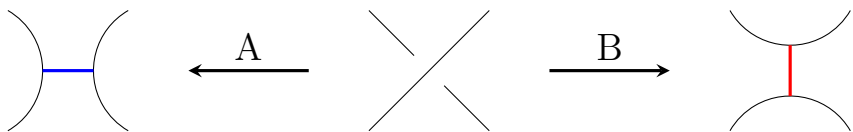
\begin{figure}[h]
$$\begin{tikzpicture}[>=stealth, scale=.8]
\draw (-1,-1) -- (1,1);
\draw (-1,1) -- (-.25,.25);
\draw (.25,-.25) -- (1,-1);
\draw (-3,0) node[above]{\Large{A}};
\draw[->,very thick] (-2,0) -- (-4,0);
\draw (3,0) node[above]{\Large{B}};
\draw[->,very thick] (2,0) -- (4,0);
\draw (-5,1) arc (120:240:1.1547cm);
\draw (-7,-1) arc (-60:60:1.1547cm);
\draw (5,1) arc (210:330:1.1547cm);
\draw (7,-1) arc (30:150:1.1547cm);
\draw[blue,very thick] (-5.57735,0) -- (-6.4226,0);
\draw[red,very thick] (6,0.422625) -- (6,-0.422625);
\end{tikzpicture}$$
\caption{The resolutions of a crossing and their traces in a link diagram.}
\label{figure:smoothing}
\end{figure}

The all-$A$ ribbon graph $\mathbb{D}$ of the link diagram $D$ is constructed as follows. First, choose the $A$-resolution for every crossing and obtain the all-$A$ Kauffman state. Next, orient the components of the all-$A$ Kauffman state in such a way that the outermost components are counterclockwise and any two nested components not separated by another component are oriented in opposite directions. The vertices of $\mathbb{D}$ are in one-to-one correspondence with the components of the all-$A$ Kauffman state, and the edges of $\mathbb{D}$ are in one-to-one correspondence with the traces of the crossings in the all-$A$ Kauffman state. An edge in $\mathbb{D}$ is incident to a vertex of $\mathbb{D}$ if and only if the corresponding trace is incident to the corresponding component in the all-$A$ Kauffman state. The orientation on each component of the all-$A$ Kauffman state gives a cyclic order on the endpoints of the traces incident to that component, which in turn induces the cyclic order of the half edges of $\mathbb{D}$ around each vertex. The construction of the all-$A$ ribbon graph of $D$ is illustrated in Figure \ref{fig:ribbonconstruct} where $D$ is a three-crossing diagram of the unknot.
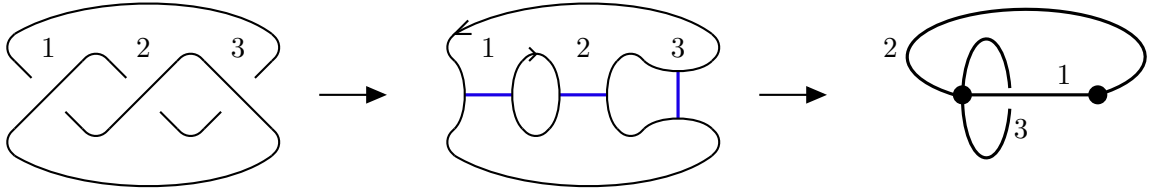
\begin{figure}[h]
$$\begin{tikzpicture}[scale=.9,thick,rounded corners = 2mm,>=triangle 45]
\draw (0,0) -- (-.2,.2) -- (0,.4);
\draw (0,1.4) -- (-.2,1.6) -- (0,1.8);
\draw (1,.4) -- (1.2,.2) -- (1.4,.4);
\draw(1,1.4) -- (1.2, 1.6) -- (1.4,1.4);
\begin{scope}[xshift = 1.4cm]
	\draw (1,.4) -- (1.2,.2) -- (1.4,.4);
	\draw(1,1.4) -- (1.2, 1.6) -- (1.4,1.4);
\end{scope}
\begin{scope}[xshift=3.8cm]
	\draw (0,0) -- (.2,.2) -- (0,.4);
	\draw (0,1.4) -- (.2,1.6) -- (0,1.8);
\end{scope}
\draw (0,1.8) .. controls (1,2.4) and (3,2.4) .. (3.8,1.8);
\draw (0,0) .. controls (1, -.6) and (3, -.6) .. (3.8,0);
\draw (0,.4) -- (1,1.4);
\draw (0,1.4) -- (.25, 1.15);
\draw (1,.4) -- (.75, .65);
\draw (.5,1.3) node[above]{$1$};
\begin{scope}[xshift = 1.4cm]
	\draw (0,.4) -- (1,1.4);
	\draw (0,1.4) -- (.25, 1.15);
	\draw (1,.4) -- (.75, .65);
	\draw (.5,1.3) node[above]{$2$};
\end{scope}
\begin{scope}[xshift = 2.8cm]
	\draw (0,1.4) -- (1,.4);
	\draw (0,.4) -- (.25, .65);
	\draw (1,1.4) -- (.75, 1.15);
	\draw (.5,1.3) node[above]{$3$};
\end{scope}
\draw[->] (4.5,.9) -- (5.5,.9);
\begin{scope}[xshift = 6.5cm]
	\draw (0,0) -- (-.2,.2) -- (0,.4);
	\draw (0,1.4) -- (-.2,1.6) -- (0,1.8);
	\draw (1,.4) -- (1.2,.2) -- (1.4,.4);
	\draw(1,1.4) -- (1.2, 1.6) -- (1.4,1.4);
	\begin{scope}[xshift = 1.4cm]
		\draw (1,.4) -- (1.2,.2) -- (1.4,.4);
		\draw(1,1.4) -- (1.2, 1.6) -- (1.4,1.4);
	\end{scope}
	\begin{scope}[xshift=3.8cm]
		\draw (0,0) -- (.2,.2) -- (0,.4);
		\draw (0,1.4) -- (.2,1.6) -- (0,1.8);
	\end{scope}
	\draw (0,1.8) .. controls (1,2.4) and (3,2.4) .. (3.8,1.8);
	\draw (0,0) .. controls (1, -.6) and (3, -.6) .. (3.8,0);
	\draw[rounded corners =0mm] (.2,2) -- (0,1.8)--(.25,1.8);
	\draw[rounded corners =0mm] (1.1,1.6) -- (1.2,1.5)--(1.1,1.4);
	\draw (0,.4) .. controls (.2,.6) and (.2,1.2) .. (0,1.4);
	\draw (1,.4) .. controls (.8,.6) and (.8,1.2) .. (1,1.4);
	\draw (1.4,.4) .. controls (1.6,.6) and (1.6,1.2) .. (1.4,1.4);
	\draw (2.4,1.4) .. controls (2.2,1.2) and (2.2,.6).. (2.4,.4);
	\draw (2.8,1.4) .. controls (3,1.2) and (3.6,1.2) .. (3.8,1.4);
	\draw (2.8,.4) .. controls (3,.6) and (3.6,.6) .. (3.8,.4);
	\draw (.5,1.3) node[above]{$1$};
	\draw[xshift = 1.4cm] (.5,1.3) node[above]{$2$};
	\draw[xshift = 2.8cm] (.5,1.3) node[above]{$3$};
	\begin{pgfonlayer}{background}
		\draw[very thick, blue!90!red] (.15,.9) -- (.85,.9);
		\draw[very thick, blue!90!red] (1.55,.9) -- (2.25,.9);
		\draw[very thick, blue!90!red] (3.3,.55) -- (3.3,1.25);
	\end{pgfonlayer}
\end{scope}
\draw[->](11,.9) -- (12,.9);
\begin{scope}[xshift =22cm, yshift=.9cm]
	\fill[black] (-6,0) circle (4pt);
	\fill[black] (-8,0) circle (4pt);
	\draw[very thick] (-8,0) -- (-6,0);
	\draw (-6.5,0) node[above]{$1$};
	\draw[very thick] (-7.3,0.1) arc (10:350:10pt and 25pt);
	\draw (-7.4,-.5) node[right]{$3$};
	\draw[very thick] (-6,0) arc (-53:233:50pt and 20pt);
	\draw (-8.8,.7) node[left]{$2$};
\end{scope}
\end{tikzpicture}$$
\caption{To construct the all-$A$ ribbon graph, first construct the all-$A$ Kauffman state. Then the components of the all-$A$ Kauffman state become the vertices, and the traces of the crossings become the edges.}
\label{fig:ribbonconstruct}
\end{figure}

An alternate method to construct the all-$A$ ribbon graph uses arrow presentations. This method is useful because it is easily generalized to the virtual link case (see Section \ref{section:virtual}). Once again, start with a link diagram $D$ and choose the $A$-resolution for every crossing in $A$. Instead of replacing each crossing with a trace, we replace the crossing by a pair of marking arrows with the same label. Orient the marking arrows by one of the two equivalent choices in Figure \ref{fig:markarrowresolution}.
\begin{figure}[h]
$$\begin{tikzpicture}[scale=.9]
\draw (-1,-1) -- (1,1);
\draw (-1,1) -- (-.25,.25);
\draw (.25,-.25) -- (1,-1);
\draw (0,-1) node[below]{\large{$x$}};
\draw[>=stealth,->,very thick] (-2,0) -- (-4,0);
\draw[>=stealth,->,very thick] (2,0) -- (4,0);

\begin{pgfonlayer}{background}
\draw (-5,1) arc (120:240:1.1547cm);
\draw (-7,-1) arc (-60:60:1.1547cm);
\begin{scope}[xshift = 12cm]
	\draw (-5,1) arc (120:240:1.1547cm);
	\draw (-7,-1) arc (-60:60:1.1547cm);
	
\end{scope}
\end{pgfonlayer}
\draw[very thick, red,->] (-6.42,0) arc (0:30:1.1547cm);
\draw[very thick, red] (-6.42,0) arc (0:-30:1.1547cm);
\draw[very thick, red] (-5.58,0) arc (180:150:1.1547cm);
\draw[very thick, red, ->] (-5.58, 0) arc (180:210:1.1547cm);
\draw (-6.7, 0) node{$x$};
\draw (-5.3,0) node{$x$};
\begin{scope}[xshift=12cm]
	\draw[very thick, red] (-6.42,0) arc (0:30:1.1547cm);
	\draw[very thick, red,->] (-6.42,0) arc (0:-30:1.1547cm);
	\draw[very thick, red,->] (-5.58,0) arc (180:150:1.1547cm);
	\draw[very thick, red] (-5.58, 0) arc (180:210:1.1547cm);
	\draw (-6.7, 0) node{$x$};
	\draw (-5.3,0) node{$x$};
\end{scope}
\end{tikzpicture}$$
\caption{The $A$-resolution of a crossing labeled by $x$ with the two equivalent choices of orientations of marking arrows.}
\label{fig:markarrowresolution}
\end{figure}
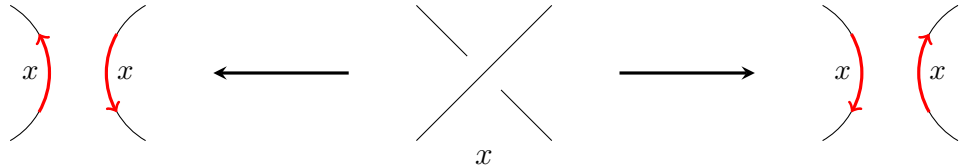
The resulting diagram is a collection of (possibly nested) circles in the plane decorated with labeled marking arrows. In order to obtain an arrow presentation, translate any circle that is nested inside another circle until it is no longer nested. An example of this process is shown in Figure \ref{fig:arrowconstruct}.

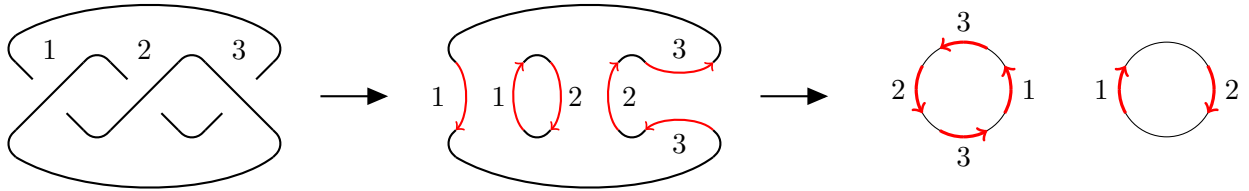
\begin{figure}[h]
$$\begin{tikzpicture}[scale=.9,thick,rounded corners = 2mm,>=triangle 45]
\draw (0,0) -- (-.2,.2) -- (0,.4);
\draw (0,1.4) -- (-.2,1.6) -- (0,1.8);
\draw (1,.4) -- (1.2,.2) -- (1.4,.4);
\draw(1,1.4) -- (1.2, 1.6) -- (1.4,1.4);
\begin{scope}[xshift = 1.4cm]
	\draw (1,.4) -- (1.2,.2) -- (1.4,.4);
	\draw(1,1.4) -- (1.2, 1.6) -- (1.4,1.4);
\end{scope}
\begin{scope}[xshift=3.8cm]
	\draw (0,0) -- (.2,.2) -- (0,.4);
	\draw (0,1.4) -- (.2,1.6) -- (0,1.8);
\end{scope}
\draw (0,1.8) .. controls (1,2.4) and (3,2.4) .. (3.8,1.8);
\draw (0,0) .. controls (1, -.6) and (3, -.6) .. (3.8,0);
\draw (0,.4) -- (1,1.4);
\draw (0,1.4) -- (.25, 1.15);
\draw (1,.4) -- (.75, .65);
\draw (.5,1.3) node[above]{$1$};
\begin{scope}[xshift = 1.4cm]
	\draw (0,.4) -- (1,1.4);
	\draw (0,1.4) -- (.25, 1.15);
	\draw (1,.4) -- (.75, .65);
	\draw (.5,1.3) node[above]{$2$};
\end{scope}
\begin{scope}[xshift = 2.8cm]
	\draw (0,1.4) -- (1,.4);
	\draw (0,.4) -- (.25, .65);
	\draw (1,1.4) -- (.75, 1.15);
	\draw (.5,1.3) node[above]{$3$};
\end{scope}
\draw[->] (4.5,.9) -- (5.5,.9);
\begin{scope}[xshift = 6.5cm]
	\draw (0,0) -- (-.2,.2) -- (0,.4);
	\draw (0,1.4) -- (-.2,1.6) -- (0,1.8);
	\draw (1,.4) -- (1.2,.2) -- (1.4,.4);
	\draw(1,1.4) -- (1.2, 1.6) -- (1.4,1.4);
	\begin{scope}[xshift = 1.4cm]
		\draw (1,.4) -- (1.2,.2) -- (1.4,.4);
		\draw(1,1.4) -- (1.2, 1.6) -- (1.4,1.4);
	\end{scope}
	\begin{scope}[xshift=3.8cm]
		\draw (0,0) -- (.2,.2) -- (0,.4);
		\draw (0,1.4) -- (.2,1.6) -- (0,1.8);
	\end{scope}
	\draw (0,1.8) .. controls (1,2.4) and (3,2.4) .. (3.8,1.8);
	\draw (0,0) .. controls (1, -.6) and (3, -.6) .. (3.8,0);
\begin{scope}[>=to]
	\draw[thick, red,<-] (0,.4) .. controls (.2,.6) and (.2,1.2) .. (0,1.4);
	\draw[thick,red,->] (1,.4) .. controls (.8,.6) and (.8,1.2) .. (1,1.4);
	\draw[thick, red,<-] (1.4,.4) .. controls (1.6,.6) and (1.6,1.2) .. (1.4,1.4);
	\draw[thick,red,<-] (2.4,1.4) .. controls (2.2,1.2) and (2.2,.6).. (2.4,.4);
	\draw[thick,red,->] (2.8,1.4) .. controls (3,1.2) and (3.6,1.2) .. (3.8,1.4);
	\draw[thick,red,<-] (2.8,.4) .. controls (3,.6) and (3.6,.6) .. (3.8,.4);
	\draw (0,.9) node[left]{$1$};
	\draw (.9,.9) node[left]{$1$};
	\draw (1.5,.9) node[right]{$2$};
	\draw (2.3,.9) node[right]{$2$};
	\draw (3.3,1.3) node[above]{$3$};
	\draw (3.3, .5) node[below]{$3$};
\end{scope}
\end{scope}
\draw[->](11,.9) -- (12,.9);
\begin{scope}[xshift =14cm,yshift = 1cm,>=to]
	\begin{pgfonlayer}{background}
\draw (0,0) circle (.7 cm);
\draw (3,0) circle (.7 cm);
\end{pgfonlayer}
\draw[very thick, red, ->] (0.7,0) arc (0:30:.7cm);
\draw[very thick, red] (0.7,0) arc (0:-30:.7cm);
\draw[very thick, red, ->] (0,0.7) arc (90:120:.7cm);
\draw[very thick, red] (0,0.7) arc (90:60:.7cm);
\draw[very thick, red] (-0.7,0) arc (180:150:.7cm);
\draw[very thick, red,->] (-0.7,0) arc (180:210:.7cm);
\draw[very thick, red] (0,-0.7) arc (270:240:.7cm);
\draw[very thick, red,->] (0,-0.7) arc (270:300:.7cm);

\draw[very thick, red,->] (2.3,0) arc (180:150:.7cm);
\draw[very thick, red] (2.3,0) arc (180:210:.7cm);
\draw[very thick, red,->] (3.7,0) arc (0:-30:.7cm);
\draw[very thick, red] (3.7,0) arc (0:30:.7cm);

\draw (0.7,0) node[right]{$1$};
\draw (0,0.7) node[above]{$3$};
\draw (-0.7,0) node[left]{$2$};
\draw (0,-0.7) node[below]{$3$};
\draw (2.3,0) node[left]{$1$};
\draw (3.7,0) node[right]{$2$};\end{scope}
\end{tikzpicture}$$
\caption{The construction of the arrow presentation for the all-$A$ ribbon graph of a link diagram $D$.}
\label{fig:arrowconstruct}
\end{figure}

The all-$A$ ribbon graph $\mathbb{D}$ is embedded on the Turaev surface. One easy way to construct the Turaev surface is to start with a cobordism from the all-$A$ Kauffman state of $D$ to the all-$B$ Kauffman state of $D$ that has saddle points corresponding to the crossings, and then cap off its boundary components with disks. For an in-depth look at the Turaev surface see \cite{DFKLS:GraphsOnSurfaces}.

\subsection{Relationship to the Khovanov complex}

Suppose that the link diagram $D$ has crossings $x_1,\dots,x_n$ and that $I=(m_1,\dots,m_n)$ is a vertex in the hypercube $\{0,1\}^n$. Define $D(I)$ to be the Kauffman state with an $A$-resolution at crossing $i$ if $m_i=0$ and a $B$-resolution at crossing $i$ if $m_i=1$ for $1\leq i \leq n$.
Let $S(D)$ be the set of all Kauffman states of $D$. Since the edges of $\mathbb{D}$ correspond to the crossings of $D$, both $S(D)$ and the set $S(\mathbb{D})$ of spanning ribbon subgraphs of $\mathbb{D}$ have $2^{n}$ elements. 
\begin{lemma}
\label{lemma:components}
Let $D$ be a link diagram with $n$ crossings and $\mathbb{D}$ be its all-$A$ ribbon graph. If $I$ is a vertex in the hypercube $\{0,1\}^n$, then the number of boundary components of $\Sigma_{\mathbb{D}(I)}$ is equal to the number of components in the Kauffman state $D(I)$.
\end{lemma}
\begin{proof}
We prove this lemma by induction on the height $h(I)$ of the vertex $I$. Recall that if $I=(m_1,\dots,m_n)$, then $h(I)=\sum_{i=1}^n m_i$. 
If $I=(0,\dots,0)$, then $\mathbb{D}(I)$ is the spanning ribbon subgraph with no edges (and thus consisting of only isolated vertices) and $D(I)$ is the all-$A$ Kauffman state. The vertices of $\mathbb{D}$, and thus the vertices of $\mathbb{D}(I)$, are in one-to-one correspondence with the components of the all-$A$ Kauffman state. The surface $\Sigma_{\mathbb{D}(I)}$ has one boundary component for each vertex of $\mathbb{D}(I)$ and therefore one boundary component for each component of $D(I)$.

Let $I=(m_1,\dots,m_n)$ be a vertex such that $m_k=1$ for some $k$ with $1\leq k \leq n$, and let $J=(m_1^\prime,\dots,m_n^\prime)$ where $m_k^\prime=0$ and $m_i^\prime=m_i$ if $i\neq k$. Then there is an edge in the hypercube from $J$ to $I$. By induction, we may assume that the number of boundary components of $\Sigma_{\mathbb{D}(J)}$ is equal to the boundary components of $D(J)$. The surface $\Sigma_{\mathbb{D}(I)}$ can be obtained from the surface $\Sigma_{\mathbb{D}(J)}$ by attaching a $2$-dimensional one-handle $h$ along an $S^0$ in the boundary of $\Sigma_{\mathbb{D}(J)}$. The $2$-dimensional one-handle $h$ corresponds to a trace $t$ in the Kauffman state $D(J)$, and since $m_k^\prime=0$, the trace $t$ corresponds to an $A$-resolution of $D$.

If the endpoints of $t$ lie on the same component of $D(J)$, then the two attaching points of $S^0$ where $h$ is attached lie on different boundary components of $\Sigma_{\mathbb{D}(J)}$. Attaching the one-handle $h$ splits one boundary component of $\Sigma_{\mathbb{D}(J)}$ into two and leaves the other boundary components unchanged. Likewise, since both endpoints of $t$ lie on the same component, changing that crossing from an $A$-resolution to a $B$-resolution corresponds to splitting on component of $D(J)$ into two, while leaving the other components unchanged. Therefore, in this case, the number of boundary component of $\Sigma_{\mathbb{D}(I)}$ equals the number of components of $D(I)$.

If the endpoints of $t$ lie on different components of $D(J)$, then the attaching points of $S^0$ for the one-handle $h$ lie on different boundary components of $\Sigma_{\mathbb{D}(J)}$. Adding the one-handle $h$ merges two boundary components of $\Sigma_{\mathbb{D}(J)}$ into one and leaves the other boundary components unchanged. Likewise, since the endpoints of $t$ lie on different components, changing that crossing from an $A$-resolution to a $B$-resolution corresponds to merging two components of $D(J)$ into one, while leaving the other components unchanged. Therefore, the number of boundary component of $\Sigma_{\mathbb{D}(I)}$ equals the number of components of $D(I)$.
\end{proof}

We now review the construction of Khovanov homology. Let $D$ be a diagram of the link $L$ with $n$ crossings. Assign the Kauffman state $D(I)$ to the vertex $I$ in the hypercube $\{0,1\}^n$. There is a directed edge $\xi$ of the hypercube from $D(I)$ to $D(J)$ if and only if $D(J)$ can be obtained from $D(I)$ by changing one $A$-resolution to a $B$-resolution. Let $|D(I)|$ denote the number of components in the Kauffman state $D(I)$. 

Associate to each vertex $I\in\mathcal{V}(n)$ the $\mathbb{Z}$-module $V(D(I)) := V^{\otimes |D(I)|}[h(I)]\{h(I)\}$ where $h(I)$ denotes the height of the vertex $I$. Let $n_+$ and $n_-$ be the number of positive and negative crossings in $D$ respectively. Define $CKh(D) = \bigoplus_{I\in \mathcal{V}(n)} V(D(I)) [-n_-]\{n_+ - 2n_-\}$. The module $CKh(D)$ is bigraded, and we denote the summand with bigrading $(i,j)$ by $CKh^{i,j}(D)$. As before, it will useful to refer to all summands in a specific homological grading but arbitrary polynomial grading; therefore, we write $CKh^{i,*}(D) = \bigoplus_j CKh^{i,j}(D)$.

Each edge $\xi$ in the hypercube has an associated edge map $d_\xi:V(D(I))\to V(D(J))$ where $\xi$ is an edge from $I$ to $J$. Changing a resolution in $D(I)$ from an $A$-resolution to a $B$-resolution either merges two components of $D(J)$ into one or splits one component of $D(I)$ into two. Define $d_{\xi}$ to be the identity on tensor factors of $V(D(I))$ that correspond to components of $D(I)$ which are not changed when the resolution is changed. If $\xi$ merges two components of $D(I)$, then the edge map $d_{\xi}$ is defined to be $m:V\otimes V\to V$ on the two tensor factors corresponding to the merging components, and if $\xi$ splits one component of $D(I)$ into two, then $d_\xi$ is defined to be $\Delta:V\to V\otimes V$ on the tensor factor of $V(D(I))$ corresponding to the component being split.

Suppose that $\xi$ is an edge from the vertex $I=(m_1,\dots,m_n)$ to the vertex $J=(m_1^\prime,\dots,m_n^\prime)$, and let $|\xi|=h(I)$. The coordinates of $I$ and $J$ are the same except at one coordinate, say the $k$th coordinate.
Define $(-1)^\xi=(-1)^{\sum_{i=1}^{k-1}m_i}$. Define the differential $d^i:CKh^{i,*}(D) \to CKh^{i+1,*}(D)$ by $d^i=\sum_{|\xi|=i-n_-} (-1)^\xi d_\xi$. Since the differential preserves the polynomial grading, we can write $d^i = \sum_j d^{i,j}$ where $d^{i,j}:CKh^{i,j}(D)\to CKh^{i+1,j}(D)$. 

The Khovanov homology of $L$, denoted $Kh(L)$, is the homology of the complex $(CKh(D), d)$. Specifically, the summand of $Kh(L)$ in the $(i,j)$ bigrading is defined as $Kh^{i,j}(L) = \text{ker}~d^{i,j} /\text{im}~d^{i-1,j}$. Khovanov \cite{Khovanov:homology} proves that $Kh(L)$ is a link invariant. In a manner similar to the construction described in Section \ref{subsec:reduced}, one can also construct reduced Khovanov homology $\widetilde{Kh}(L)$. 

\begin{theorem}
\label{theorem:ribbontokh}
Let $L$ be a link with diagram $D$ and all-$A$ ribbon graph $\mathbb{D}$. Then there are  bigraded isomorphisms of complexes
\begin{eqnarray*}
C(\mathbb{D})[-n_-]\{n_+ - 2n_-\} & \cong & CKh(D)~\text{and}\\
\widetilde{C}(\mathbb{D})[-n_-]\{n_+-2n_- \}  & \cong &\widetilde{CKh}(D).
\end{eqnarray*}
\end{theorem}
\begin{proof}
Lemma \ref{lemma:components} implies that there is a bijection from the spanning ribbon subgraphs of $\mathbb{D}$ and the Kauffman states of $D$ sending $\mathbb{D}(I)\mapsto D(I)$ such that the number of boundary components of $\Sigma_{\mathbb{D}(I)}$ equals the number of components in $D(I)$. Therefore, for each $I\in\mathcal{V}(n)$, there is a bigraded isomorphism $\rho_I:V(\mathbb{D}(I))\to V(D(I))$.

Hence, the map $\rho=\sum_{I\in\mathcal{V}(n)}\rho_I$ is grading preserving isomorphism $$\rho:CKh(\mathbb{D})[-n_-]\{n_+-2n_-\}\to CKh(D).$$ If there is an edge in $\mathcal{E}(n)$ from $I$ to $J$ and two components of $D(I)$ merge to form $D(J)$, then the corresponding boundary components of $\Sigma_{\mathbb{D}(I)}$ merge to give the boundary components of $\Sigma_{\mathbb{D}(J)}$. Similarly, if a boundary component of $D(I)$ splits into two boundary components of $D(J)$, then the corresponding boundary component of $\Sigma_{\mathbb{D}(I)}$ splits into two boundary components of $\Sigma_{\mathbb{D}(J)}$. Therefore $\rho$ commutes with the differentials of $CKh(\mathbb{D})$ and $CKh(D)$, and hence the result follows. The proof of the statement for reduced homology is similar.
\end{proof}

Theorem \ref{theorem:maintheorem1} is an obvious corollary of Theorem \ref{theorem:ribbontokh}.

As an example, in Figure \ref{fig:kauffman_cube} we depict the cube of resolutions complex of the three-crossing unknot from Figure \ref{fig:ribbonconstruct}. One can explicitly see the correspondence of Lemma \ref{lemma:components} for this example.
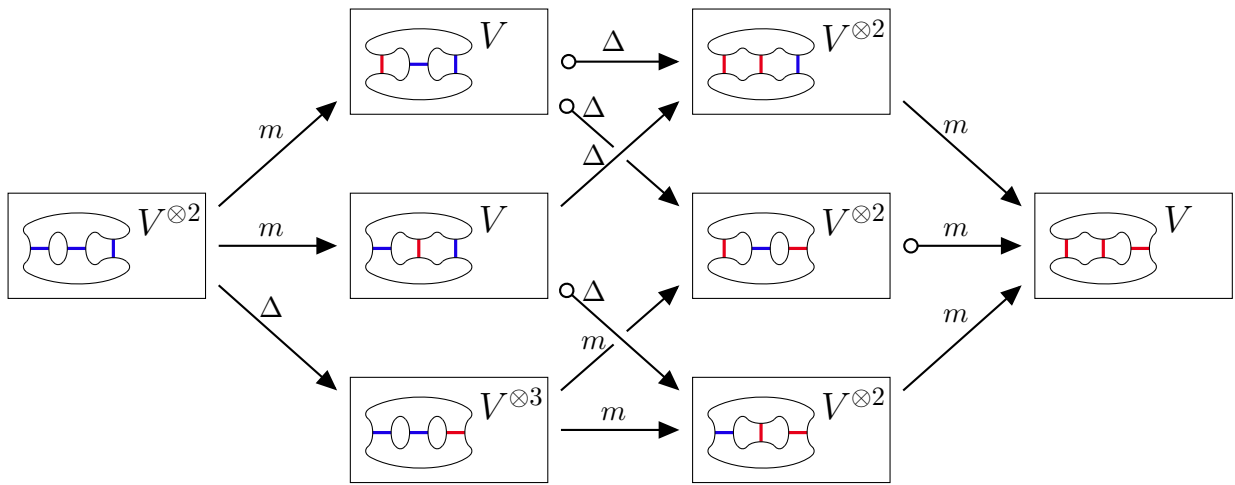
\begin{figure}[h]
$$\begin{tikzpicture}[scale=.35]
\draw (-1.5,-1.5) rectangle (6,2.5);
\begin{scope}[xshift=-.8cm,yshift=-.5cm,rounded corners=.7mm]
\draw (0,0) -- (-.2,.2) -- (0,.4);
\draw (0,1.4) -- (-.2,1.6) -- (0,1.8);
\draw (1,.4) -- (1.2,.2) -- (1.4,.4);
\draw(1,1.4) -- (1.2, 1.6) -- (1.4,1.4);
\begin{scope}[xshift = 1.4cm]
	\draw (1,.4) -- (1.2,.2) -- (1.4,.4);
	\draw(1,1.4) -- (1.2, 1.6) -- (1.4,1.4);
\end{scope}
\begin{scope}[xshift=3.8cm]
	\draw (0,0) -- (.2,.2) -- (0,.4);
	\draw (0,1.4) -- (.2,1.6) -- (0,1.8);
\end{scope}
\draw (0,1.8) .. controls (1,2.4) and (3,2.4) .. (3.8,1.8);
\draw (0,0) .. controls (1, -.6) and (3, -.6) .. (3.8,0);
\draw (0,.4) .. controls (.2,.6) and (.2,1.2) .. (0,1.4);
\draw (1,.4) .. controls (.8,.6) and (.8,1.2) .. (1,1.4);
\begin{pgfonlayer}{background}
	\draw[very thick, blue!90!red] (.15,.9) -- (.85,.9);
\end{pgfonlayer}
\begin{scope}[xshift = 1.4cm]
\draw (0,.4) .. controls (.2,.6) and (.2,1.2) .. (0,1.4);
\draw (1,.4) .. controls (.8,.6) and (.8,1.2) .. (1,1.4);
\begin{pgfonlayer}{background}
	\draw[very thick, blue!90!red] (.15,.9) -- (.85,.9);
\end{pgfonlayer}
\end{scope}
\begin{scope}[xshift=2.8cm]
\draw (0,.4) .. controls (.2,.6) and (.8,.6) .. (1,.4);
\draw (0,1.4) .. controls (.2,1.2) and (.8,1.2) .. (1,1.4);
\begin{pgfonlayer}{background}
	\draw[very thick, blue!90!red] (.5,.55) -- (.5,1.25);
\end{pgfonlayer}
\end{scope}
\end{scope}
\draw (3,1.5) node[right]{\Large{$V^{\otimes 2}$}};
\begin{scope}[xshift = 13cm]
\draw (-1.5,-1.5) rectangle (6,2.5);
\begin{scope}[xshift=-.8cm,yshift=-.5cm,rounded corners=.7mm]
\draw (0,0) -- (-.2,.2) -- (0,.4);
\draw (0,1.4) -- (-.2,1.6) -- (0,1.8);
\draw (1,.4) -- (1.2,.2) -- (1.4,.4);
\draw(1,1.4) -- (1.2, 1.6) -- (1.4,1.4);
\begin{scope}[xshift = 1.4cm]
	\draw (1,.4) -- (1.2,.2) -- (1.4,.4);
	\draw(1,1.4) -- (1.2, 1.6) -- (1.4,1.4);
\end{scope}
\begin{scope}[xshift=3.8cm]
	\draw (0,0) -- (.2,.2) -- (0,.4);
	\draw (0,1.4) -- (.2,1.6) -- (0,1.8);
\end{scope}
\draw (0,1.8) .. controls (1,2.4) and (3,2.4) .. (3.8,1.8);
\draw (0,0) .. controls (1, -.6) and (3, -.6) .. (3.8,0);
\draw (0,.4) .. controls (.2,.6) and (.2,1.2) .. (0,1.4);
\draw (1,.4) .. controls (.8,.6) and (.8,1.2) .. (1,1.4);
\begin{pgfonlayer}{background}
	\draw[very thick, blue!90!red] (.15,.9) -- (.85,.9);
\end{pgfonlayer}
\begin{scope}[xshift = 1.4cm]
\draw (0,.4) .. controls (.2,.6) and (.8,.6) .. (1,.4);
\draw (0,1.4) .. controls (.2,1.2) and (.8,1.2) .. (1,1.4);
\begin{pgfonlayer}{background}
	\draw[very thick, red!90!blue] (.5,.55) -- (.5,1.25);
\end{pgfonlayer}\end{scope}
\begin{scope}[xshift=2.8cm]
\draw (0,.4) .. controls (.2,.6) and (.8,.6) .. (1,.4);
\draw (0,1.4) .. controls (.2,1.2) and (.8,1.2) .. (1,1.4);
\begin{pgfonlayer}{background}
	\draw[very thick, blue!90!red] (.5,.55) -- (.5,1.25);
\end{pgfonlayer}
\end{scope}
\end{scope}

\draw (3,1.5) node[right]{\Large{$V$}};
\end{scope}
\begin{scope}[xshift=13cm, yshift=7cm]
\draw (-1.5,-1.5) rectangle (6,2.5);
\begin{scope}[xshift=-.8cm,yshift=-.5cm,rounded corners=.7mm]
\draw (0,0) -- (-.2,.2) -- (0,.4);
\draw (0,1.4) -- (-.2,1.6) -- (0,1.8);
\draw (1,.4) -- (1.2,.2) -- (1.4,.4);
\draw(1,1.4) -- (1.2, 1.6) -- (1.4,1.4);
\begin{scope}[xshift = 1.4cm]
	\draw (1,.4) -- (1.2,.2) -- (1.4,.4);
	\draw(1,1.4) -- (1.2, 1.6) -- (1.4,1.4);
\end{scope}
\begin{scope}[xshift=3.8cm]
	\draw (0,0) -- (.2,.2) -- (0,.4);
	\draw (0,1.4) -- (.2,1.6) -- (0,1.8);
\end{scope}
\draw (0,1.8) .. controls (1,2.4) and (3,2.4) .. (3.8,1.8);
\draw (0,0) .. controls (1, -.6) and (3, -.6) .. (3.8,0);
\draw (0,.4) .. controls (.2,.6) and (.8,.6) .. (1,.4);
\draw (0,1.4) .. controls (.2,1.2) and (.8,1.2) .. (1,1.4);
\begin{pgfonlayer}{background}
	\draw[very thick, red!90!blue] (.5,.55) -- (.5,1.25);
\end{pgfonlayer}
\begin{scope}[xshift = 1.4cm]
\draw (0,.4) .. controls (.2,.6) and (.2,1.2) .. (0,1.4);
\draw (1,.4) .. controls (.8,.6) and (.8,1.2) .. (1,1.4);
\begin{pgfonlayer}{background}
	\draw[very thick, blue!90!red] (.15,.9) -- (.85,.9);
\end{pgfonlayer}
\end{scope}
\begin{scope}[xshift=2.8cm]
\draw (0,.4) .. controls (.2,.6) and (.8,.6) .. (1,.4);
\draw (0,1.4) .. controls (.2,1.2) and (.8,1.2) .. (1,1.4);
\begin{pgfonlayer}{background}
	\draw[very thick, blue!90!red] (.5,.55) -- (.5,1.25);
\end{pgfonlayer}
\end{scope}
\end{scope}
\draw (3,1.5) node[right]{\Large{$V$}};
\end{scope}
\begin{scope}[xshift=13cm, yshift=-7cm]
\draw (-1.5,-1.5) rectangle (6,2.5);
\begin{scope}[xshift=-.8cm,yshift=-.5cm,rounded corners=.7mm]
\draw (0,0) -- (-.2,.2) -- (0,.4);
\draw (0,1.4) -- (-.2,1.6) -- (0,1.8);
\draw (1,.4) -- (1.2,.2) -- (1.4,.4);
\draw(1,1.4) -- (1.2, 1.6) -- (1.4,1.4);
\begin{scope}[xshift = 1.4cm]
	\draw (1,.4) -- (1.2,.2) -- (1.4,.4);
	\draw(1,1.4) -- (1.2, 1.6) -- (1.4,1.4);
\end{scope}
\begin{scope}[xshift=3.8cm]
	\draw (0,0) -- (.2,.2) -- (0,.4);
	\draw (0,1.4) -- (.2,1.6) -- (0,1.8);
\end{scope}
\draw (0,1.8) .. controls (1,2.4) and (3,2.4) .. (3.8,1.8);
\draw (0,0) .. controls (1, -.6) and (3, -.6) .. (3.8,0);
\draw (0,.4) .. controls (.2,.6) and (.2,1.2) .. (0,1.4);
\draw (1,.4) .. controls (.8,.6) and (.8,1.2) .. (1,1.4);
\begin{pgfonlayer}{background}
	\draw[very thick, blue!90!red] (.15,.9) -- (.85,.9);
\end{pgfonlayer}
\begin{scope}[xshift = 1.4cm]
\draw (0,.4) .. controls (.2,.6) and (.2,1.2) .. (0,1.4);
\draw (1,.4) .. controls (.8,.6) and (.8,1.2) .. (1,1.4);
\begin{pgfonlayer}{background}
	\draw[very thick, blue!90!red] (.15,.9) -- (.85,.9);
\end{pgfonlayer}
\end{scope}
\begin{scope}[xshift=2.8cm]
\draw (0,.4) .. controls (.2,.6) and (.2,1.2) .. (0,1.4);
\draw (1,.4) .. controls (.8,.6) and (.8,1.2) .. (1,1.4);
\begin{pgfonlayer}{background}
	\draw[very thick, red!90!blue] (.15,.9) -- (.85,.9);
\end{pgfonlayer}
\end{scope}
\end{scope}
\draw (3,1.5) node[right]{\Large{$V^{\otimes{3}}$}};
\end{scope}
\begin{scope}[xshift=26cm]
\draw (-1.5,-1.5) rectangle (6,2.5);
\begin{scope}[xshift=-.8cm,yshift=-.5cm,rounded corners=.7mm]
\draw (0,0) -- (-.2,.2) -- (0,.4);
\draw (0,1.4) -- (-.2,1.6) -- (0,1.8);
\draw (1,.4) -- (1.2,.2) -- (1.4,.4);
\draw(1,1.4) -- (1.2, 1.6) -- (1.4,1.4);
\begin{scope}[xshift = 1.4cm]
	\draw (1,.4) -- (1.2,.2) -- (1.4,.4);
	\draw(1,1.4) -- (1.2, 1.6) -- (1.4,1.4);
\end{scope}
\begin{scope}[xshift=3.8cm]
	\draw (0,0) -- (.2,.2) -- (0,.4);
	\draw (0,1.4) -- (.2,1.6) -- (0,1.8);
\end{scope}
\draw (0,1.8) .. controls (1,2.4) and (3,2.4) .. (3.8,1.8);
\draw (0,0) .. controls (1, -.6) and (3, -.6) .. (3.8,0);
\draw (0,.4) .. controls (.2,.6) and (.8,.6) .. (1,.4);
\draw (0,1.4) .. controls (.2,1.2) and (.8,1.2) .. (1,1.4);
\begin{pgfonlayer}{background}
	\draw[very thick, red!90!blue] (.5,.55) -- (.5,1.25);
\end{pgfonlayer}
\begin{scope}[xshift = 1.4cm]
\draw (0,.4) .. controls (.2,.6) and (.2,1.2) .. (0,1.4);
\draw (1,.4) .. controls (.8,.6) and (.8,1.2) .. (1,1.4);
\begin{pgfonlayer}{background}
	\draw[very thick, blue!90!red] (.15,.9) -- (.85,.9);
\end{pgfonlayer}
\end{scope}
\begin{scope}[xshift=2.8cm]
\draw (0,.4) .. controls (.2,.6) and (.2,1.2) .. (0,1.4);
\draw (1,.4) .. controls (.8,.6) and (.8,1.2) .. (1,1.4);
\begin{pgfonlayer}{background}
	\draw[very thick, red!90!blue] (.15,.9) -- (.85,.9);
\end{pgfonlayer}
\end{scope}
\end{scope}
\draw (3,1.5) node[right]{\Large{$V^{\otimes{2}}$}};
\end{scope}
\begin{scope}[xshift=26cm, yshift=7cm]
\draw (-1.5,-1.5) rectangle (6,2.5);
\begin{scope}[xshift=-.8cm,yshift=-.5cm,rounded corners=.7mm]
\draw (0,0) -- (-.2,.2) -- (0,.4);
\draw (0,1.4) -- (-.2,1.6) -- (0,1.8);
\draw (1,.4) -- (1.2,.2) -- (1.4,.4);
\draw(1,1.4) -- (1.2, 1.6) -- (1.4,1.4);
\begin{scope}[xshift = 1.4cm]
	\draw (1,.4) -- (1.2,.2) -- (1.4,.4);
	\draw(1,1.4) -- (1.2, 1.6) -- (1.4,1.4);
\end{scope}
\begin{scope}[xshift=3.8cm]
	\draw (0,0) -- (.2,.2) -- (0,.4);
	\draw (0,1.4) -- (.2,1.6) -- (0,1.8);
\end{scope}
\draw (0,1.8) .. controls (1,2.4) and (3,2.4) .. (3.8,1.8);
\draw (0,0) .. controls (1, -.6) and (3, -.6) .. (3.8,0);
\draw (0,.4) .. controls (.2,.6) and (.8,.6) .. (1,.4);
\draw (0,1.4) .. controls (.2,1.2) and (.8,1.2) .. (1,1.4);
\begin{pgfonlayer}{background}
	\draw[very thick, red!90!blue] (.5,.55) -- (.5,1.25);
\end{pgfonlayer}
\begin{scope}[xshift = 1.4cm]
\draw (0,.4) .. controls (.2,.6) and (.8,.6) .. (1,.4);
\draw (0,1.4) .. controls (.2,1.2) and (.8,1.2) .. (1,1.4);
\begin{pgfonlayer}{background}
	\draw[very thick, red!90!blue] (.5,.55) -- (.5,1.25);
\end{pgfonlayer}
\end{scope}
\begin{scope}[xshift=2.8cm]
\draw (0,.4) .. controls (.2,.6) and (.8,.6) .. (1,.4);
\draw (0,1.4) .. controls (.2,1.2) and (.8,1.2) .. (1,1.4);
\begin{pgfonlayer}{background}
	\draw[very thick, blue!90!red] (.5,.55) -- (.5,1.25);
\end{pgfonlayer}
\end{scope}
\end{scope}
\draw (3,1.5) node[right]{\Large{$V^{\otimes{2}}$}};
\end{scope}
\begin{scope}[xshift=26cm, yshift=-7cm]
\draw (-1.5,-1.5) rectangle (6,2.5);
\begin{scope}[xshift=-.8cm,yshift=-.5cm,rounded corners=.7mm]
\draw (0,0) -- (-.2,.2) -- (0,.4);
\draw (0,1.4) -- (-.2,1.6) -- (0,1.8);
\draw (1,.4) -- (1.2,.2) -- (1.4,.4);
\draw(1,1.4) -- (1.2, 1.6) -- (1.4,1.4);
\begin{scope}[xshift = 1.4cm]
	\draw (1,.4) -- (1.2,.2) -- (1.4,.4);
	\draw(1,1.4) -- (1.2, 1.6) -- (1.4,1.4);
\end{scope}
\begin{scope}[xshift=3.8cm]
	\draw (0,0) -- (.2,.2) -- (0,.4);
	\draw (0,1.4) -- (.2,1.6) -- (0,1.8);
\end{scope}
\draw (0,1.8) .. controls (1,2.4) and (3,2.4) .. (3.8,1.8);
\draw (0,0) .. controls (1, -.6) and (3, -.6) .. (3.8,0);
\draw (0,.4) .. controls (.2,.6) and (.2,1.2) .. (0,1.4);
\draw (1,.4) .. controls (.8,.6) and (.8,1.2) .. (1,1.4);
\begin{pgfonlayer}{background}
	\draw[very thick, blue!90!red] (.15,.9) -- (.85,.9);
\end{pgfonlayer}
\begin{scope}[xshift = 1.4cm]
\draw (0,.4) .. controls (.2,.6) and (.8,.6) .. (1,.4);
\draw (0,1.4) .. controls (.2,1.2) and (.8,1.2) .. (1,1.4);
\begin{pgfonlayer}{background}
	\draw[very thick, red!90!blue] (.5,.55) -- (.5,1.25);
\end{pgfonlayer}
\end{scope}
\begin{scope}[xshift=2.8cm]
\draw (0,.4) .. controls (.2,.6) and (.2,1.2) .. (0,1.4);
\draw (1,.4) .. controls (.8,.6) and (.8,1.2) .. (1,1.4);
\begin{pgfonlayer}{background}
	\draw[very thick, red!90!blue] (.15,.9) -- (.85,.9);
\end{pgfonlayer}
\end{scope}
\end{scope}
\draw (3,1.5) node[right]{\Large{$V^{\otimes{2}}$}};
\end{scope}
\begin{scope}[xshift=39cm]
\draw (-1.5,-1.5) rectangle (6,2.5);
\begin{scope}[xshift=-.8cm,yshift=-.5cm,rounded corners=.7mm]
\draw (0,0) -- (-.2,.2) -- (0,.4);
\draw (0,1.4) -- (-.2,1.6) -- (0,1.8);
\draw (1,.4) -- (1.2,.2) -- (1.4,.4);
\draw(1,1.4) -- (1.2, 1.6) -- (1.4,1.4);
\begin{scope}[xshift = 1.4cm]
	\draw (1,.4) -- (1.2,.2) -- (1.4,.4);
	\draw(1,1.4) -- (1.2, 1.6) -- (1.4,1.4);
\end{scope}
\begin{scope}[xshift=3.8cm]
	\draw (0,0) -- (.2,.2) -- (0,.4);
	\draw (0,1.4) -- (.2,1.6) -- (0,1.8);
\end{scope}
\draw (0,1.8) .. controls (1,2.4) and (3,2.4) .. (3.8,1.8);
\draw (0,0) .. controls (1, -.6) and (3, -.6) .. (3.8,0);
\draw (0,.4) .. controls (.2,.6) and (.8,.6) .. (1,.4);
\draw (0,1.4) .. controls (.2,1.2) and (.8,1.2) .. (1,1.4);
\begin{pgfonlayer}{background}
	\draw[very thick, red!90!blue] (.5,.55) -- (.5,1.25);
\end{pgfonlayer}
\begin{scope}[xshift = 1.4cm]
\draw (0,.4) .. controls (.2,.6) and (.8,.6) .. (1,.4);
\draw (0,1.4) .. controls (.2,1.2) and (.8,1.2) .. (1,1.4);
\begin{pgfonlayer}{background}
	\draw[very thick, red!90!blue] (.5,.55) -- (.5,1.25);
\end{pgfonlayer}
\end{scope}
\begin{scope}[xshift=2.8cm]
\draw (0,.4) .. controls (.2,.6) and (.2,1.2) .. (0,1.4);
\draw (1,.4) .. controls (.8,.6) and (.8,1.2) .. (1,1.4);
\begin{pgfonlayer}{background}
	\draw[very thick, red!90!blue] (.15,.9) -- (.85,.9);
\end{pgfonlayer}
\end{scope}
\end{scope}
\draw (3,1.5) node[right]{\Large{$V$}};
\end{scope}
\begin{scope}[>=triangle 45]
	\draw[->,thick] (6.5,.5) -- (11,.5);
	\draw (8.5,.5) node[above]{$m$};
	\draw[->,thick] (6.5,2) -- (11, 6);
	\draw (8.5, 4.2) node[above]{$m$};
	\draw[->,thick] (6.5,-1) -- (11,-5);
	\draw (8.5, -2.6) node[above]{$\Delta$};
	\begin{scope}[xshift = 13cm]
		\draw[->,thick] (6.5,2) -- (11,6);
		\draw (7.75,3.2) node[above]{$\Delta$};
		\draw[o->,thick](6.5,-1)--(11,-5);
		\draw (7.75,-2) node[above]{$\Delta$};
	\end{scope}
	\begin{scope}[xshift=13cm,yshift = 7cm]
		\draw[o->,thick](6.5,.5)--(11,.5);
		\draw (8.5,.5) node[above]{$\Delta$};
		\draw[o-,thick] (6.5,-1) -- (8.5,-2.75);
		\draw (7.75,-2) node[above]{$\Delta$};
		\draw[->,thick] (9,-3.25)--(11,-5);
	\end{scope}
	\begin{scope}[xshift=13cm, yshift=-7cm]
		\draw[->,thick](6.5,.5) -- (11,.5);
		\draw (8.5,.5) node[above]{$m$};
		\draw[thick] (6.5,2) -- (8.5,3.75);
		\draw (7.75,3.2) node[above]{$m$};
		\draw[->,thick](9,4.25) -- (11,6);
	\end{scope}
	\begin{scope}[xshift = 26cm, yshift=7cm]
		\draw[->,thick](6.5,-1) -- (11,-5);
		\draw (8.5, -2.6) node[above]{$m$};
	\end{scope}
	\begin{scope}[xshift = 26cm]
		\draw[o->,thick] (6.5,.5) -- (11,.5);
		\draw (8.5,.5) node[above]{$m$};
	\end{scope}
	\begin{scope}[xshift = 26cm, yshift = -7cm]
		\draw[->,thick] (6.5,2) -- (11, 6);
		\draw (8.5, 4.2) node[above]{$m$};
	\end{scope}
\end{scope}
\end{tikzpicture}$$
\caption{The Kauffman states are arranged into a hypercube, in the same way that the spanning ribbon subgraphs are. Since this is the all-$A$ ribbon graph of a diagram of the unknot, the homology of this complex is isomorphic to $\mathbb{Z}\oplus\mathbb{Z}$.}
\label{fig:kauffman_cube}
\end{figure}

\section{Virtual links}
\label{section:virtual}

A {\em virtual link diagram} is a closed one-manifold generically immersed in the plane so that each double point is either a {\em classical crossing} or a {\em virtual crossing}. Classical crossings are depicted exactly as in classical knot theory, and virtual crossings are depicted by a small circle surrounding the double point. See Figure \ref{fig:virtual}.
\begin{figure}[h]
$$\begin{tikzpicture}
\draw[thick] (0,0) -- (2,2);
\draw[thick] (2,0) -- (1.25,0.75);
\draw[thick] (0.75,1.25) -- (0,2);
\draw (1,0) node[below]{Classical};

\begin{scope}[xshift = 3cm]
\draw[thick] (0,0) -- (2,2);
\draw[thick] (2,0) -- (0,2);
\draw (1,1) circle (.2cm);
\draw (1,0) node[below]{Virtual};
\end{scope}

\begin{scope}[xshift=9cm,yshift = 0 cm, rounded corners = .75cm]


\draw[thick,xshift=-1cm,yshift=1.732cm] (230:2cm) arc (230:-110:2cm);
\draw[thick] (0,0) circle (2cm);
\draw[thick, xshift=1cm, yshift=1.732cm] (190:2cm) arc (190:230:2cm);
\draw[thick,xshift=1cm,yshift=1.732cm] (170:2cm) arc (170:-110:2cm);

\draw (2,0) circle (.2cm);
\draw (1,1.732) circle (.2cm);
\draw (0,3.464) circle (.2cm);

\end{scope}

\end{tikzpicture}$$
\caption{A classical crossing, a virtual crossing, and a virtual link diagram}
\label{fig:virtual}
\end{figure}
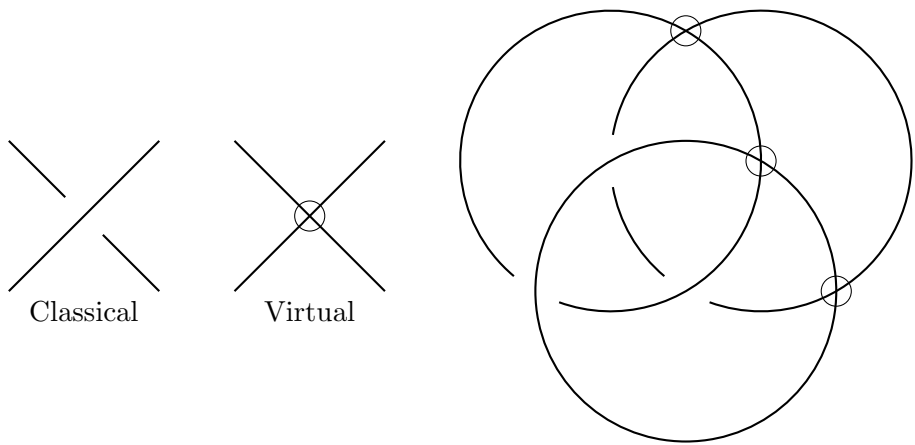

A {\em virtual link} is an equivalence class of virtual link diagrams where two virtual link diagrams are equivalent if they are related by a finite sequence of classical or virtual Reidemeister moves (see Figures \ref{fig:Rmoves} and \ref{fig:virtualRmoves}). Virtual link theory was discovered by Kauffman \cite{Kauffman:VirtualKnotTheory} and rediscovered by Goussarov, Polyak, and Viro \cite{GPV:VirtualKnots}.
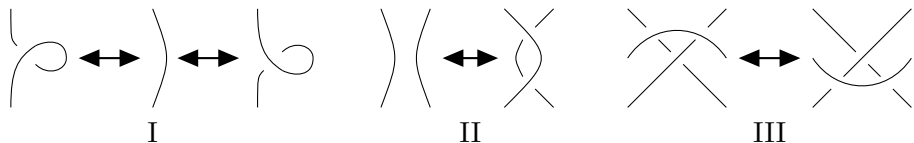
\begin{figure}[h]
$$\begin{tikzpicture}[scale = .82]

\draw (0,.4) .. controls (0,1) and (.1,1.2) .. (.4,1.4);
\draw (0,2) .. controls (0,1.5) and (0,1.5) .. (.1,1.4);
\draw (.4,1.4) .. controls (.6,1.5) and (.9,1.5) .. (.9,1.2);
\draw (.4,1.1) .. controls (.6,.9) and (.9,1) .. (.9,1.2);

\draw (2.3,.4) .. controls (2.6,1.2) and (2.6,1.2) .. (2.3,2);

\begin{scope}[xshift = 4cm]
	\draw (0,.4) .. controls (0,.9) and (0,.9) .. (.1,1);
	\draw (0,2) .. controls (0,1.4) and (.1,1.2) .. (.4,1);
	\draw (.4,1) .. controls (.6,.9) and (.9,.9) .. (.9, 1.2);
	\draw (.9,1.2) ..controls (.9,1.4) and (.6,1.5) .. (.4,1.3);
\end{scope}

\begin{scope}[>=triangle 45, thick]
	\draw[<->] (1.1,1.2) -- (2.1,1.2);
	\draw[<->] (2.7,1.2) -- (3.7,1.2);
\end{scope}

\draw (2.3,0) node{I};



\begin{scope}[xshift=6cm]

	\draw (0,.4) .. controls (.3,1.2) and (.3,1.2) .. (0,2);
	\draw (.8,.4) .. controls (.5,1.2) and (.5,1.2) .. (.8,2);
	
	\draw (2,.4) .. controls (2.8,1.2) and (2.8,1.2) .. (2,2);
	\draw (2.8,.4) .. controls (2.6,.6) and (2.6,.6) .. (2.5,.7);
	\draw (2.8,2) .. controls (2.6,1.8) and (2.6,1.8) .. (2.5,1.7);
	\draw (2.3,.9) .. controls (2.15,1.2) and (2.15,1.2) .. (2.3,1.5);
	
	\draw[thick, >=triangle 45,<->] (1,1.2) -- (1.9,1.2);
	
	\draw (1.45,0) node{II};
	
\end{scope}



\begin{scope}[xshift = 10cm]

	\draw (0,1.2) .. controls (0.4,1.8) and (1.2,1.8) .. (1.6,1.2);
	\draw (0,.4) -- (1.1,1.5);
	\draw (1.3,1.7) -- (1.6,2);
	\draw (0,2) -- (.3,1.7);
	\draw (.5,1.5) -- (.7,1.3);
	\draw (.9,1.1) -- (1.6,.4);

	\draw (3,1.2) .. controls (3.4,.6) and (4.2,.6) .. (4.6,1.2);
	\draw (3,.4) -- (3.3,.7);
	\draw (4.6,.4) -- (4.3,.7);
	\draw (4.1,.9) -- (3.9,1.1);
	\draw (3.7,1.3) -- (3,2);
	\draw (3.5,.9) -- (4.6,2);
	
	\draw[thick, >=triangle 45, <->] (1.8,1.2) -- (2.8,1.2);
	
	\draw (2.3,0) node{III};

\end{scope}

\end{tikzpicture}$$
\caption{The classical Reidemeister moves}
\label{fig:Rmoves}
\end{figure}

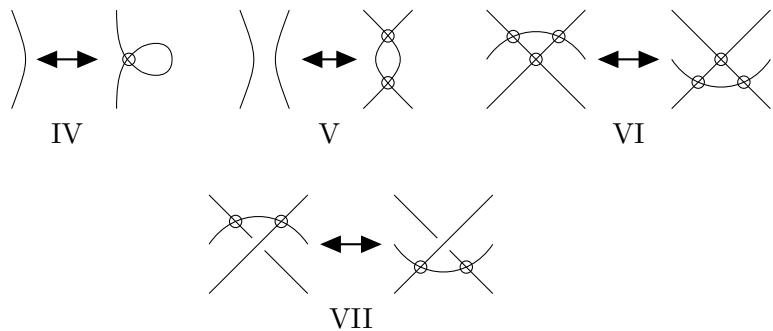
\begin{figure}[h]
$$\begin{tikzpicture}[scale = .82]

\draw (2.3,.4) .. controls (2.6,1.2) and (2.6,1.2) .. (2.3,2);

\begin{scope}[xshift = 4cm]
	\draw (0,.4) .. controls (0,1) and (.1,1.2) .. (.4,1.4);
	\draw (0,2) .. controls (0,1.4) and (.1,1.2) .. (.4,1);
	\draw (.4,1) .. controls (.6,.9) and (.9,.9) .. (.9, 1.2);
	\draw (.4,1.4) .. controls (.6,1.5) and (.9,1.5) .. (.9,1.2);
	\draw (.2,1.2) circle (.1cm);

\end{scope}

\begin{scope}[>=triangle 45, thick]
	\draw[<->] (2.7,1.2) -- (3.7,1.2);
\end{scope}

\draw (3.2,0) node{IV};



\begin{scope}[xshift=6cm]

	\draw (0,.4) .. controls (.3,1.2) and (.3,1.2) .. (0,2);
	\draw (.8,.4) .. controls (.5,1.2) and (.5,1.2) .. (.8,2);
	
	\draw (2,.4) .. controls (2.8,1.2) and (2.8,1.2) .. (2,2);
	\draw (2.8,.4) .. controls (2,1.2) and (2,1.2) .. (2.8,2);
	\draw (2.4,.82) circle (.1cm);
	\draw (2.4,1.58) circle (.1cm);

	\draw[thick, >=triangle 45,<->] (1,1.2) -- (1.9,1.2);
	
	\draw (1.45,0) node{V};
	
\end{scope}



\begin{scope}[xshift = 10cm]

	\draw (0,1.2) .. controls (0.4,1.8) and (1.2,1.8) .. (1.6,1.2);
	\draw (0,.4) -- (1.6,2);
	\draw (1.6,.4) -- (0,2);
	\draw (.8,1.2) circle (.1cm);
	\draw (.43,1.57) circle (.1cm);
	\draw (1.17, 1.57) circle (.1cm);

	\draw (3,1.2) .. controls (3.4,.6) and (4.2,.6) .. (4.6,1.2);
	\draw (3,.4) -- (4.6,2);
	\draw (3,2) -- (4.6,.4);
	\draw (3.8,1.2) circle (.1cm);
	\draw (4.17, .83) circle (.1cm);
	\draw (3.43,.83) circle (.1cm);
	
	\draw[thick, >=triangle 45, <->] (1.8,1.2) -- (2.8,1.2);
	
	\draw (2.3,0) node{VI};

\end{scope}


\begin{scope}[xshift = 5.5cm, yshift=-3cm]

	\draw (0,1.2) .. controls (0.4,1.8) and (1.2,1.8) .. (1.6,1.2);
	\draw (0,.4) -- (1.6,2);
	\draw (1.6,.4) -- (.9,1.1);
	\draw (.7,1.3) -- (0,2);
	
	\draw (.43,1.57) circle (.1cm);
	\draw (1.17, 1.57) circle (.1cm);

	\draw (3,1.2) .. controls (3.4,.6) and (4.2,.6) .. (4.6,1.2);
	\draw (3,.4) -- (4.6,2);
	\draw (4.6,.4) -- (3.9,1.1);
	\draw (3.7,1.3) -- (3,2);
	\draw (4.17, .83) circle (.1cm);
	\draw (3.43,.83) circle (.1cm);
	
	\draw[thick, >=triangle 45, <->] (1.8,1.2) -- (2.8,1.2);
	
	\draw (2.3,0) node{VII};

\end{scope}

\end{tikzpicture}$$
\caption{The virtual Reidemeister moves}
\label{fig:virtualRmoves}
\end{figure}

\subsection{Ribbon graphs and virtual links}

Chmutov \cite{Chmutov:GeneralizedDuality} extended the construction of the all-$A$ ribbon graph $\mathbb{D}$ of a link diagram to virtual link diagrams. The following procedure produces an arrow presentation $P(D)$ of the all-$A$ ribbon graph $\mathbb{D}$ for each virtual link diagram $D$. Label the classical crossings of $D$ by $1,2,\dots, n$. For each classical crossing of $D$, choose the $A$-resolution and record marking arrows near the resolution whose orientations are one of the two equivalent choices depicted in Figure \ref{fig:markarrowresolution}. The label of the pair of marking arrows is the same as the label of the crossing (in Figure \ref{fig:markarrowresolution}, the label of the crossing and marking arrows is $x$). 

After resolving each crossing, the resulting diagram is a collection of immersed curves $\widetilde{c}_1, \widetilde{c}_2, \dots, \widetilde{c}_k$ in the plane with marking arrows along the curves. Each immersed curve $\widetilde{c}_i$ has an associated circle $c_i$ in the arrow presentation $P(D)$. Traversing the curve $\widetilde{c}_i$ (in either direction) gives a cyclic order of the marking arrows along $\widetilde{c}_i$. Additionally, the orientation of each marking arrow on $\widetilde{c}_i$ either agrees or disagrees with the direction $\widetilde{c}_i$ is traversed. Place marking arrows on the circle $c_i$ in $P(D)$ so that the cyclic order of the marking arrows on $c_i$ (also in either direction) is the same as the cyclic order of the arrows around $\widetilde{c}_i$. Additionally, each marking arrow on $c_i$ agrees with the direction that $c_i$ is traversed if and only if the corresponding marking arrow on $\widetilde{c}_i$ agrees with the direction that $\widetilde{c}_i$ is traversed. 

Said more informally, the collection of immersed curves $\widetilde{c}_1,\dots, \widetilde{c}_k$ can be considered as a virtual link diagram $\widetilde{D}$ with only virtual crossings. The arrow presentation $P(D)$ is obtained by transforming $\widetilde{D}$ into a $k$-component unlink using the virtual Reidemeister moves while carrying the marking arrow information in the obvious way. See Figure \ref{fig:virtualribbon} for an example. 
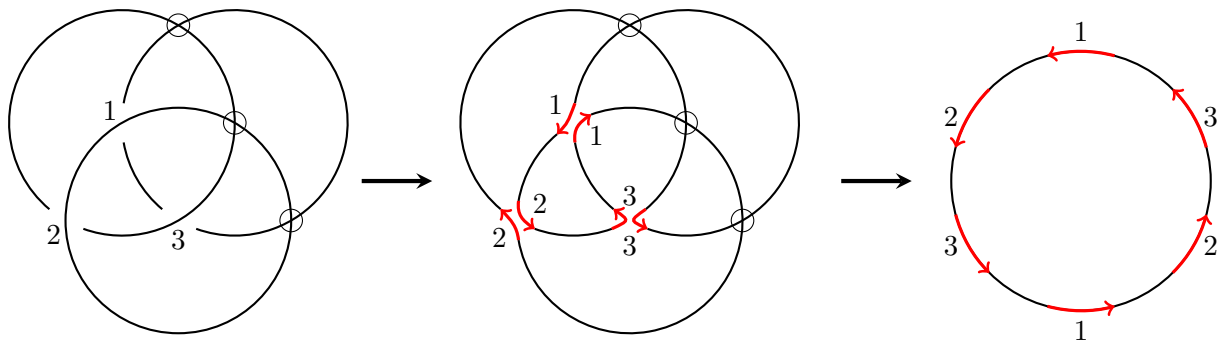
\begin{figure}[h]
$$\begin{tikzpicture}[scale=.75]
\draw[thick,xshift=-1cm,yshift=1.732cm] (230:2cm) arc (230:-110:2cm);
\draw[thick] (0,0) circle (2cm);
\draw[thick, xshift=1cm, yshift=1.732cm] (190:2cm) arc (190:230:2cm);
\draw[thick,xshift=1cm,yshift=1.732cm] (170:2cm) arc (170:-110:2cm);

\draw (2,0) circle (.2cm);
\draw (1,1.732) circle (.2cm);
\draw (0,3.464) circle (.2cm);

\draw (-.9,1.952) node[left]{$1$};
\draw (0,0) node[below]{$3$};
\draw (-1.9,-0.25) node[left]{$2$};

\begin{scope}[xshift = 8cm]
\draw (-1,2.052) node[left]{$1$};
\draw (-.9,1.5) node[right]{$1$};
\draw (0,-0.1) node[below]{$3$};
\draw (0, 0.1) node[above]{$3$};
\draw (-2.0,-0.3) node[left]{$2$};
\draw (-1.9,0.3) node[right]{$2$};

\path[xshift=-1cm,yshift=1.732cm] (230:2cm)  coordinate (P1);
\path[xshift=-1cm,yshift=1.732cm] (250:2cm) coordinate (P4);
\path(170:2cm) coordinate (P3);
\path(190:2cm) coordinate (P2);

\path[xshift=1cm,yshift=1.732cm] (170:2cm) coordinate (Q1);
\path[xshift=1cm,yshift=1.732cm] (190:2cm) coordinate (Q4);
\path (130:2cm) coordinate (Q2);
\path (110:2cm) coordinate (Q3);

\path[xshift=-1cm,yshift=1.732cm] (290:2cm) coordinate (R2);
\path[xshift=-1cm,yshift=1.732cm] (310:2cm) coordinate (R3);
\path[xshift=1cm,yshift=1.732cm] (230:2cm) coordinate (R1);
\path[xshift=1cm,yshift=1.732cm] (250:2cm) coordinate (R4);

\draw[thick] (P1) arc (230:-50:2cm);
\draw[thick] (R2) arc (290:250:2cm);
\draw[thick] (Q2) arc (130:170:2cm);
\draw[thick] (P2) arc (-170:110:2cm);
\draw[thick] (Q1) arc (170:-110:2cm);
\draw[thick] (Q4) arc (190:230:2cm);

\begin{scope}[very thick, red,->]
\draw (P2) .. controls (-2,-.1) and (-2.1,0) .. (P1);
\draw (P3) .. controls (-2,.1) and (-1.9,0) .. (P4);

\draw (Q1) .. controls (-1.1,1.732) and (-1.1,1.732) .. (Q2);
\draw (Q4) .. controls (-1,1.632) and (-.9,1.732) .. (Q3);

\draw (R2) .. controls (0,0) and (0,0) .. (R1);
\draw (R3) .. controls (0,0) and (0,0) .. (R4);

\end{scope}

\draw (2,0) circle (.2cm);
\draw (1,1.732) circle (.2cm);
\draw (0,3.464) circle (.2cm);

\end{scope}

\begin{scope}[xshift = 16cm, yshift = .7cm]
\draw (90:2.3cm) node[above]{$1$};
\draw (150:2.3cm) node[left]{$2$};
\draw (210:2.3cm) node[left]{$3$};
\draw (270:2.3cm) node[below]{$1$};
\draw (330:2.3cm) node[right]{$2$};
\draw (30:2.3cm) node[right]{$3$};

\begin{pgfonlayer}{background}
\draw[thick] (0,0) circle (2.3cm);
\end{pgfonlayer}
\begin{scope}[very thick, red,->]
\draw (75:2.3cm) arc (75:105:2.3cm);
\draw (135:2.3cm) arc (135:165:2.3cm);
\draw (195:2.3cm) arc (195:225:2.3cm);
\draw (255:2.3cm) arc (255:285:2.3cm);
\draw (315:2.3cm) arc (315:345:2.3cm);
\draw (15:2.3cm) arc (15:45:2.3cm);
\end{scope}

\end{scope}

\draw[ultra thick,>=stealth,->] (3.25,.7) -- (4.5,.7);
\draw[ultra thick,>=stealth,->] (11.75,.7) -- (13,.7);

\end{tikzpicture}$$
\caption{Transforming a virtual link diagram $D$ into an arrow presentation $P(D)$.}
\label{fig:virtualribbon}
\end{figure}

Recall that each arrow presentation $P(D)$ has an associated surface $\Sigma_{\mathbb{D}}$ obtained considering each circle of $P(D)$ as the boundary of a vertex disk and by attaching edge bands to identically labeled marking arrows as in Figure \ref{fig:band}. The arrow presentation $P(D)$ is orientable if $\Sigma_{\mathbb{D}}$ is orientable; otherwise $P(D)$ is non-orientable. Chmutov and Pak \cite{ChmutovPak:BollobasRiordan} observed that every ribbon graph can be obtained as the all-$A$ ribbon graph of some (not necessarily unique) virtual link diagram. See Figure \ref{fig:sameribbon}. 
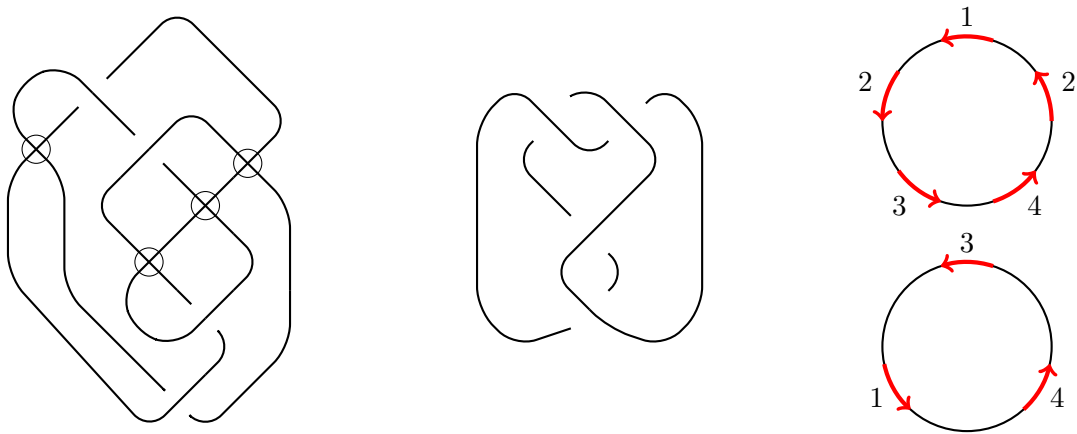
\begin{figure}[h]
$$\begin{tikzpicture}[scale=.75]
\draw[thick](0.5,8.5) -- (1.25,9.25);
\draw[thick](1.5,9.5) -- (2.25, 8.75);
\draw[thick](2.5,8.5) -- (2,8);
\draw[thick](2.75,8.25) -- (4,7);
\draw[thick] (4,8) -- (2.5,6.5);
\draw[thick] (3,6) -- (2,7);
\draw[thick] (3,6) -- (3.25,5.75);
\draw[thick] (4,6) -- (3.5,5.5);

\begin{scope}[rounded corners = 3mm, thick]
	\draw (.5,8.5) -- (0,9) -- (.25,9.75) -- (1,10) -- (1.5,9.5);
	\draw (.5,8.5) -- (1,8) -- (1,6) -- (2.5,4.5) -- (2.75, 4.25);
	\draw (3.5,4.5) -- (2.5,3.5) -- (0,6.25) -- (0,8) -- (.5,8.5);
	\draw (3.5,5.5) -- (3,5) -- (2.25,5.25) -- (2,6) -- (2.5,6.5);
	\draw (3.75,5.25) -- (4,5) -- (3.5,4.5);
	\draw (4,7) -- (4.5,6.5) -- (4,6);
	\draw (2,7) -- (1.5,7.5) -- (2,8);
	\draw (2.5,8.5) -- (3.25,9.25) -- (5,7.5) -- (5,6);
	\draw (4,8) -- (5,9) -- (3,11) -- (1.75,9.75);
	\draw (3.25,3.75) -- (3.5,3.5) -- (5,5) -- (5,6);
\end{scope}

\draw (0.5,8.5) circle (.25cm);
\draw (4.25,8.25) circle (.25cm);
\draw (3.5,7.5) circle (.25cm);
\draw (2.5,6.5) circle (.25cm);


\begin{scope}[thick, rounded corners = 3mm, xshift= 1cm, yshift=1cm, scale=1.33]
\draw (7,3.5) -- (6.5,4) -- (8,5.5) -- (7,6.5) -- (6.75,6.35);
\draw (7.25,3.75) -- (7.5,4) -- (7.25,4.25);
\draw (6.75,4.75) -- (6,5.5) -- (6.25,5.75);
\draw (7.25, 5.75) -- (7,5.5) -- (6,6.5) -- (5.5,6) -- (5.5,3.5) -- (6,3)  -- (6.75,3.25);
\draw (7,3.5) -- (7.25,3.25) -- (8,3) -- (8.5,3.5) -- (8.5,6) -- (8,6.5) -- (7.75,6.25) ;
\end{scope}


\begin{pgfonlayer}{background}
	\draw[thick] (17,9) circle (1.5cm);
\end{pgfonlayer}
\begin{scope}[xshift=17cm,yshift=9cm] 
	\draw[ultra thick, red,->] (72:1.5cm) arc (72:108:1.5cm);
	\draw[ultra thick, red,->] (144:1.5cm) arc (144:180: 1.5cm);
	\draw[ultra thick, red,->] (216:1.5cm) arc (216: 252:1.5cm);
	\draw[ultra thick, red,->] (288:1.5cm) arc (288: 324:1.5cm);
	\draw[ultra thick, red,->] (0:1.5cm) arc (0:36:1.5cm);
\end{scope}
\draw (17,10.5) node[above]{$1$};
\draw (15.2, 9.7) node{$2$};
\draw (15.8,7.5) node{$3$};
\draw (18.2,7.5) node{$4$};
\draw (18.8,9.7) node{$2$};

\begin{pgfonlayer}{background}
	\draw[thick] (17,5) circle (1.5cm);
\end{pgfonlayer}
\begin{scope}[xshift=17cm, yshift=5cm]
	\draw[ultra thick, red,->] (72:1.5cm) arc (72:108:1.5cm);
	\draw[ultra thick, red, ->] (192:1.5cm) arc (192:228:1.5cm);
	\draw[ultra thick, red, ->] (312:1.5cm) arc (312:348:1.5cm);
\end{scope}
\draw (17, 6.5) node[above]{$3$};
\draw (15.4, 4.1) node{$1$};
\draw (18.6,4.1) node{$4$};

\end{tikzpicture}$$
\caption{A virtual knot diagram and a classical knot diagram with the same all-$A$ ribbon graph, whose arrow presentation is depicted.}
\label{fig:sameribbon}
\end{figure}

A virtual link diagram is {\em alternating} if the classical crossings alternate between over and under along each component of the virtual link. A virtual link diagram is called {\em alternatable} if it can be transformed into an alternating diagram via some number of crossing changes. Viro \cite{Viro:ChordDiagrams} gave various characterizations of alternatable virtual link diagrams. The following proposition provides one more characterization of alternatable virtual link diagrams.

\begin{proposition}
\label{prop:alternatable}
A virtual link diagram $D$ is alternatable if and only if its all-$A$ ribbon graph $\mathbb{D}$ is orientable.
\end{proposition}
\begin{proof}
Let $\Sigma_{\mathbb{D}}$ be the surface associated to the all-$A$ arrow presentation of $D$. A cycle in $\Sigma_{\mathbb{D}}$ is a collection of vertex disks $v_1,\dots,v_k$ and edge bands $e_1,\dots, e_k$ such that the edge band $e_i$ is attached to the vertex disks $v_{i}$ and $v_{i+1}$ (with subscripts taken modulo $k$). A cycle in $\Sigma_{\mathbb{D}}$ is either an embedded annulus or an embedded M\"obius band. The surface $\Sigma_{\mathbb{D}}$ is orientable if and only if it contains no embedded M\"obius bands. We proceed by induction on the number of cycles contained in $\Sigma_{\mathbb{D}}$.

Suppose that $\Sigma_{\mathbb{D}}$ contains no cycles. Then $D$ is alternatable since it can be transformed into a diagram of the unknot using only Reidemeister I moves and the virtual Reidemeister moves. Moreover, since $\Sigma_{\mathbb{D}}$ contains no cycles, it is necessarily orientable.

Suppose that $\Sigma_{\mathbb{D}}$ contains one cycle. Then $D$ can be transformed into a closed virtual two-braid $D^\prime$ using only Reidemeister I moves and the virtual Reidemeister moves. If $D^\prime$ contains an odd number of virtual crossings, then $D^\prime$, and hence $D$, is not alternatable and $\Sigma_{\mathbb{D}^\prime}$ is a M\"obius band. If $D^\prime$ contains an even number of virtual crossings, then $D^\prime$, and hence $D$, is alternatable and $\Sigma_{\mathbb{D^\prime}}$ is an annulus.

Suppose that $D$ is a virtual diagram such that $\Sigma_{\mathbb{D}}$ contains at least two cycles $C_1$ and $C_2$. Let $e_1$ and $e_2$ be two edge bands in $\Sigma_{\mathbb{D}}$ such that $e_1$ is in $C_1$ but not $C_2$ and $e_2$ is in $C_2$ but not $C_1$. Let $c_1$ and $c_2$ be the two crossings of $D$ corresponding to $e_1$ and $e_2$, and let $D_1$ and $D_2$ be the diagrams obtained by performing an $A$-resolution at $c_1$ and $c_2$ respectively. Let $\Sigma_{\mathbb{D}_1}$ and $\Sigma_{\mathbb{D}_2}$ be the surfaces obtained from the arrow presentations of $D_1$ and $D_2$. The surfaces $\Sigma_{\mathbb{D}_1}$ and $\Sigma_{\mathbb{D}_2}$ can be obtained from the surface $\Sigma_{\mathbb{D}}$ by deleting the edge bands $e_1$ and $e_2$ respectively.

If $D$ is alternatable, then $D_1$ and $D_2$ are alternatable, and since they each contain fewer cycles, the surfaces $\Sigma_{\mathbb{D}_1}$ and $\Sigma_{\mathbb{D}_2}$ are orientable. Since both $\Sigma_{\mathbb{D}_1}$ and $\Sigma_{\mathbb{D}_2}$ are orientable, it follows that $\Sigma_{\mathbb{D}}$ is orientable. If $\Sigma_{\mathbb{D}}$ is orientable, then $\Sigma_{\mathbb{D}_1}$ and $\Sigma_{\mathbb{D}_2}$ are orientable, and hence $D_1$ and $D_2$ are alternatable. Since $D_1$ and $D_2$ are alternatable, it follows that $D$ is alternatable.
\end{proof}

\subsection{Khovanov homology of virtual links}

Manturov \cite{Manturov:KhovanovVirtual} extended the definition of a Khovanov homology to virtual links. The construction of Khovanov homology for virtual links is similar to the construction for classical links. Let $D$ be a virtual link diagram with classical crossings labeled $1,\dots,n$, and let $I=(m_1,\dots,m_n)$ be a vertex in the hypercube $\{0,1\}^n$. The Kauffman state of $D(I)$ is the collection of immersed curves obtained by choosing an $A$-resolution for those crossings where $m_i=0$ and a $B$-resolution for  those crossings where $m_i=1$. 

Associate to each vertex $I\in\mathcal{V}(n)$ the space $V(D(I)):=V^{\otimes|D(I)|}[h(I)]\{h(I)\}$ where $|D(I)|$ denotes the number of (possibly immersed) curves in the Kauffman state $D(I)$. Define $CKh(D) = \bigoplus_{I\in\mathcal{V}(n)} V(D(I))[-n_-]\{n_+-2n_-\}$ where $n_+$ and $n_-$ are the number of positive and negative (classical) crossings of $D$, respectively. As before, it will useful to refer to all summands in a specific homological grading but arbitrary polynomial grading; therefore, we write $CKh^{i,*}(D) = \bigoplus_j CKh^{i,j}(D)$.

The differential in the Khovanov complex of a virtual link diagram is signed sum of maps corresponding to the edges of the hypercube $\{0,1\}^n$. Suppose that there is an edge $\xi$ in the hypercube from vertex $I$ to vertex $J$. The number of curves in $D(J)$ is either one greater, one less, or the same as the number of curves in $D(I)$. If $|D(J)| = |D(I)|+1$, then define the edge map $d_{\xi}$ to be $\Delta:V\to V\otimes V$ on the tensor factor of $V$ corresponding to the curve in $D(I)$ that is split into two curves in $D(J)$. If $|D(J)| = |D(I)|-1$, then define the edge map $d_{\xi}$ to be $m:V \otimes V\to V$ on the two tensor factors of $V$ corresponding to the curves in $D(I)$ that are merged into one curve in $D(J)$. In either case, define $d_{\xi}$ to be the identity on those tensor factors of $V$ that do not correspond to either merging or splitting circles. There are various constructions of the edge map between vertices where $|D(I)| = |D(J)|$ (one possibility is to set that edge map to zero). However, if such an edge of the hypercube exists, then the all-$A$ ribbon graph $\mathbb{D}$ of $D$ is non-orientable. Since our focus in this paper is on oriented ribbon graphs, we are able to avoid any issues that arise when $|D(J)|=|D(I)|$.

The differential $d^i:CKh^{i,*}(D)\to CKh^{i+1,*}(D)$ is defined as $d^i=\sum_{|\xi|=i - n_-}(-1)^{\xi} d_{\xi}$, where $\xi$ is an edge corresponding to crossing $k$, originating from $I=(m_1,\dots,m_n)$, and $(-1)^{\xi}=\sum_{i=1}^{k} m_i$. The Khovanov homology $Kh(D)$ of the virtual link with diagram $D$ is defined to be the homology of the complex $(CKh(D),d)$. Once again this homology is bigraded, and we write $Kh(D) =\oplus_{i,j}Kh^{i,j}(D)$. The reduced Khovanov homology $\widetilde{Kh}(D)$ of a virtual link diagram $D$ can be defined in a similar manner as the classical link case.

Theorem \ref{theorem:maintheorem1} can be generalized to virtual link diagrams with orientable all-$A$ ribbon graphs.
\begin{theorem}
Let $L$ be a virtual link with diagram $D$ whose all-$A$ ribbon graph $\mathbb{D}$ is orientable. Suppose that $D$ has $n_+$ positive crossings and $n_-$ negative crossings. There are grading preserving isomorphisms
\begin{eqnarray*}
Kh(\mathbb{D})[-n_-]\{n_+-2n_-\} & \cong & Kh(L)~\text{and}\\
\widetilde{Kh}(\mathbb{D})[-n_-]\{n_+-2n_-\} & \cong &\widetilde{Kh}(L).
\end{eqnarray*}
\end{theorem}
\begin{proof}
The proof of this theorem is identical to the proof of Theorem \ref{theorem:maintheorem1}.
\end{proof}
\section{Reidemeister moves}
\label{section:Reidemeister}

In this section we define local moves on ribbon graphs that generalize Reidemeister moves on both classical and virtual link diagrams. The main result of this section states that if two ribbon graphs are related by a sequence of these moves, then their Khovanov homologies are isomorphic. The ribbon graph Reidemeister moves are most easily described using arrow presentations. Let $\mathbb{G}$ be a ribbon graph with arrow presentation $P$.
\medskip

\noindent{\bf Reidemeister I:} The following two operations on $P$ (and their inverses) are called {\em Reidemeister I moves}. See Figure \ref{fig:R1}.
\begin{itemize}
\item Add an additional circle $C$ to $P$ and add a pair of marking arrows with the same label $x$ to $P$ such that one arrow lies on $C$ and the other arrow lies one of the circles of $P$. 

\item Add two adjacent marking arrows with the same orientation and  label $x$ to an arc of a circle of $P$.

\end{itemize}
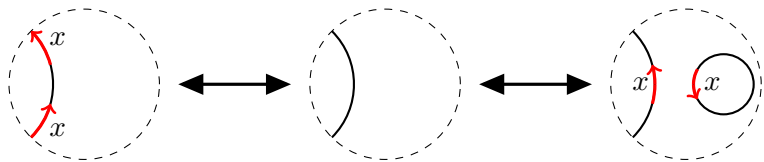
\begin{figure}[h]
$$\begin{tikzpicture}
\draw[dashed] (0,0) circle (1cm);
\draw[thick] (-.707,-.707) arc (-45:45:1cm);

\begin{scope}[xshift=4cm]
\draw[dashed] (0,0) circle (1cm);
\begin{pgfonlayer}{background}
	\draw[thick] (-.707,-.707) arc (-45:45:1cm);
	\draw[thick] (.5,0) circle (.4cm);
\end{pgfonlayer}
\draw[red,very thick,->] (-.413,0) arc (0:15:1cm);
\draw[red,very thick] (-.413,0) arc (0:-15:1cm);
\draw[red,very thick] (.1,0) arc (180:150:.4cm);
\draw[red,very thick,->] (.1,0) arc (180:210:.4cm);
\draw (-.6,0) node{$x$};
\draw (.35,0) node{$x$};
\end{scope}

\begin{scope}[xshift=-4cm]
	\draw[dashed] (0,0) circle (1cm);
	\draw[thick] (-.707,-.707) arc (-45:45:1cm);
	\draw[red,very thick, ->] (-.707,-.707) arc (-45:-15:1cm);
	\draw[red,very thick, <-] (-.707,.707) arc (45:15:1cm);
	\draw (-.35,-.6) node{$x$};
	\draw (-.35,.6) node{$x$};
\end{scope}

\draw[very thick, >=triangle 45,<->] (1.25,0) -- (2.75,0);
\draw[very thick, >=triangle 45,<->] (-2.75,0) -- (-1.25,0);

\end{tikzpicture}$$
\caption{The two different Reidemeister I moves on an arrow presentation.}
\label{fig:R1}
\end{figure}

The first type of Reidemeister I move adds an isolated vertex to $\mathbb{G}$ and then connects that isolated vertex to $\mathbb{G}$ with a single edge, and the second type of Reidemeister I move adds a loop to a vertex $v$ of $\mathbb{G}$ such that the two half-edges of the loop are adjacent in the cyclic order of the half-edges around $v$.
\medskip

\noindent{\bf Reidemeister II:} The following move (and its inverse) is called a {\em Reidemeister II move}. Suppose that $P$ contains two arcs (segments of the circles of $P$) with no marking arrows. Then the two arcs can be replaced as in Figure \ref{fig:R2} in such a way that all circles of the resulting arrow presentation are non-nested. Also, two new pairs of marking arrows, labeled $x$ and $y$, are added to the resulting arrow presentation, as in Figure \ref{fig:R2}.
\begin{figure}[h]
$$\begin{tikzpicture}
\draw[dashed] (0,0) circle (1cm);
\draw[thick] (-.707,-.707) arc (-45:45:1cm);
\draw[thick] (.707, .707) arc (135:225:1cm);

\begin{scope}[xshift = 5cm]
	\draw[dashed] (0,0) circle (1cm);
	\draw[thick] (-.707,-.707) arc (135:45:1cm);
	\draw[thick] (.707,.707) arc (-45:-135:1cm);
	\draw[very thick, ->,red] (-.707,-.707) arc (135: 115: 1cm);
	\draw[very thick, <-,red] (.707,-.707) arc (45:65:1cm);
	\draw[very thick, ->, red] (0,-.413) arc (90:75:1cm);
	\draw[very thick, red] (0,-.413) arc (90:105:1cm);
	\draw[very thick, red, ->] (0,.413) arc (270:255:1cm);
	\draw[very thick, red] (0,.413) arc (270:285:1cm);
	
	\draw (-.707,-.707) node[left]{$x$};
	\draw (.707,-.707) node[right]{$x$};
	\draw (0,-.413) node[above]{$y$};
	\draw (0,.413) node[above]{$y$};
\end{scope}


\draw[very thick, <->, >=triangle 45] (1.5,0) -- (3.5,0);

\end{tikzpicture}$$
\caption{A Reidemeister II move on an arrow presentation.}
\label{fig:R2}
\end{figure}
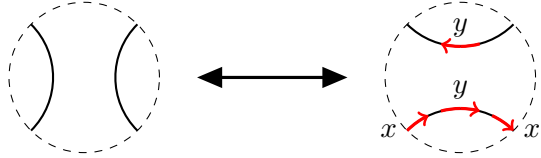

If the two arcs in the Reidemeister II move lie on different circles in the arrow presentation $P$, then the move merges two vertices of $\mathbb{G}$ into one. It is possible to transform an oriented ribbon graph into a non-orientable ribbon graph by merging two vertices into one. However, if the arcs lie on the same circle in $P$, then the move splits a single vertex of $\mathbb{G}$ into two.
\medskip

\noindent{\bf Reidemeister III:} The following move (and its inverse) is called a {\em Reidemeister III move}. If three arcs on $P$ have three pairs of marking arrows labeled as in the left side of Figure \ref{fig:R3}, then the marking arrows can be changed to the right side of Figure \ref{fig:R3}.
\begin{figure}[h]
$$\begin{tikzpicture}
\draw[dashed] (0,0) circle (1cm);
\draw[thick] (.707, .707) arc (315:225:1cm);
\draw[thick] (-.966,.259) arc (75:-15:1cm);
\draw[thick] (.966, .259) arc (105:195:1cm);

\draw[very thick,red,<-] (.707,.707) arc (315:290:1cm);
\draw[very thick, red, ->] (-.707,.707) arc (225:250:1cm);
\draw[very thick, red, ->] (0, .413) arc (270:280:1cm);
\draw[very thick,red] (0,.413) arc (270:260:1cm);
\draw[very thick, red,<-] (-.966,.259) arc (75:50:1cm);
\draw[very thick, red,->] (-.259,-.966) arc (-15:10:1cm);
\draw[very thick, red,->] (.51,0) arc (135:165:1cm);

\draw (-.707,.707) node[left]{$x$};
\draw (-.966,.259) node[left]{$x$};
\draw (.707,.707) node[right]{$z$};
\draw (-.3,-1) node[below]{$z$};
\draw (0,.7) node{$y$};
\draw (.6,-.3) node{$y$};
 
\begin{scope}[xshift = 5cm]
	\draw[dashed] (0,0) circle (1cm);
	\begin{pgfonlayer}{background}
	\draw[thick] (.707, .707) arc (315:225:1cm);
	\draw[thick] (-.966,.259) arc (75:-15:1cm);
	\draw[thick] (.966, .259) arc (105:195:1cm);
	\end{pgfonlayer}
	
	\draw[very thick,red,<-] (.707,.707) arc (315:290:1cm);
	\draw[very thick, red, ->] (-.707,.707) arc (225:250:1cm);
	\draw[very thick, red, ->] (0, .413) arc (270:280:1cm);
	\draw[very thick,red] (0,.413) arc (270:260:1cm);
	\draw[very thick, red,->] (.966,.259) arc (105:130:1cm);
	\draw[very thick, red,<-] (.259,-.966) arc (195:170:1cm);
	\draw[very thick, red,<-] (-.51,0) arc (45:15:1cm);
	
	\draw (-.707,.707) node[left]{$z$};
	\draw (.966,.259) node[right]{$z$};
	\draw (.707,.707) node[right]{$x$};
	\draw (.3,-1) node[below]{$x$};
	\draw (0,.7) node{$y$};
	\draw (-.6,-.3) node{$y$};
 
\end{scope}

\draw[very thick, <->, >=triangle 45] (1.5,0) -- (3.5,0);

\end{tikzpicture}$$
\caption{A Reidemeister III move on an arrow presentation.}
\label{fig:R3}
\end{figure}
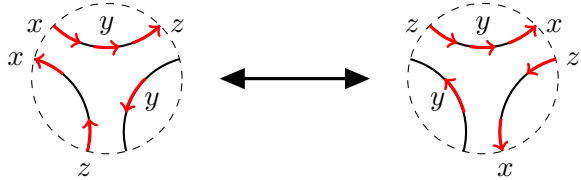

\begin{figure}[h]
$$\begin{tikzpicture}[scale=.9]


\draw (-8,0) node{\large{{\bf Reidemeister I:}}};


\def \middlebase{(0,0) circle (1cm);}
\def \middleside{(-1.414,0) circle (1cm);}

\begin{pgfonlayer}{background3}
\begin{scope}
	\clip \middlebase;
	\fill[blue!30!white] \middleside;
\end{scope}
\end{pgfonlayer}

\draw[dashed] (0,0) circle (1cm);

\draw[thick] (-.707,-.707) arc (-45:45:1cm);

\begin{scope}[xshift=4cm]

\def \rightbase{(0,0) circle (1cm)};
\def \rightside{(-1.414,0) circle (1cm)};

\draw[dashed] (0,0) circle (1cm);
	\draw[thick] (-.707,-.707) arc (-45:45:1cm);
	\draw[thick] (.5,0) circle (.4cm);

\begin{pgfonlayer}{background}
	\fill[blue!30!white] (.5,0) circle (.4cm);
\end{pgfonlayer}

\begin{pgfonlayer}{background3}
\begin{scope}
	\clip \rightbase;
	\fill[blue!30!white] \rightside;
\end{scope}
\end{pgfonlayer}

\begin{pgfonlayer}{background4}
	\draw[thick] (-.5,.15) -- (.5,.15);
	\draw[thick] (-.5,-.15) -- (.5,-.15);
\end{pgfonlayer}

\begin{pgfonlayer}{background5}
	\fill[blue!30!white] (-.5,.15) -- (.5,.15) -- (.5,-.15) -- (-.5,-.15);
\end{pgfonlayer}

\end{scope}

\begin{scope}[xshift=-4cm]

	\def \leftbase{(0,0) circle (1cm);}
	\def \leftside{(-1.414,0) circle (1cm);}
	
	\begin{pgfonlayer}{background2}
		\begin{scope}
			\clip \leftbase;
			\fill[blue!30!white] \leftside;
		\end{scope}
	\end{pgfonlayer}

	\draw[dashed] (0,0) circle (1cm);
	\draw[thick] (-.707,-.707) arc (-45:45:1cm);
	
	\def \edgeout{(-.35,0) circle (.6cm);}
	\def \edgein{(-.35,0) circle (.3cm)}
	
	\begin{pgfonlayer}{background3}
		\draw[thick] (-.35,0) circle (.6cm);
		\draw[thick] (-.35,0) circle (.3cm);
	\end{pgfonlayer}
	
	\begin{pgfonlayer}{background4}
		\fill[white] \edgein;
	\end{pgfonlayer}
	
	\begin{pgfonlayer}{background5}
		\fill[blue!30!white] \edgeout;
	\end{pgfonlayer}
	
\end{scope}

\draw[very thick, >=triangle 45,<->] (1.25,0) -- (2.75,0);
\draw[very thick, >=triangle 45,<->] (-2.75,0) -- (-1.25,0);


\draw (-8,-3) node{\large{{\bf Reidemeister II:}}};

\begin{scope}[xshift = -2.5cm, yshift=-3cm]

\def \firstmain{(0,0) circle (1cm);}
\def \firstleft{(-1.414,0) circle (1cm);}
\def \firstright{(1.414,0) circle (1cm);}

\begin{pgfonlayer}{background}
	\begin{scope}
		\clip \firstmain;
		\fill[white] \firstleft;
		\fill[white] \firstright;
	\end{scope}
\end{pgfonlayer}

\begin{pgfonlayer}{background2}
	\fill[blue!30!white] \firstmain;
\end{pgfonlayer}

\draw[dashed] (0,0) circle (1cm);
\draw[thick] (-.707,-.707) arc (-45:45:1cm);
\draw[thick] (.707, .707) arc (135:225:1cm);

\begin{scope}[xshift = 5cm]
	
	\def \secondmain{(0,0) circle (1cm);}
	\def \secondup{(0,1.414) circle (1cm);}
	\def \seconddown{(0,-1.414) circle (1cm);}
	
	\begin{pgfonlayer}{background}
		\begin{scope}
			\clip \secondmain;
			\fill[blue!30!white] \secondup;
			\fill[blue!30!white] \seconddown;
		\end{scope}
	\end{pgfonlayer}
	
	\draw[dashed] (0,0) circle (1cm);
	\draw[thick] (-.707,-.707) arc (135:45:1cm);
	\draw[thick] (.707,.707) arc (-45:-135:1cm);
	
	\begin{pgfonlayer}{background6}
		\fill[blue!30!white] (0,-.35) circle (.6cm);
	\end{pgfonlayer}
	
	\begin{pgfonlayer}{background5}
		\fill[white] (0,-.35) circle (.3cm);
	\end{pgfonlayer}
	
	\begin{pgfonlayer}{background4}
		\draw[thick] (0,-.35) circle (.6cm);
		\draw[thick] (0,-.35) circle (.3cm);
	\end{pgfonlayer}
	
	\begin{pgfonlayer}{background3}
		\fill[blue!30!white] (-.15,-.45) -- (-.15,.45) -- (.15,.45) -- (.15,-.45);
	\end{pgfonlayer}
	
	\begin{pgfonlayer}{background2}
		\draw[thick] (-.15,-.45) -- (-.15,.45);
		\draw[thick] (.15,-.45) -- ( .15,.45);
	\end{pgfonlayer}
	
\end{scope}


\draw[very thick, <->, >=triangle 45] (1.5,0) -- (3.5,0);
\end{scope}


\draw (-8,-6) node{\large{{\bf Reidemeister III:}}};

\begin{scope}[xshift  = -2.5cm, yshift = -6cm]

\def \firstmain{(0,0) circle (1cm);}
\def \firstup{(0,1.414) circle (1cm);}
\def \firstleft{(-1.22497, -.706886) circle (1cm);}
\def \firstright{(1.22497,-.706886) circle (1cm);}

\begin{pgfonlayer}{background}
	\begin{scope}
		\clip \firstmain;
		\fill[blue!30!white] \firstup;
		\fill[blue!30!white] \firstleft;
		\fill[blue!30!white] \firstright;
	\end{scope}
\end{pgfonlayer}

\begin{pgfonlayer}{background2}
	\draw[thick] (-.607,.707) -- (-.966,.086027);
	\draw[thick] (-.357,.707) -- (-.716,.086027);
	\draw[thick] (-.05, .5) -- (.51,0);
	\draw[thick] (-.5,.601786) -- (.51,-.3);
\end{pgfonlayer}

\begin{pgfonlayer}{background3}
	\fill[blue!30!white] (-.607,.707) -- (-.966,.086027) -- (-.716,.086027) -- (-.357,.707);
	\fill[blue!30!white] (-.05, .5) -- (.51,0) -- (.51,-.3) -- (-.5,.601786);
\end{pgfonlayer}

\begin{pgfonlayer}{background4}
	\draw[thick] (.657,.707) -- (-.25,-.9);
	\draw[thick] (.357,.707) -- (-.45,-.722822);
\end{pgfonlayer}

\begin{pgfonlayer}{background5}
	\fill[blue!30!white] (.657,.707) -- (-.25,-.9) -- (-.45,-.722822) -- (.357,.707);
\end{pgfonlayer}

\draw[dashed] (0,0) circle (1cm);
\draw[thick] (.707, .707) arc (315:225:1cm);
\draw[thick] (-.966,.259) arc (75:-15:1cm);
\draw[thick] (.966, .259) arc (105:195:1cm);

\begin{scope}[xshift = 5cm]

	\def \secondmain{(0,0) circle (1cm);}
	\def \secondup{(0,1.414) circle (1cm);}
	\def \secondleft{(-1.22497, -.706886) circle (1cm);}
	\def \secondright{(1.22497,-.706886) circle (1cm);}
	
	\begin{pgfonlayer}{background}
		\begin{scope}
			\clip \secondmain;
			\fill[blue!30!white] \secondup;
			\fill[blue!30!white] \secondleft;
			\fill[blue!30!white] \secondright;
		\end{scope}
	\end{pgfonlayer}

	\draw[dashed] (0,0) circle (1cm);
	\begin{pgfonlayer}{background}
	\draw[thick] (.707, .707) arc (315:225:1cm);
	\draw[thick] (-.966,.259) arc (75:-15:1cm);
	\draw[thick] (.966, .259) arc (105:195:1cm);
	\end{pgfonlayer}
	
	\begin{pgfonlayer}{background2}
		\begin{scope}
			\clip \secondmain;
			\draw[thick] (1.3,-.2) circle (1.2cm);
			\draw[thick] (1.5,-.2) circle (1.2cm);
		\end{scope}
	\end{pgfonlayer}
	
	\begin{pgfonlayer}{background3}
		\begin{scope}
			\clip \secondmain;
			\fill[white] (1.5,-.2) circle (1.2cm);
		\end{scope}
	\end{pgfonlayer}
	
	\begin{pgfonlayer}{background4}
		\begin{scope}
			\clip \secondmain;
			\fill[blue!30!white] (1.3,-.2) circle (1.2cm);
		\end{scope}
	\end{pgfonlayer}

	\begin{pgfonlayer}{background5}
		\begin{scope}
			\clip \secondmain;
			\draw[thick] (.3,1) circle (1.06cm);
			\draw[thick] (.4,1.1) circle (1cm);
		\end{scope}
	\end{pgfonlayer}
	
	\begin{pgfonlayer}{background6}
		\begin{scope}
			\clip \secondmain;
			\fill[white] (.4,1.1) circle (1cm);
		\end{scope}
	\end{pgfonlayer}
	
	\begin{pgfonlayer}{background7}
		\begin{scope}
			\clip \secondmain;
			\fill[blue!30!white] (.3,1) circle (1.06cm);
		\end{scope}
	\end{pgfonlayer}
	
	\def \whitecirc{(1.5,-.2) circle (1.2cm);}
	\begin{pgfonlayer}{background2}
		\begin{scope}
			\clip \secondmain;
			\clip \whitecirc;
			\draw[thick] (.3,1) circle (1.06cm);
			\draw[thick] (.4,1.1) circle (1cm);
		\end{scope}
	\end{pgfonlayer}
	
	\begin{pgfonlayer}{background2a}
		\begin{scope}
			\clip \secondmain;
			\clip \whitecirc;
			\fill[white] (.4,1.1) circle (1cm);
		\end{scope}
	\end{pgfonlayer}{background2a}
	
	\begin{pgfonlayer}{background2b}
		\begin{scope}
			\clip \secondmain;
			\clip \whitecirc;
			\fill[blue!30!white] (.3,1) circle (1.06cm);
		\end{scope}
	\end{pgfonlayer}
	
	\begin{pgfonlayer}{background2}
		\draw[thick] (.2,.5) -- ( -.3,-.5);
		\draw[thick] (-.1,.5) -- (-.6,-.5);	
	\end{pgfonlayer}
	
	\begin{pgfonlayer}{background2a}
		\fill[blue!30!white] (.2,.5) -- ( -.3,-.5) -- (-.6,-.5) -- (-.1,.5);
	\end{pgfonlayer}

\end{scope}

\draw[very thick, <->, >=triangle 45] (1.5,0) -- (3.5,0);

\end{scope}

\end{tikzpicture}$$
\caption{Depictions of each of the ribbon graph Reidemeister moves on $\Sigma_{\mathbb{G}}$. The Reidemeister II move where one vertex is split into two is shown.}
\label{fig:ribbonReid}
\end{figure}
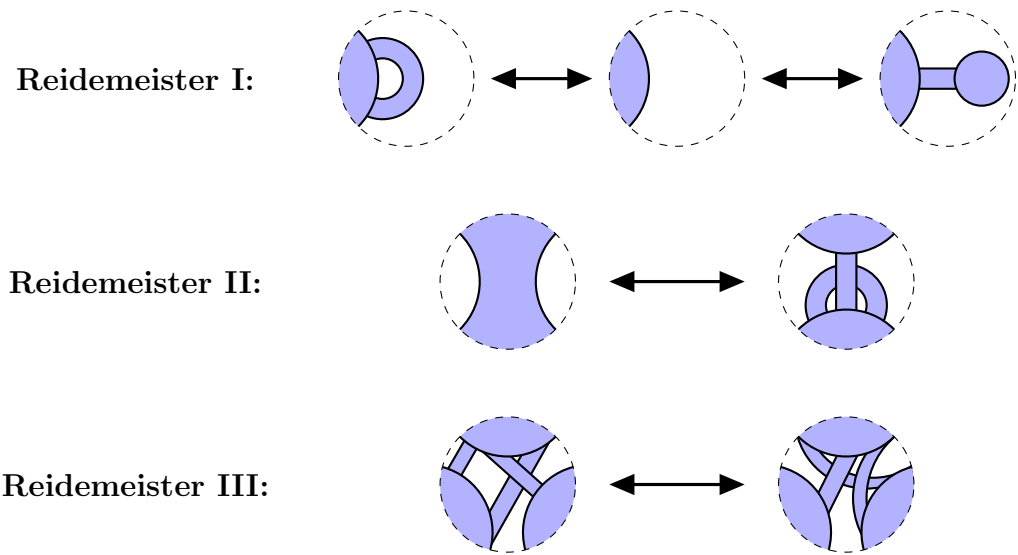

\begin{remark}
\label{prop:orientability}
Let $\mathbb{D}_1$ and $\mathbb{D}_2$ be ribbon graphs such that $\mathbb{D}_1$ can be obtained from $\mathbb{D}_2$ by a single Reidemeister I move or  Reidemeister III move. Then $\mathbb{D}_1$ is orientable if and only if $\mathbb{D}_2$ is orientable. However, since a Reidemeister II move either merges two vertices together or splits them apart, it may possibly create or destroy an embedded M\"obius band, and hence could change an orientable ribbon graph into a non-orientable ribbon graph, or vice versa.
\end{remark}
 A ribbon graph Reidemeister move that takes an oriented ribbon graph to another oriented ribbon graph is called an {\em oriented ribbon graph Reidemeister move.} The ribbon graph Reidemeister moves generalize Reidemeister moves for classical links and also generalize the virtual Reidemeister moves.

\begin{proposition}
\label{prop:reidemeister}
Suppose that $D_1$ and $D_2$ are diagrams of the same (classical) link, and suppose that $\mathbb{D}_1$ and $\mathbb{D}_2$ are their respective all-$A$ ribbon graphs. Then there is a sequence of oriented ribbon graph Reidemeister moves transforming $\mathbb{D}_1$ into $\mathbb{D}_2$.
\end{proposition}
\begin{proof}
Since $D_1$ and $D_2$ are diagrams of the same link, there is a sequence of Reidemeister moves for links transforming $D_1$ into $D_2$. Hence, it suffices to show that each Reidemeister move on a link diagram induces one of the Reidemeister moves on a ribbon graph. Since the all-$A$ ribbon graph of a classical link diagram is always oriented, the ribbon graph Reidemeister moves in our sequence are also oriented. Figure \ref{fig:Reidproof} exhibits how each link Reidemeister move induces a ribbon graph Reidemeister move.
\begin{figure}[h]
$$\begin{tikzpicture}

\draw (0,.4) .. controls (0,1) and (.1,1.2) .. (.4,1.4);
\draw (0,2) .. controls (0,1.5) and (0,1.5) .. (.1,1.4);
\draw (.4,1.4) .. controls (.6,1.5) and (.9,1.5) .. (.9,1.2);
\draw (.4,1.1) .. controls (.6,.9) and (.9,1) .. (.9,1.2);
\draw (.3,1.4) node[above]{$x$};

\draw (2.3,.4) .. controls (2.6,1.2) and (2.6,1.2) .. (2.3,2);

\begin{scope}[xshift = 4cm]
	\draw (0,.4) .. controls (0,.9) and (0,.9) .. (.1,1);
	\draw (0,2) .. controls (0,1.4) and (.1,1.2) .. (.4,1);
	\draw (.4,1) .. controls (.6,.9) and (.9,.9) .. (.9, 1.2);
	\draw (.9,1.2) ..controls (.9,1.4) and (.6,1.5) .. (.4,1.3);
	\draw (.3,1.4) node[above]{$x$};
\end{scope}

\begin{scope}[>=triangle 45, thick]
	\draw[<->] (1.1,1.2) -- (2.1,1.2);
	\draw[<->] (2.7,1.2) -- (3.7,1.2);
\end{scope}

\begin{scope}[yshift = -3cm]

\draw (2.3,.4) .. controls (2.6,1.2) and (2.6,1.2) .. (2.3,2);

\begin{pgfonlayer}{background}
\draw (0,.4) -- (0,.7);
\draw (0,2) -- (0,1.7);
\draw (0,.7) arc (180:75:.3cm);
\draw (0, 1.7) arc (180:285:.3cm);
\draw (.375,.988) arc (-105:105:.22cm);
\end{pgfonlayer}
\draw[very thick, red] (.375,.988) arc (75:130:.3cm);
\draw[very thick, red, ->] (.375,.988) arc (-105:-90:.22cm);
\draw[very thick, red] (.375,1.412) arc (105:90:.22cm);
\draw[very thick, red, ->] (.375,1.412) arc (285:230:.3cm);
\draw (.3,1.5) node[above]{$x$};
\draw (.3,.9) node[below]{$x$};

\begin{pgfonlayer}{background}
\draw (4,.4) -- (4,2);
\draw (4.8,1.2) circle (.3cm);
\end{pgfonlayer}
\draw[red, very thick, ->] (4,1) -- (4,1.4);
\draw[xshift=4.8cm, yshift=1.2cm,red,very thick, ->] (130:.3cm) arc (130:230:.3cm);
\draw (4.8,1.2) node{$x$};
\draw (4.25,1.2) node{$x$};

\begin{scope}[>=triangle 45, thick]
	\draw[<->] (1.1,1.2) -- (2.1,1.2);
	\draw[<->] (2.7,1.2) -- (3.7,1.2);
\end{scope}

\end{scope}


\begin{scope}[xshift=6cm]

	\draw (0,.4) .. controls (.3,1.2) and (.3,1.2) .. (0,2);
	\draw (.8,.4) .. controls (.5,1.2) and (.5,1.2) .. (.8,2);
	
	\draw (2,.4) .. controls (2.8,1.2) and (2.8,1.2) .. (2,2);
	\draw (2.8,.4) .. controls (2.6,.6) and (2.6,.6) .. (2.5,.7);
	\draw (2.8,2) .. controls (2.6,1.8) and (2.6,1.8) .. (2.5,1.7);
	\draw (2.3,.9) .. controls (2.15,1.2) and (2.15,1.2) .. (2.3,1.5);
	\draw (2.45,.7) node[below]{$x$};
	\draw (2.45,1.7) node[above]{$y$};
	
	\draw[thick, >=triangle 45,<->] (1,1.2) -- (1.9,1.2);

\end{scope}

\begin{scope}[xshift=6cm,yshift = -3cm]

	\draw (0,.4) .. controls (.3,1.2) and (.3,1.2) .. (0,2);
	\draw (.8,.4) .. controls (.5,1.2) and (.5,1.2) .. (.8,2);
	
	\begin{pgfonlayer}{background}
		\draw (2.15,.4) arc (-45:45:.4cm);
		\draw (2.65,.4) arc (225:135:.4cm);
		\draw (2.15,.965) arc (225:-45:.355cm);
		\draw (2.1,2) arc (225:315:.428cm);
	\end{pgfonlayer}
	\draw[red,very thick, ->] (2.265,.6825) arc (0:30:.4cm);
	\draw[red,very thick] (2.265,.6825) arc (0:-30:.4cm);
	\draw[red,very thick, ->] (2.535,.6825) arc (180:210:.4cm);
	\draw[red,very thick] (2.535,.6825) arc (180:150:.4cm);
	\draw[red,very thick,->] (2.4,1.57) arc (90:50:.355cm);
	\draw[red,very thick] (2.4,1.57) arc (90:130:.355cm);
	\draw[red, very thick, ->] (2.4,1.875) arc (-90:-120:.355cm);
	\draw[red, very thick] (2.4,1.875) arc (-90:-60:.355cm);
	
	\draw (2.2,.6825) node[left]{$x$};
	\draw (2.6,.6825) node[right]{$x$};
	\draw (2.4,1.57) node[below]{$y$};
	\draw (2.4,1.875) node[above]{$y$};
	
	\draw[thick, >=triangle 45,<->] (1,1.2) -- (1.9,1.2);

\end{scope}



\begin{scope}[xshift = 10cm]

	\draw (0,1.2) .. controls (0.4,1.8) and (1.2,1.8) .. (1.6,1.2);
	\draw (0,.4) -- (1.1,1.5);
	\draw (1.3,1.7) -- (1.6,2);
	\draw (0,2) -- (.3,1.7);
	\draw (.5,1.5) -- (.7,1.3);
	\draw (.9,1.1) -- (1.6,.4);

	\draw (3,1.2) .. controls (3.4,.6) and (4.2,.6) .. (4.6,1.2);
	\draw (3,.4) -- (3.3,.7);
	\draw (4.6,.4) -- (4.3,.7);
	\draw (4.1,.9) -- (3.9,1.1);
	\draw (3.7,1.3) -- (3,2);
	\draw (3.5,.9) -- (4.6,2);
	
	\draw[thick, >=triangle 45, <->] (1.8,1.2) -- (2.8,1.2);
	
	\draw (1.2,1.7) node[above]{$x$};
	\draw (.4,1.65) node[above]{$y$};
	\draw (.8,1.1) node[below]{$z$};
	
	\draw (3.8,1.3)node[above]{$z$};
	\draw (3.4,.7) node[below]{$x$};
	\draw (4.2,.7) node[below]{$y$};

\end{scope}
\begin{scope}[xshift = 10cm,yshift = -3cm]

	\begin{scope}[rounded corners = .2cm]
		\begin{pgfonlayer}{background}
			\draw (0,1.2) -- (.4,1.6) -- (0,2);
			\draw (1.6,0.4) -- (.8,1.2) -- (1.2,1.6) -- (1.6,1.2);
			\draw (1.6,2) -- (1.2,1.6) -- (.8,2) -- (.4,1.6) -- (.8,1.2) -- (0,.4);
			\draw (3,.4) -- (3.4,.8) -- (3.8,.4) -- (4.2,.8) -- (3.8,1.2) -- (4.6,2);
			\draw (3,1.2) -- (3.4,.8) -- (3.8,1.2) -- (3,2);
			\draw (4.6,.4) -- (4.2,.8) -- (4.6,1.2);
		\end{pgfonlayer}
	\begin{scope}[very thick, red, ->]
		\draw (.2,1.4) -- (.4,1.6) -- (.2,1.8);
		\draw (.6,1.8) -- (.4,1.6) -- (.6,1.4);
		\draw (.65,1.35) -- (.8,1.2) -- (.65,1.05);
		\draw (.95,1.05) -- (.8,1.2) -- (.95,1.35);
		\draw (1,1.4) -- (1.2,1.6) -- (1.4,1.4);
		\draw (1.4,1.8) -- (1.2,1.6) -- (1,1.8);
		\draw (3.2,1) -- (3.4,.8) -- (3.6,1);
		\draw (3.65,1.05) -- (3.8,1.2) -- (3.6,1.4);
		\draw (4,1.4) -- (3.8,1.2) -- (4,1);
		\draw (4.05,.95) -- (4.2,.8) -- (4,.6);
		\draw (3.6,.6) -- (3.4,.8) -- (3.2,.6);
		\draw (4.4,.6) -- (4.2,.8) -- (4.4,1);
	\end{scope}
	\end{scope}
	
	\draw (.3,1.6) node[left]{$y$};
	\draw (.5,1.6) node[right]{$y$};
	\draw (.7,1.1)node[left]{$z$};
	\draw (.9,1.1)node[below]{$z$};
	\draw (1.2,1.5) node[below]{$x$};
	\draw (1.2,1.7) node[above]{$x$};
	
	\draw (3.7,1.3) node[above]{$z$};
	\draw (3.9,1.2) node[right]{$z$};
	\draw (3.4,.9) node[above]{$x$};
	\draw (3.4,.7) node[below]{$x$};
	\draw (4.1,.8) node[left]{$y$};
	\draw (4.3,.8) node[right]{$y$};
	
	\draw[thick, >=triangle 45, <->] (1.8,1.2) -- (2.8,1.2);

\end{scope}

\draw[very thick, >=stealth,->] (2.3,.1) -- (2.3,-.7);
\draw[very thick, >=stealth,->] (7.45,.1) -- (7.45,-.7);
\draw[very thick, >=stealth,->] (12.3,.1) -- (12.3,-.7);

\end{tikzpicture}$$
\caption{The top row depicts the three Reidemeister moves. Beneath each Reidemeister move is the corresponding all-$A$ resolution together with the marking arrows for each resolution. The crossings and the marking arrows are labeled either $x$, $y$, or $z$.}
\label{fig:Reidproof}
\end{figure}
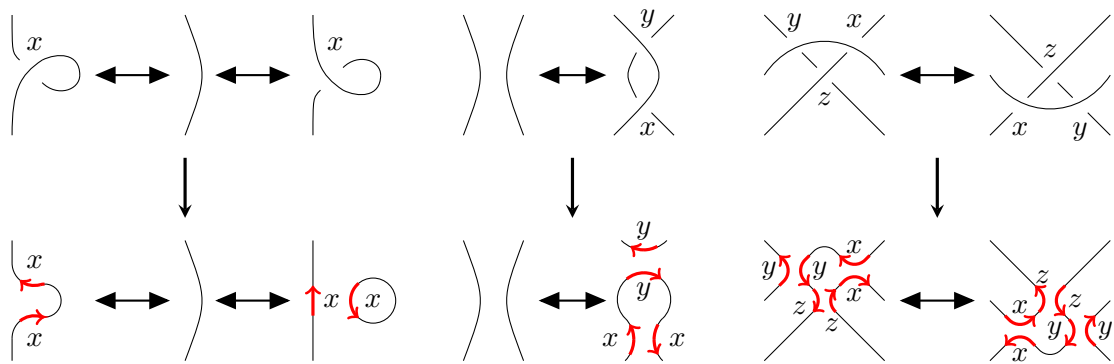
\end{proof}

Work of Viro \cite{Viro:ChordDiagrams} combined with Proposition \ref{prop:alternatable} provide an analog of Proposition \ref{prop:orientability} for virtual links.

\begin{proposition}
\label{prop:virtualreidemeister}
Suppose that $D_1$ and $D_2$ are diagrams of the same virtual link, and suppose that their respective all-$A$ ribbon graphs $\mathbb{D}_1$ and $\mathbb{D}_2$ are oriented. Then there is a sequence of oriented ribbon graph Reidemeister moves transforming $\mathbb{D}_1$ into $\mathbb{D}_2$.
\end{proposition}
\begin{proof} Viro \cite{Viro:ChordDiagrams} proved that there is a sequence of classical and virtual Reidemeister moves transforming $D_1$ into $D_2$ such that each intermediate diagram is also alternatable. Proposition \ref{prop:alternatable} states that a virtual link diagram is alternatable if and only if its all-$A$ ribbon graph is oriented, and hence each virtual link diagram in the sequence between $D_1$ and $D_2$ has an oriented all-$A$ ribbon graph.

It remains to show that each classical and virtual Reidemeister moves on diagrams induce ribbon graph Reidemeister moves. The proof of Proposition \ref{prop:reidemeister} shows that the classical Reidemeister moves induce ribbon graph Reidemeister moves. The virtual Reidemeister moves labeled IV, V, and VI in Figure \ref{fig:virtualRmoves} involve only virtual crossings, and they do not change the all-$A$ ribbon graph at all. Figure \ref{fig:virtReidproof} shows that the virtual Reidemeister move labeled VII in Figure \ref{fig:virtualRmoves} does not change the all-$A$ ribbon graph, and hence the result follows.
\begin{figure}[h]
$$\begin{tikzpicture}
\draw (0,1.2) .. controls (0.4,1.8) and (1.2,1.8) .. (1.6,1.2);
\draw (0,.4) -- (1.6,2);
\draw (1.6,.4) -- (.9,1.1);
\draw (.7,1.3) -- (0,2);
	
\draw (.43,1.57) circle (.1cm);
\draw (1.17, 1.57) circle (.1cm);

\draw (.8,1.1) node[below]{$x$};
\draw (3.8,1.3) node[above]{$x$};

\draw (3,1.2) .. controls (3.4,.6) and (4.2,.6) .. (4.6,1.2);
\draw (3,.4) -- (4.6,2);
\draw (4.6,.4) -- (3.9,1.1);
\draw (3.7,1.3) -- (3,2);
\draw (4.17, .83) circle (.1cm);
\draw (3.43,.83) circle (.1cm);
	
\draw[thick, >=triangle 45, <->] (1.8,1.2) -- (2.8,1.2);

\begin{scope}[rounded corners = .2cm, xshift =8cm]
	\draw (0,1.2) .. controls (0.4,1.8) and (1.2,1.8) .. (1.6,1.2);
	\begin{pgfonlayer}{background}
	\draw (0,.4) -- (.8,1.2) -- (0,2);
	\draw (1.6,.4) -- (.8,1.2) -- (1.6,2);
	\end{pgfonlayer}
	\draw[very thick, red, ->] (.6,1) -- (.8,1.2) -- (.6,1.4);
	\draw[very thick, red, ->] (1,1.4) -- (.8,1.2) -- (1,1);
	
	\draw (.43,1.57) circle (.1cm);
	\draw (1.17, 1.57) circle (.1cm);
	
	\draw (.7,1.2) node[left]{$x$};
	\draw (.9,1.2) node[right]{$x$};

	\draw (3,1.2) .. controls (3.4,.6) and (4.2,.6) .. (4.6,1.2);
	\begin{pgfonlayer}{background}
	\draw (3,.4) -- (3.8,1.2) -- (3,2);
	\draw (4.6,.4) -- (3.8,1.2) -- (4.6,2);
	\end{pgfonlayer}
	\draw[very thick, red, ->] (3.6,1) -- (3.8,1.2) -- (3.6,1.4);
	\draw[very thick, red, ->] (4,1.4) -- (3.8,1.2) -- (4,1);

	\draw (4.17, .83) circle (.1cm);
	\draw (3.43,.83) circle (.1cm);
	
	\draw (3.7,1.2) node[left]{$x$};
	\draw (3.9,1.2) node[right]{$x$};
	
	\draw[thick, >=triangle 45, <->] (1.8,1.2) -- (2.8,1.2);

\end{scope}
\end{tikzpicture}$$
\caption{Virtual Reidemeister VII and its all-$A$ resolutions.}
\label{fig:virtReidproof}
\end{figure}
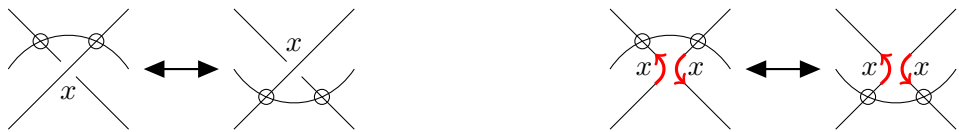
\end{proof}

The main result of this section states that oriented ribbon graph Reidemeister moves induce isomorphisms on Khovanov homology.

\begin{theorem}
\label{thm:reidinv}
Suppose that $\mathbb{G}_1$ and $\mathbb{G}_2$ are ribbon graphs such that $\mathbb{G}_1$ can be transformed into $\mathbb{G}_2$ by a sequence of oriented ribbon graph Reidemeister moves. Then (up to a grading shift), there are isomorphisms
\begin{eqnarray*}
Kh(\mathbb{G}_1) & \cong & Kh(\mathbb{G}_2)~\text{and}\\
\widetilde{Kh}(\mathbb{G}_1) & \cong & \widetilde{Kh}(\mathbb{G}_2).
\end{eqnarray*}
\end{theorem}
\begin{proof}
The proof of this theorem is identical to Bar-Natan's proof \cite{Bar-Natan:Khovanov} of Reidemeister move invariance for Khovanov homology of links. The quasi-isomorphisms between complexes of link diagrams that differ by a single Reidemeister move depend only on the local difference between the state circles in the two different diagrams. Since each ribbon graph Reidemeister move changes the boundary components of the regular neighborhood of the spanning subgraphs in the same way as a Reidemeister move changes the state circles of the Kauffman states, the  Bar-Natan quasi-isomorphisms between complexes generalize naturally to the ribbon graph setting.
\end{proof}

\begin{remark}
The example in Figure \ref{fig:sameribbon} implies that oriented ribbon graphs modulo oriented ribbon graph Reidemeister moves are distinct from virtual link diagrams modulo the classical and virtual Reidemeister moves. The ribbon graph in the example can be transformed into the ribbon graph of the standard diagram of the unknot by a Reidemeister II move followed by two Reidemeister I moves. However, the virtual link diagram on the left of Figure \ref{fig:sameribbon} is not equivalent to the unknot.
\end{remark}

\section{Quasi-tree expansion}
\label{section:quasi}

In this section, we describe a spanning quasi-tree model for ribbon graph homology. Recall that a spanning quasi-tree $\mathbb{H}$ of a connected ribbon graph $\mathbb{G}$ is a spanning ribbon subgraph of $\mathbb{G}$ such that $\Sigma_{\mathbb{H}}$ has one boundary component. 

\subsection{Deletion - ribbon contraction decomposition}

Let $\mathbb{G}$ be a ribbon graph and let $e$ be an edge in $E(\mathbb{G})$. Define $\mathbb{G}\backslash e$ to be the ribbon graph with $e$ removed and the cyclic order of all the other edges preserved. Define the {\em ribbon contraction of $e$}, denoted $\mathbb{G}/ e$, as follows. If $e$ is not a loop, then $\mathbb{G}/ e$ is $\mathbb{G}$ with the edge $e$ contracted and the cyclic ordering of the edges around each vertex induced from $\mathbb{G}$. If $e$ is a loop incident to the vertex $v$, then $\mathbb{G}/ e$ is $\mathbb{G}$ with $e$ deleted and 
the vertex $v$ split into two vertices, with the other edges incident to $v$ as in Figure \ref{fig:ribboncontract}.
A loop $e$ in $\mathbb{G}$ is {\em separating} if $\mathbb{G}/e$ has one more component than $\mathbb{G}$; otherwise, $e$ is {\em nonseparating}.
\begin{figure}[h]
$$\begin{tikzpicture}[thick, >=triangle 45]
\draw (1,1) -- (0.3,0.3) node[left]{$3$};
\draw (1,1) -- (0,1) node[left]{$2$};
\draw (1,1) -- (0.3,1.7) node[left]{$1$};
\draw (1,1) -- (1.7,1.7) node[right]{$5$};
\draw (1,1) -- (1.7,0.3) node[right]{$4$};
\draw (2,1) circle (1cm);
\draw (2.75,1.7) node[right]{$e$};
\fill[black] (1,1) circle (.1cm);
\draw[very thick, ->] (4,1) -- (6,1);
\begin{scope}[xshift = 7cm]
	\draw (1,1) -- (0.3,0.3) node[left]{$3$};
	\draw (1,1) -- (0,1) node[left]{$2$};
	\draw (1,1) -- (0.3,1.7) node[left]{$1$};
	\fill[black] (1,1) circle (.1cm);
\end{scope}
\begin{scope}[xshift = 8cm]
	\fill[black] (1,1) circle (.1cm);
	\draw (1,1) -- (1.7,1.7) node[right]{$5$};
	\draw (1,1) -- (1.7,0.3) node[right]{$4$};
\end{scope}
\end{tikzpicture}$$
\caption{Ribbon contracting the loop $e$ in $\mathbb{G}$ on the left to form the ribbon graph $\mathbb{G}/e$ on the right.}
\label{fig:ribboncontract}
\end{figure}
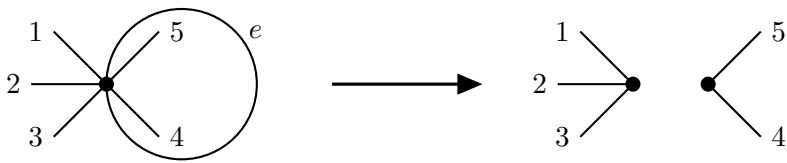

Suppose that $C_0$ and $C_1$ are cochain complexes with differentials $d_0$ and $d_1$ respectively, and that $w:C_0\to C_1$ is a cochain transformation. The mapping cone of $w$ is defined to be the cochain complex $C(w) := C_0 \oplus C_1[1]$ with differential $d_{C(w)}(x,y) := (d_0(x), w(x) - d_1(y))$.

\begin{proposition}
\label{prop:mapping_cone}
Let $\mathbb{G}$ be a ribbon graph with an edge $e$. Then there exists cochain transformations $w$ and $\widetilde{w}$ such that $CKh(\mathbb{G})$ is isomorphic to the mapping cone of $w:CKh(\mathbb{G}\backslash e)\to CKh(\mathbb{G}/e)\{1\}$ and $\widetilde{CKh}(\mathbb{G})$ is isomorphic to the mapping cone of $\widetilde{w}:\widetilde{CKh}(\mathbb{G}\backslash e) \to \widetilde{CKh}(\mathbb{G}/e)\{1\}$.
\end{proposition}
\begin{proof}
The spanning ribbon subgraphs of $\mathbb{G}$ can be partitioned into two sets: ribbon subgraphs not containing the edge $e$ and ribbon subgraphs containing the edge $e$. There is a natural identification between the spanning ribbon subgraphs of $\mathbb{G}\backslash e$ and the spanning ribbon subgraphs of $\mathbb{G}$ not containing the edge $e$ that assigns to a ribbon subgraph $\mathbb{H}\in S(\mathbb{G}\backslash e)$ the spanning ribbon subgraph of $\mathbb{G}$ containing the same edges as $\mathbb{H}$. Similarly, there is a natural identification between the spanning ribbon subgraphs of $\mathbb{G}/ e$ and the spanning ribbon subgraphs of $\mathbb{G}$ that assigns to a ribbon subgraph $\mathbb{H}\in S(\mathbb{G} / e)$ the spanning ribbon subgraph of $\mathbb{G}$ containing all the edges of $\mathbb{H}$ plus the edge $e$. Together, these identifications define a bijection
$$f: S(\mathbb{G}\backslash e)\cup S(\mathbb{G}/e)\to S(\mathbb{G}).$$

It is straightforward to check that $\Sigma_{\mathbb{H}}$ and $\Sigma_{f(\mathbb{H})}$ have the same number of boundary components. Therefore, as modules we have $CKh(\mathbb{G}) = CKh(\mathbb{G}\backslash e)\oplus CKh(\mathbb{G}/ e)[1]\{1\}$, where the shift grading occurs because if $\mathbb{H}\in S(\mathbb{G}/e)$, then $\mathbb{H}$ has one less edge than $f(\mathbb{H})$. 

Let $d_0$ and $d_1$ denote the differentials in the complex $C(\mathbb{G}\backslash e)$ and $C(\mathbb{G}/e)$ respectively. Then $d_0$ and $d_1$ can be represented by square matrices (also called $d_0$ and $d_1$) whose rows and columns are indexed by bases of $CKh(\mathbb{G}\backslash e)$ and $CKh(\mathbb{G}/e)$ respectively. The differential of $C(\mathbb{G})$ is given by the block matrix
$$d=\left(
\begin{array}{c | c}
d_0 & 0\\
\hline
w & -d_1
\end{array}\right),$$
where the first block of rows is indexed by a basis of $\bigoplus_{I \in\mathcal{V}(n)} V(\mathbb{G}(I))$ where $\mathbb{G}(I)$ does not contain $e$, the second block of rows is indexed by $\bigoplus_{I\in\mathcal{V}(n)}V(\mathbb{G}(I))$ where $\mathbb{G}(I)$ contains $e$, and likewise for the two blocks of columns. The bottom right block is $-d_1$ instead of $d_1$ since all edges in the hypercube for $\mathbb{G}$ contain one more edge (the edge $e$) than the corresponding edges in the hypercube of $\mathbb{G}/e$. Therefore $CKh(\mathbb{G})$ is isomorphic to the mapping cone of $w$. The proof for $\widetilde{CKh}(\mathbb{G})$ is similar.
\end{proof}

Let $\mathbb{G}$ be a connected ribbon graph with edges $e_1,\dots, e_n$. We construct a binary tree  $\mathcal{T}(\mathbb{G})$ whose vertices correspond to ribbon graphs recursively constructed from $\mathbb{G}$. By a slight abuse of notation, we label the vertices of $\mathcal{T}(\mathbb{G})$ by their associated ribbon graphs. The root of the tree $\mathcal{T}(\mathbb{G})$ is $\mathbb{G}$. The {\em depth} of a vertex in $\mathcal{T}(\mathbb{G})$ is the length of the path from the vertex to the root. Suppose that the ribbon graph $\mathbb{H}$ is a vertex of depth $k$. If the edge $e_{n-k}$ is either a bridge or a separating loop in $\mathbb{H}$, then $\mathbb{H}$ has only one child vertex and the ribbon graph assigned to that vertex is also $\mathbb{H}$. If the edge $e_{n-k}$ is neither a bridge nor a separating loop, then $\mathbb{H}$ has two children: $\mathbb{H}\setminus e_{n-k}$ and $\mathbb{H}/e_{n-k}$. The leaves of $\mathcal{T}(\mathbb{G})$ are those vertices with depth $n$. Let $\mathcal{L}(\mathbb{G})$ denote the set of leaves of $\mathcal{T}(\mathbb{G})$. Figure \ref{fig:res_tree} depicts an example of the tree $\mathcal{T}(\mathbb{G})$ for the ribbon graph of Figure \ref{fig:fatgraph}.

Define a map $\tau:\mathcal{L}(\mathbb{G})\to S(\mathbb{G})$ from the leaves of $\mathcal{T}(\mathbb{G})$ to the spanning ribbon subgraphs of $\mathbb{G}$ as follows. Suppose $\mathbb{L}$ is a leaf of $\mathcal{T}(\mathbb{G})$. Since the codomain of $\tau$ is the set of spanning ribbon subgraphs of $\mathbb{G}$, it follows that the ribbon graph $\tau(\mathbb{L})$ is determined by its edges. There is a unique path $P$ in $\mathcal{T}(\mathbb{G})$ of length $n$ between $\mathbb{L}$ and the root $\mathbb{G}$. Let $\mathbb{H}$ and $\mathbb{H}^\prime$ be vertices in $P$ such that the depth of $\mathbb{H}$ is $k$ and the depth of $\mathbb{H}^\prime$ is $k+1$. Either the edge $e_{n-k}$ is a bridge or separating loop in $\mathbb{H}$, or $\mathbb{H}^\prime$ is obtained from $\mathbb{H}$ by deleting or ribbon contracting $e_{n-k}$. The edge $e_{n-k}$ is in $\tau(\mathbb{L})$ if and only if either $e_{n-k}$ is a bridge in $\mathbb{H}$ or $\mathbb{H}^\prime$ is obtained from $\mathbb{H}$ by ribbon contracting $e_{n-k}$. The following proposition states that $\tau$ is actually a bijection onto the set $\mathcal{Q}(\mathbb{G})$ of spanning quasi-trees of $\mathbb{G}$.

\begin{proposition}
Let $\mathbb{G}$ be a connected ribbon graph with edges $e_1,\dots, e_n$. The map $\tau:\mathcal{L}(\mathbb{G})\to S(\mathbb{G})$ is a bijection onto the set $\mathcal{Q}(\mathbb{G})$ of spanning quasi-trees of $\mathbb{G}$.
\end{proposition}
\begin{proof}
It is clear from the construction that $\tau$ is injective. It remains to show that the image of $\tau$ is contained in $\mathcal{Q}(\mathbb{G})$ and that $\tau$ is a surjection onto $\mathcal{Q}(\mathbb{G})$.

Let $\mathbb{L}$ be a leaf in $\mathcal{T}(\mathbb{G})$. If each edge of $\tau(\mathbb{L})$ is ribbon contracted, then the resulting ribbon graph is a single vertex. If $\mathbb{H}$ is a connected ribbon graph such that $\Sigma_{\mathbb{H}/e}$ has one boundary component for some edge $e$, then $\Sigma_{\mathbb{H}}$ also has one boundary component. It follows that $\tau(\mathbb{L})$ is a spanning quasi-tree of $\mathbb{G}$, and hence the image of $\tau$ is contained in $\mathcal{Q}(\mathbb{G})$.

Let $\mathbb{T}\in\mathcal{Q}(\mathbb{G})$ be a spanning quasi-tree of $\mathbb{G}$. By way of induction, if $e_n$ is not an edge in $\mathbb{T}$, assume that $\mathbb{T}$ is in the image of $\tau(\mathbb{G}\setminus e_n)$, and if $e_n$ is an edge of $\mathbb{T}$, assume that $\mathbb{T}/e_n$ is in the image of $\tau(\mathbb{G}/e_n)$. The edge $e_n$ is either a bridge, a separating loop, or neither in $\mathbb{G}$. If $e_n$ is a separating loop, then $\mathbb{T}$ is a spanning quasi-tree of $\mathbb{G}\setminus e_n$. If $e_n$ is a bridge, then $\mathbb{T}/e_n$ is a spanning quasi-tree of $\mathbb{G}/e_n$. Suppose $e_n$ is neither a bridge nor a separating loop in $\mathbb{G}$. If $e_n$ is an edge in $\mathbb{T}$, then $\mathbb{T}/e_n$ is a spanning quasi-tree of $\mathbb{G}/e_n$, and if $e_n$ is not an edge in $\mathbb{T}$, then $\mathbb{T}$ is a spanning quasi-tree of $\mathbb{G}\setminus e_n$.  Therefore, $\mathbb{T}$ is in the image $\tau(\mathbb{G})$, and hence $\tau$ is a bijection between the leaves of $\mathcal{T}(\mathbb{G})$ and the spanning quasi-trees of $\mathbb{G}$.
\end{proof}

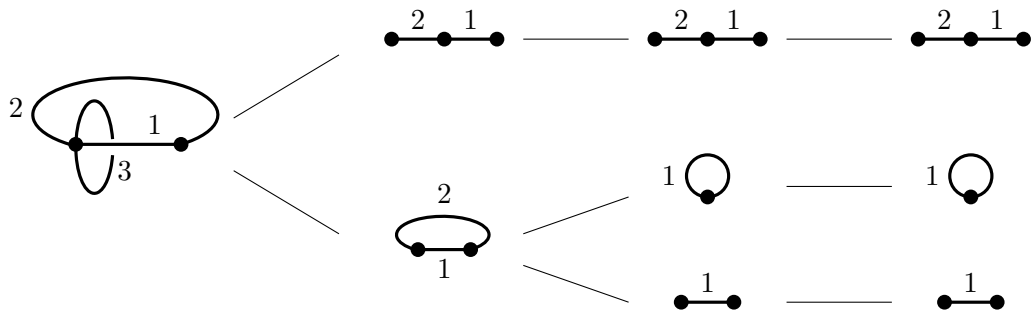
\begin{figure}[h]
$$\begin{tikzpicture}[scale=.7]
	\fill[black] (-6,0) circle (4pt);
	\fill[black] (-8,0) circle (4pt);
	\draw[very thick] (-8,0) -- (-6,0);
	\draw (-6.5,0) node[above]{$1$};
	\draw[very thick] (-7.3,0.1) arc (10:350:10pt and 25pt);
	\draw (-7.4,-.5) node[right]{$3$};
	\draw[very thick] (-6,0) arc (-53:233:50pt and 20pt);
	\draw (-8.8,.7) node[left]{$2$};
	\fill[black] (-2,2) circle (4pt);
	\fill[black] (-1,2) circle (4pt);
	\fill[black] (0,2) circle (4pt);
	\draw[very thick] (-2,2) -- (0,2);
	\draw(-1.5,2) node[above]{$2$};
	\draw(-.5,2) node[above]{$1$};
	\fill[black] (-1.5,-2) circle (4pt);
	\fill[black] (-.5,-2) circle (4pt);
	\draw[very thick] (-1.5,-2) -- (-.5,-2);
	\draw (-1,-2) node[below]{$1$};
	\draw[very thick] (-.5,-2) arc (-53:233:25pt and 10 pt);
	\draw (-1,-1.3) node[above]{$2$};
	\begin{scope}[xshift = 5cm]
		\fill[black] (-2,2) circle (4pt);
		\fill[black] (-1,2) circle (4pt);
		\fill[black] (0,2) circle (4pt);
		\draw[very thick] (-2,2) -- (0,2);
		\draw(-1.5,2) node[above]{$2$};
		\draw(-.5,2) node[above]{$1$};
	\end{scope}
	\fill[black] (4,-1) circle (4pt);
	\draw[very thick] (4,-.6) circle (.4cm);
	\draw (3.6, -.6) node[left]{$1$};
	\fill[black] (3.5,-3) circle (4pt);
	\fill[black] (4.5,-3) circle (4pt);
	\draw[very thick] (3.5,-3) -- (4.5,-3);
	\draw (4,-3) node[above]{$1$};
	\begin{scope}[xshift = 10cm]
		\fill[black] (-2,2) circle (4pt);
		\fill[black] (-1,2) circle (4pt);
		\fill[black] (0,2) circle (4pt);
		\draw[very thick] (-2,2) -- (0,2);
		\draw(-1.5,2) node[above]{$2$};
		\draw(-.5,2) node[above]{$1$};
	\end{scope}
	\begin{scope}[xshift = 5cm]
		\fill[black] (4,-1) circle (4pt);
		\draw[very thick] (4,-.6) circle (.4cm);
		\draw (3.6, -.6) node[left]{$1$};
	\end{scope}
	\begin{scope}[xshift = 5cm]
		\fill[black] (3.5,-3) circle (4pt);
		\fill[black] (4.5,-3) circle (4pt);
		\draw[very thick] (3.5,-3) -- (4.5,-3);
		\draw (4,-3) node[above]{$1$};
	\end{scope}
	\draw (-5,.5) -- (-3,1.7);
	\draw (-5,-.5) -- (-3, -1.7);
	\draw (.5,2) -- (2.5,2);
	\draw (5.5,2) -- (7.5,2);
	\draw (5.5,-.8) -- (7.5,-.8);
	\draw (5.5, -3) -- (7.5,-3);
	\draw (.5, -1.7) -- ( 2.5,-1);
	\draw (.5, -2.3) -- (2.5, -3);
\end{tikzpicture}$$
\caption{The tree $\mathcal{T}(\mathbb{G})$ for the ribbon graph of Figure \ref{fig:fatgraph}.}
\label{fig:res_tree}
\end{figure}

\subsection{Quasi-tree expansion}

In this section, we describe a decomposition $CKh(\mathbb{G})\cong A\oplus B$ where $B$ is a contractible complex and $A$ is a complex generated by the spanning quasi-trees of $\mathbb{G}$. Our approach is a modification of Wehrli's method for proving a similar result in Khovanov homology for links \cite{Wehrli:SpanningTrees}. 

A knot diagram that can be transformed into the crossingless diagram of the unknot via a sequence of Reidemeister I moves is called a {\em twisted unknot}.

\begin{proposition}
\label{prop:twisted_unknot}
Let $\mathbb{L}$ be a leaf in the resolution tree $\mathcal{T}(\mathbb{G})$ of a connected ribbon graph $\mathbb{G}$. Then $\mathbb{L}$ is the all-$A$ ribbon graph of a twisted unknot.
\end{proposition}
\begin{proof}
Each edge of $\mathbb{L}$ is either a bridge or a separating loop. A twisted unknot can be constructed from $\mathbb{L}$ as follows. Replace each vertex of $\mathbb{L}$ with a circle. Replace each non-loop edge $e$ of $\mathbb{L}$ with a blue arc that connects the two circles corresponding to the endpoints of the edge $e$. The arc should lie in the exterior of the circles. Replace each loop $e$ at a vertex $v$ of $\mathbb{L}$ with a red arc whose interior lies inside the circle corresponding to $v$ and whose endpoints lie on the circle corresponding to $v$. Furthermore, arrange the red and blue arcs so that the cyclic order of the arcs around any circle is the same as the cyclic order of the half edges around the corresponding vertex in $\mathbb{L}$. 

Since $\mathbb{L}$ is a tree with loops, it follows that the blue arcs can be embedded into the plane so that they are pairwise nonintersecting. The red arcs are also pairwise nonintersecting since they correspond to separating loops.

The diagram of circles and blue and red arcs can be modified to obtained a twisted unknot. Replace each blue arc and neighborhood of the endpoints of the arc in the circles with a crossing whose $A$-resolution is locally a diagram of a neighborhood of the endpoints of the blue arc in the circles. Similarly, replace each red arc and a neighborhood of its endpoints in the circle with a crossing whose $B$-resolution results is locally a diagram of a neighborhood of the endpoints of the red arc in the circle. The resulting diagram is a twisted unknot.
\end{proof}

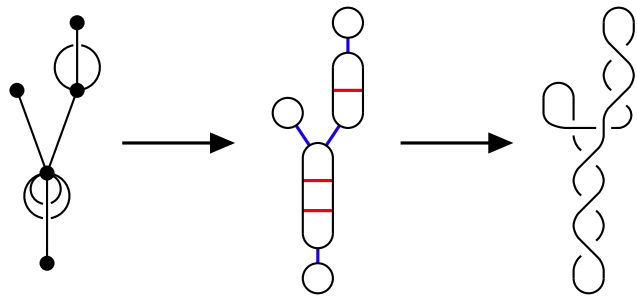
\begin{figure}[h]
$$\begin{tikzpicture}[rounded corners = 1mm,thick,rotate= 90,>=triangle 45]
\draw[->, very thick] (1.8,7) -- (1.8,5.5);
\draw[->, very thick] (1.8,3.3) -- (1.8,1.8);
\draw (0,1) arc (90:270:.2cm);
\draw (0,0.6) -- (0.2,0.6) -- (0.6,1)--(0.8,1) -- (1.2,.6) -- (1.4,.6) -- (1.5,.7);
\draw (0,1) -- (0.2,1) -- (0.3,0.9);
\draw (0.5,0.7) -- (0.6,0.6) -- (0.8,0.6) -- (0.9,0.7);
\draw (1.1,0.9) -- (1.2,1) -- (1.4,1) -- (1.8,.6) -- (2.2,.6) --(2.6,.2) -- (2.8, .2) -- (3.2,.6)--(3.4,.6);
\draw (3.4,0.2) arc (-90:90:.2cm);
\draw (3.4,0.2) -- (3.2,0.2) -- (3.1,0.3);
\draw (2.9,0.5) -- (2.8,0.6) -- (2.6,0.6) -- (2.5,0.5);
\draw (2.3,0.3) -- (2.2,0.2) -- (2,.3) -- (2,.5);
\draw (1.7,.9) -- (1.8,1) --(1.9,1);
\draw (2,.7) -- (2,1.2) -- (2.1,1.4) -- (2.4,1.4);
\draw (2.4,1)--(2.1,1);
\draw (2.4,1) arc (-90:90:.2cm);
\begin{scope}[yshift = 2cm]
\draw (3.4,2) circle (.2cm);
\draw (2.8,1.8) arc (-90:90:.2cm);
\draw (2.2,1.8) -- (2.8,1.8);
\draw (2.2,2.2) -- (2.8, 2.2);
\draw (2.2,2.2) arc (90:270:.2cm);
\draw (2.2,2.8) circle (.2cm);
\draw (1.6,2.2) arc (-90:90:.2cm);
\draw (0,2.4) circle (.2cm);
\draw (0.6,2.6) arc (90:270:.2cm);
\draw (0.6,2.6) -- (1.6,2.6);
\draw (0.6,2.2) -- (1.6,2.2);
\begin{pgfonlayer}{background}
	\draw[very thick, blue!90!red] (0.2,2.4) -- (0.4,2.4);
	\draw[very thick, red!90!blue] (0.9,2.6) -- (0.9,2.2);
	\draw[very thick, red!90!blue] (1.3,2.6) -- (1.3,2.2);
	\draw[very thick, red!90!blue] (2.5,1.8) -- (2.5,2.2);
	\draw[very thick, blue!90!red] (3.2,2) -- (3,2);
	\fill[white] (2.2,2.8) circle (.2cm);
	\fill[white] (1.6,2.4) circle (.2cm);
	\fill[white] (2.2,2) circle (.2cm);
\end{pgfonlayer}
\begin{pgfonlayer}{background2}
	\draw[very thick, blue!90!red] (2.2,2.8) -- (1.6,2.4);
	\draw[very thick, blue!90!red] (2.2,2) -- (1.6,2.4);
\end{pgfonlayer}
\end{scope}
\begin{scope}[yshift = 4cm]
\fill[black] (3.4,3.6) circle (.1cm);
\fill[black] (2.5,3.6) circle (.1cm);
\fill[black] (2.5, 4.4) circle (.1cm);
\fill[black] (1.4,4) circle (.1 cm);
\fill[black] (.2,4) circle (.1 cm);
\draw (.2,4) -- (1.4,4);
\draw (1.4,4) -- (2.5,4.4);
\draw (1.4,4) -- (2.5,3.6);
\draw (2.5,3.6) -- (3.4,3.6);
\draw (3.1,3.65) arc (10:350:.3cm);
\draw (1, 3.95) arc (-160:170:.2cm);
\draw (.8, 3.95) arc (-170:170:.3cm);
\end{scope}
\end{tikzpicture}$$
\caption{An example showing the process described in the proof of Proposition \ref{prop:twisted_unknot}.}
\end{figure}

Recall that $V=\mathbb{Z}\oplus\mathbb{Z}$ where the two summands have bigradings $(0,-1)$ and $(0,1)$ and $\widetilde{V}=\mathbb{Z}$ supported in bigrading $(0,0)$ . One can consider both $V$ and $\widetilde{V}$ to be a bigraded cochain complexes with zero differential.

Since each leaf of $\mathcal{T}(\mathbb{G})$ is the all-$A$ ribbon graph of a twisted unknot, one can easily compute the Khovanov homology of the leaves. If $\mathbb{L}$ is a leaf of $\mathcal{T}(\mathbb{G})$, then let $\text{Bridge}(\mathbb{L})$ be the number of edges which are bridges in $\mathbb{L}$ and let $\text{Loop}(\mathbb{L})$ be the number of edges which are loops in $\mathbb{L}$.

\begin{lemma}
\label{lemma:contractible}
Let $\mathbb{L}$ be a leaf in the resolution tree $\mathcal{T}(\mathbb{G})$. The complex $CKh(\mathbb{L})$ is isomorphic to a complex 
$$V[\text{Loop}(\mathbb{L})]\{2~\text{Loop}(\mathbb{L}) - \text{Bridge}(\mathbb{L})\}\oplus\mathcal{B},$$
and the complex $\widetilde{CKh}(\mathbb{L})$ is isomorphic to a complex
$$\widetilde{V}[\text{Loop}(\mathbb{L})]\{2~\text{Loop}(\mathbb{L})  - \text{Bridge}(\mathbb{L})\}\oplus\widetilde{\mathcal{B}},$$
where the complexes $\mathcal{B}$ and $\widetilde{\mathcal{B}}$ are contractible.
\end{lemma}
\begin{proof}
Since $\mathbb{L}$ is a leaf of the partial resolution tree, it is the all-$A$ ribbon graph of a twisted unknot $T$. The Khovanov complex of a twisted unknot $CKh(T)$ is isomorphic to a contractible complex direct sum with $V$ and the reduced Khovanov complex $\widetilde{CKh}(T)$ is isomorphic to a contractible complex direct sum with $\widetilde{V}$ \cite{Khovanov:homology}. Theorem \ref{theorem:ribbontokh} implies that $CKh(\mathbb{L}) = V[n_-]\{2n_- - n_+\}\oplus \mathcal{B}$ and that $\widetilde{CKh}(\mathbb{L}) =\widetilde{V}[n_-]\{2n_- - n_+\}\oplus\widetilde{\mathcal{B}}$, where $\mathcal{B}$ and $\widetilde{\mathcal{B}}$ are contractible.
The result follows from the fact that negative crossings in $T$ correspond to loops in $\mathbb{L}$ and positive crossings in $T$ correspond to bridges in $\mathbb{L}$.\end{proof}

Wehrli \cite{Wehrli:SpanningTrees} proves the following lemma
\begin{lemma}
\label{lemma:cone_contract}
Let $C_0$ and $C_1$ be two complexes with $C_i=A_i\oplus B_i$ for complexes $A_i$ and $B_i$ with $B_i$ contractible. Let $w:C_0\to C_1$ be a grading preserving chain transformation and let $w_{AA}:A_0\to A_1$ denote $w$ composed with the obvious projection and inclusion. Let $A$ be the mapping cone of $w_{AA}$, let $B$ be the contractible complex $B_0\oplus B_1[1]$, and let $C$ be the mapping cone of $w$. Then $C$ is isomorphic to $A\oplus B$.
\end{lemma}

Let $\text{Con}(\mathbb{L})$ denote the number of edges that are ribbon contracted from $\mathbb{G}$ in order to obtain $\mathbb{L}$. Proposition \ref{prop:mapping_cone}, Lemma \ref{lemma:contractible} and Lemma \ref{lemma:cone_contract} imply that ribbon graph homology has the following spanning quasi-tree expansion.

\begin{theorem}
\label{theorem:quasi-tree_expansion}
Let $\mathbb{G}$ be a ribbon graph, and let $\mathcal{L}(\mathbb{G})$ be the set of leaves of the resolution tree $\mathcal{T}(\mathbb{G})$ of $\mathbb{G}$. There is an isomorphism of complexes $CKh(\mathbb{G})\cong A\oplus B$, where $B$ is contractible and as a bigraded module, $A$ is given by
$$A = \bigoplus_{\mathbb{L}\in\mathcal{L}(\mathbb{G})}
V[\text{Loop}(\mathbb{L})+\text{Con}(\mathbb{L})]\{2~\text{Loop}(\mathbb{L}) - \text{Bridge}(\mathbb{L}) + \text{Con}(\mathbb{L})\}.$$
Similarly, there is an isomorphism of complexes $\widetilde{CKh}(\mathbb{G})\cong\widetilde{A}\oplus\widetilde{B}$, where $B$ is contractible and as a bigraded module, $\widetilde{A}$ is given by
$$\widetilde{A} = \bigoplus_{\mathbb{L}\in\mathcal{T}(\mathbb{G})}
\widetilde{V}[\text{Loop}(\mathbb{L})+\text{Con}(\mathbb{L})]\{2~\text{Loop}(\mathbb{L}) - \text{Bridge}(\mathbb{L}) + \text{Con}(\mathbb{L})\}.$$

\end{theorem}
 
Theorem \ref{theorem:maintheorem2} is a consequence of Theorem \ref{theorem:quasi-tree_expansion}, where the complex $\widetilde{C}(\mathbb{G})$ of Theorem \ref{theorem:maintheorem2} is the complex $\widetilde{A}$ from the previous theorem. The homology of $A$ is $Kh(\mathbb{G})$, and the homology of $\widetilde{A}$ is $\widetilde{Kh}(\mathbb{G})$. Moreover, the generators of $\widetilde{A}$ are in one-to-one correspondence with the spanning quasi-trees of $\mathbb{G}$. If $\mathbb{T}$ is a spanning quasi-tree whose corresponding leaf in the resolution tree is $\mathbb{L}$, then $\mathbb{T}$ may be considered as an element of a basis of $A$ with bigrading $(i(\mathbb{T}), j(\mathbb{T}))$ where
\begin{eqnarray}
\label{eqn:grading1}
i(\mathbb{T}) & = & \text{Loop}(\mathbb{L}) + \text{Con}(\mathbb{L})~\text{and}\\
\label{eqn:grading2}
j(\mathbb{T}) & = & 2~\text{Loop}(\mathbb{L}) -\text{Bridge}(\mathbb{L}) +\text{Con}(\mathbb{L}).
\end{eqnarray}
Similarly, the generators of $A$ are in two-to-one correspondence with spanning quasi-trees of $\mathbb{G}$. For each spanning quasi-tree $\mathbb{T}$ there are two generators $\mathbb{T}_\pm$ of $A$ where the bigrading of $\mathbb{T}_\pm$ is $(i(\mathbb{T}_\pm),j(\mathbb{T}_\pm))$ where
\begin{eqnarray*}
i(\mathbb{T_\pm}) & = & \text{Loop}(\mathbb{L}) + \text{Con}(\mathbb{L})~\text{and}\\
j(\mathbb{T_\pm}) & = & 2~\text{Loop}(\mathbb{L}) -\text{Bridge}(\mathbb{L}) +\text{Con}(\mathbb{L})\pm 1.
\end{eqnarray*}

\subsection{Quasi-tree activities}
\label{subsec:activities}

Tutte \cite{Tutte:ChromaticPoly, Tutte:GraphPoly} defined activities for spanning trees of a graph, which can be used to express the Tutte polynomial of that graph. Thistlethwaite \cite{Thistlethwaite:SpanningTreeExpansion} used activity words for spanning trees of the checkerboard graph of a link diagram to express the Jones polynomial. Champanerkar and Kofman \cite{ChampanerkarKofman:SpanningTrees} showed that Thistlethwaite's activity words give gradings on a spanning tree complex for Khovanov homology. Champanerkar, Kofman, and Stoltzfus \cite{CKS:Quasi-trees} extended the idea of activities to spanning quasi-trees of a ribbon graph and used them to express the Bollobas-Riordan-Tutte polynomial. In this section, we show the gradings of the quasi-tree expansion can be expressed in terms of ribbon graph activity words.

Let $\mathbb{G}$ be a ribbon graph. Each spanning quasi-tree $\mathbb{T}$ of $\mathbb{G}$ has an associated chord diagram $C(\mathbb{G},\mathbb{T})$ constructed as follows. Suppose that the edges of $\mathbb{G}$ are $e_1,\dots, e_n$ where $e_i < e_j$ if and only if $i <  j$. The surface $\Sigma_{\mathbb{T}}$ is a collection of disks, which correspond to the vertices of $\mathbb{G}$, and $2$-dimensional one-handles attached to the disks, which correspond to the edges of $\mathbb{G}$. Each one-handle contributes two segments $S^0\times [0,1]$ to the boundary of $\Sigma_{\mathbb{T}}$. Label the two points $S^0\times \{\frac{1}{2}\}$ with the label on the corresponding edge of $\mathbb{G}$. Also, label points on the boundary of the disks of $\Sigma_{\mathbb{T}}$ where one-handles would be attached for edges in $\mathbb{G}$ but not $\mathbb{T}$ by the label on the corresponding edge. Since $\Sigma_{\mathbb{T}}$ is a quasi-tree, it has one boundary component. 

The orientation of the boundary of $\Sigma_{\mathbb{T}}$ induces a cyclic ordering of the $2n$ labeled points on the boundary of $\Sigma_{\mathbb{T}}$. The chord diagram $C(\mathbb{G},\mathbb{T})$ is formed by taking a circle with $2n$ marked points, labeling those points according to the cyclic ordering given by the boundary of $\Sigma_{\mathbb{T}}$, and connecting two marked points if they have the same label. An example of this construction is shown in Figure \ref{fig:chord_diagram}.

\begin{figure}[h]
$$\begin{tikzpicture}[scale=.85]
\draw (-1.5,0) node[above]{$\mathbb{T}_3$};
\fill[blue!30!white] (0,0) circle (10pt);
\fill[blue!30!white] (2,0) circle (10pt);

\begin{pgfonlayer}{background}
\fill[blue!30!white] (0, 0.2) rectangle (2,-0.2);
\end{pgfonlayer}

\begin{pgfonlayer}{background2}
\def \firstellipse {(0.7,0) ellipse (15pt and 30pt)};
\def \secondellipse {(0.3,0) ellipse (15pt and 30pt)};
\fill[blue!90!red] \firstellipse;
\fill[blue!30] \secondellipse;
\begin{scope}
      \clip \firstellipse;
      \fill[white] \secondellipse;
\end{scope}
\end{pgfonlayer}
\begin{pgfonlayer}{background3}
\def \firstrectangle {(0.35,1.05) rectangle (0.65,-1.05)};
\fill[blue!30!] \firstrectangle;
\end{pgfonlayer}
\begin{pgfonlayer}{background4}
\def \firstarc {(2.3, .2) arc (-52:233:60pt and 25pt)};
\def \secondarc {(2.3, -.2) arc (-52:233:60pt and 25pt)};
\fill[blue!90!red] \firstarc;
\fill[blue!30!]\secondarc;
\begin{scope}
      	\clip \firstarc;
      	\fill[white] \secondarc;
\end{scope}
\end{pgfonlayer}
\begin{pgfonlayer}{background5}
\def \firstrectangle {(-1.1,.55) rectangle (3.1,.85)};
\fill[blue!30!] \firstrectangle;
\end{pgfonlayer}
\draw (-.1,.7) node[left]{$3$};
\draw (.2,-.6) node[right]{$3$};
\draw (1.5,.2) node[above]{$1$};
\draw (1.5,-.2) node[below]{$1$};
\draw (1.4,1.4) node[below]{$2$};
\draw (1.4,1.7) node[above]{$2$};
\begin{scope}[xshift  = 1cm]
\def \chorddiagram{(10,.5) circle (1.5cm)};
\draw[very thick] \chorddiagram;
\begin{pgfonlayer}{background}
	\draw[thick] (10,2) -- (10,-1);
	\draw[thick] (8.57,1) -- (11.43,1);
	\draw[thick] (8.57,0) -- (11.43,0);
\end{pgfonlayer}
\draw (10,2) node[above]{$3$};
\draw (10,-1) node[below]{$3$};
\draw (8.5,1) node[left]{$1$};
\draw (11.5,1) node[right]{$1$};
\draw (8.5,0) node[left]{$2$};
\draw (11.5,0) node[right]{$2$};
\end{scope}
\begin{scope}[yshift = 4cm]
\draw (-1.5,0) node[above]{$\mathbb{T}_2$};
\fill[blue!30!white] (0,0) circle (10pt);
\fill[blue!30!white] (2,0) circle (10pt);
\begin{pgfonlayer}{background4}
\def \firstarc {(2.3, .2) arc (-52:233:60pt and 25pt)};
\def \secondarc {(2.3, -.2) arc (-52:233:60pt and 25pt)};
\fill[blue!90!red] \firstarc;
\fill[blue!30!]\secondarc;
\begin{scope}
      	\clip \firstarc;
      	\fill[white] \secondarc;
\end{scope}
\fill[white](2.3,-.22) rectangle (-.3,.22);
\end{pgfonlayer}
\begin{pgfonlayer}{background5}
\def \firstrectangle {(-1.1,.55) rectangle (3.1,.85)};
\fill[blue!30!] \firstrectangle;
\end{pgfonlayer}
\begin{pgfonlayer}{background}
	\draw[thick, dashed] (.2,0) -- (1.8,0);
	\draw[thick, dashed] (.5,0) ellipse (.5cm and 1cm);
\end{pgfonlayer}
\draw (1.4,1.4) node[below]{$2$};
\draw (1.4,1.7) node[above]{$2$};
\draw (.3,.2) node[right]{$1$};
\draw (1.7,.2) node[left]{$1$};
\begin{scope}[xshift  = 1cm]
\def \chorddiagram{(10,.5) circle (1.5cm)};
\draw[very thick] \chorddiagram;
\begin{pgfonlayer}{background}
	\draw[thick] (10,2) -- (10,-1);
	\draw[thick] (8.57,1) -- (11.43,1);
	\draw[thick] (8.57,0) -- (11.43,0);
\end{pgfonlayer}
\draw (10,2) node[above]{$1$};
\draw (10,-1) node[below]{$1$};
\draw (8.5,1) node[left]{$2$};
\draw (11.5,1) node[right]{$2$};
\draw (8.5,0) node[left]{$3$};
\draw (11.5,0) node[right]{$3$};
\end{scope}
\draw (0.1,.6) node[left]{$3$};
\draw (0.1, -.6) node[left]{$3$};
\end{scope}
\begin{scope}[yshift = 8cm]
\draw (-1.5,0) node[above]{$\mathbb{T}_1$};
\fill[blue!30!white] (0,0) circle (10pt);
\fill[blue!30!white] (2,0) circle (10pt);

\begin{pgfonlayer}{background}
\fill[blue!30!white] (0, 0.2) rectangle (2,-0.2);
\end{pgfonlayer}
\begin{pgfonlayer}{background2}
\draw[thick,dashed] (.5,0) ellipse (.5cm and 1cm);
\draw[thick, dashed] (2.3, 0) arc (-52:233:60pt and 25pt);
\end{pgfonlayer}
\draw (1.5,.2) node[above]{$1$};
\draw (1.5,-.2) node[below]{$1$};
\draw (-.2,-.2) node[left]{$2$};
\draw (2.3,-.2) node[right]{$2$};
\draw (0.1,.6) node[left]{$3$};
\draw (0.1, -.6) node[left]{$3$};
\begin{scope}[xshift  = 1cm]
\def \chorddiagram{(10,.5) circle (1.5cm)};
\draw[very thick] \chorddiagram;
\begin{pgfonlayer}{background}
	\draw[thick] (10,2) -- (10,-1);
	\draw[thick] (8.57,1) -- (11.43,1);
	\draw[thick] (8.57,0) -- (11.43,0);
\end{pgfonlayer}
\draw (10,2) node[above]{$2$};
\draw (10,-1) node[below]{$2$};
\draw (8.5,1) node[left]{$3$};
\draw (11.5,1) node[right]{$3$};
\draw (8.5,0) node[left]{$1$};
\draw (11.5,0) node[right]{$1$};
\end{scope}
\end{scope}
\end{tikzpicture}$$
\caption{The three spanning quasi-trees $\mathbb{T}_1$, $\mathbb{T}_2$, and $\mathbb{T}_3$ of $\mathbb{G}$ and their corresponding chord diagrams.}
\label{fig:chord_diagram}
\end{figure}
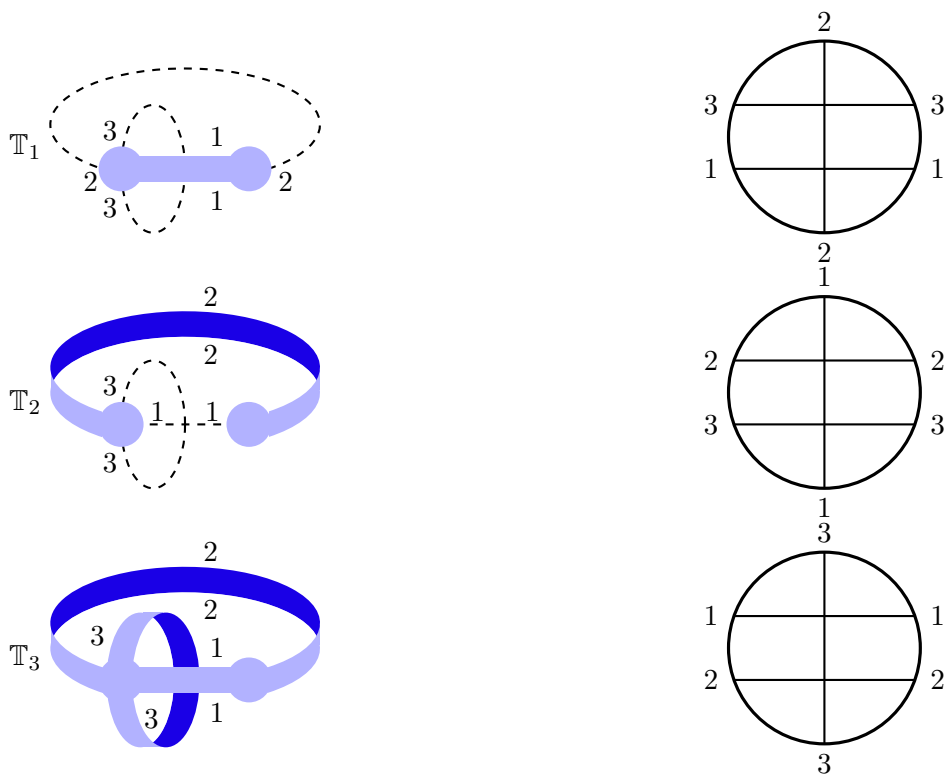

Each edge in $\mathbb{G}$ can be described as either internally active, internally inactive, externally active, or externally inactive with respect to $\mathbb{T}$. An edge $e_i$ is internal with respect to $\mathbb{T}$ if $e_i\in E(\mathbb{T})$; otherwise, the edge $e_i$ is external with respect to $\mathbb{T}$. An edge $e_i$ is active with respect to $\mathbb{T}$ if the chord $c_i$ in $C(\mathbb{G},\mathbb{T})$  corresponding to $e_i$ in $C$ does not intersect any chord $c_j$ corresponding to the edge $e_j$ where $e_j<e_i$; otherwise, the edge $e_i$ is inactive with respect to $\mathbb{T}$. Let $ia(\mathbb{T})$ and $ea(\mathbb{T})$ denote the number of internally active and the number of externally active edges in $\mathbb{G}$ with respect to $\mathbb{T}$ respectively.

Let $\mathbb{G}$ be the ribbon graph in Figure \ref{fig:ribbonconstruct}, and let $\mathbb{T}_1$, $\mathbb{T}_2$, and $\mathbb{T}_3$ be its three spanning quasi-trees. Figure \ref{fig:chord_diagram} shows the chord diagrams for each of the spanning quasi-trees. We have $ia(\mathbb{T}_1) = 1, ea(\mathbb{T}_1)=0,ia(\mathbb{T}_2)=0,ea(\mathbb{T}_2)=1,ia(\mathbb{T}_3)=2$, and $ea(\mathbb{T}_3)=0$.

\begin{proposition}
\label{prop:chord_diagram}
Let $\mathbb{G}$ be a ribbon graph with spanning quasi-tree $\mathbb{T}$ and associated chord diagram $C(\mathbb{G},\mathbb{T})$. Let $e$ be an edge of $\mathbb{G}$.
\begin{enumerate}

\item If $e$ is an edge in $\mathbb{T}$, then the chord diagram $C(\mathbb{G}/e,\mathbb{T}/e)$ can be obtained from $C(\mathbb{G},\mathbb{T})$ by deleting the chord associated to the edge $e$.

\item If $e$ is an edge in $\mathbb{G}$ but not $\mathbb{T}$, then the chord diagram $C(\mathbb{G}\setminus e, \mathbb{T})$ can be obtained from $C(\mathbb{G},\mathbb{T})$ by deleting the chord associated to the edge $e$.

\end{enumerate}
\end{proposition}
\begin{proof}
Suppose that $e$ is an edge in $\mathbb{T}$. The cyclic order of the labeled points on the boundary of $\Sigma_{\mathbb{T}/e}$ is inherited from the cyclic order of the labeled points on the boundary of $\Sigma_{\mathbb{T}}$. Now suppose that $e$ is an edge in $\mathbb{G}$ but not in $\mathbb{T}$. If one considers $\mathbb{T}$ as a spanning quasi-tree of $\mathbb{G}$, then the surface $\Sigma_{\mathbb{T}}$ has $2n$ labeled points on its boundary, and if one considers $\mathbb{T}$ as a spanning quasi-tree of $\mathbb{G}\setminus e$, then the surface $\Sigma_{\mathbb{T}}$ has $2n-2$ labeled points on its boundary. The labeled points when $\mathbb{T}$ is a spanning quasi-tree of $\mathbb{G}/e$ are obtained by deleting the $2$ labeled points corresponding to $e$ on the boundary of $\Sigma_{\mathbb{T}}$ when $\mathbb{T}$ is considered as a spanning quasi-tree of $\mathbb{G}$.
\end{proof}

\begin{proposition}
\label{prop:bridge}
Let $\mathbb{T}$ be a spanning quasi-tree of a ribbon graph $\mathbb{G}$. A chord $c$ in $C(\mathbb{G},\mathbb{T})$ does not intersect any other chords if and only if $c$ corresponds to an edge $e$ in $\mathbb{G}$ that is either a bridge or a separating loop in $\mathbb{G}$.
\end{proposition}
\begin{proof}
If $c$ corresponds to an edge $e$ that is a bridge in $\mathbb{G}$, then $e$ is an edge of $\mathbb{T}$, and if $c$ corresponds to an edge $e$ that is a separating loop in $\mathbb{G}$, then $e$ is not an edge of $\mathbb{T}$. Suppose that the edges of $\mathbb{G}$ are labeled $1,\dots,n$, and that the label of the edge $e$ is $k$. If $e$ is either a bridge or a separating loop in $\mathbb{G}$, then the cyclic order of the labels given by the boundary of $\Sigma_{\mathbb{T}}$ is of the form $k,a_1,\dots,a_p,k,b_1,\dots b_q$ where $a_i\neq b_j$ for $1\leq i\leq p$ and $1\leq j \leq q$. Therefore the chord $c$ does not intersect any other chords in $C(\mathbb{G},\mathbb{T})$.

Suppose that $c$ intersects another chord in $C(\mathbb{G},\mathbb{T})$. We will show that the chord $c$ corresponds to an edge $e$ that is not a bridge or separating loop via induction on the number of chords in $C(\mathbb{G},\mathbb{T})$. For the base case, suppose that $C(\mathbb{G},\mathbb{T})$ has $2$ chords, i.e. $\mathbb{G}$ has $2$ edges. Then $\mathbb{T}$ is a spanning quasi-tree of one of the five ribbon graphs in Figure \ref{fig:2edges}, and a straightforward check shows that the result holds.
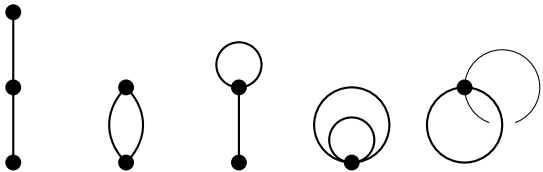
\begin{figure}[h]
$$\begin{tikzpicture}[thick]
\fill[black] (0,0) circle (3pt);
\fill[black] (0,1) circle (3pt);
\fill[black] (0,2) circle (3pt);
\draw (0,0) -- (0,2);
\begin{scope}[yshift = -.5cm]
	\fill[black] (1.5,.5) circle (3pt);
	\fill[black](1.5,1.5) circle (3pt);
	\draw (1.5,.5) .. controls (1.2,.8) and (1.2,1.2).. (1.5,1.5);
	\draw (1.5,.5) .. controls (1.8,.8) and (1.8,1.2) .. (1.5,1.5);
\end{scope}
\fill[black] (3,0) circle (3pt);
\fill[black] (3,1) circle (3pt);
\draw (3,0) -- (3,1);
\draw (3,1.3) circle (.3cm);
\fill[black](4.5,0) circle (3pt);
\draw (4.5,.3) circle (.3cm);
\draw (4.5,.5) circle (.5cm);
\fill[black] (6,1) circle (3pt);
\draw (6,.5) circle (.5cm);
\begin{pgfonlayer}{background2}
	\draw(6.5,1) circle (.5cm);
\end{pgfonlayer}
\begin{pgfonlayer}{background}
	\fill[white] (6.5,.5) circle (5pt);
\end{pgfonlayer}
\end{tikzpicture}$$
\caption{The five connected ribbon graphs with exactly $2$ edges.}
\label{fig:2edges}
\end{figure}

Now suppose that $C(\mathbb{G},\mathbb{T})$ has $n$ chords, where $n\geq 3$. Let $c^\prime$ be a chord distinct from $c$ that intersects $c$, and let $c^{\prime\prime}$ be any chord distinct from $c$ and $c^\prime$. Furthermore, let $e^{\prime\prime}$ be the edge in $\mathbb{G}$ that corresponds to the chord $c^{\prime\prime}$. If $e^{\prime\prime}$ is an edge in $\mathbb{T}$, then the chord diagram of $C(\mathbb{G}/e^{\prime\prime},\mathbb{T}/e^{\prime\prime})$ is obtained from the chord diagram $C(\mathbb{G},\mathbb{T})$ by deleting the chord $c^{\prime\prime}$. In this case, the chords $c$ and $c^\prime$ still intersect in $C(\mathbb{G}/e^{\prime\prime},\mathbb{T}/e^{\prime\prime})$, and hence the edge $e$ is not a bridge or a separating loop in $\mathbb{G}/e^{\prime\prime}$. Therefore, $e$ is also not a bridge or a separating loop in $\mathbb{G}$. If $e^{\prime\prime}$ is not an edge in $\mathbb{T}$, then the chord diagram $C(\mathbb{G}\setminus e^{\prime\prime},\mathbb{T})$ is obtained by deleting the chord $c^{\prime\prime}$ from $C(\mathbb{G},\mathbb{T})$. As in the previous case, the chords $c$ and $c^\prime$ intersect in $C(\mathbb{G}\setminus e^{\prime\prime},\mathbb{T})$ and hence $e$ is not a bridge or separating loop in $\mathbb{G}\setminus e^{\prime\prime}$. Therefore $e$ is not a bridge or separating loop in $\mathbb{G}$.
\end{proof}

\begin{theorem}
Let $\mathbb{G}$ be a connected ribbon graph with edges $e_1,\dots, e_n$. There is a complex $\widetilde{A}(\mathbb{G})$ whose generators are in one-to-one correspondence with the spanning quasi-trees of $\mathbb{G}$ and whose homology is $\widetilde{Kh}(\mathbb{G})$. Moreover, the bigrading of each spanning quasi-tree $\mathbb{T}$ is $(i(\mathbb{T}),j(\mathbb{T}))$ where
\begin{eqnarray}
\label{eqn:i-grading}
i(\mathbb{T})  & = & 2g(\mathbb{T}) +ea(\mathbb{T}) - ia(\mathbb{T}) + |V(\mathbb{G})| -1 ~\text{and}\\
\label{eqn:j-grading}
j(\mathbb{T}) & = & 2(g(\mathbb{T}) + ea(\mathbb{T}) - ia(\mathbb{T})) +|V(\mathbb{G})| -1.
\end{eqnarray}
\end{theorem}
\begin{proof}
The existence of the complex $\widetilde{A}(\mathbb{G})$ is the content of Theorem \ref{theorem:quasi-tree_expansion}. It remains to show that the grading formulas in Equations \ref{eqn:i-grading} and \ref{eqn:j-grading} are equivalent to those in Equations \ref{eqn:grading1} and \ref{eqn:grading2}.

We prove that the grading formulas hold in three steps. First we show that if every edge of $\mathbb{G}$ is either a bridge or a separating loop, then the grading formulas hold. Next, we show that if an edge $e$ is not in $\mathbb{T}$, then the grading formulas for $\mathbb{G}\setminus e$ imply the grading formulas for $\mathbb{G}$. Finally, we show that if an edge $e$ is in $\mathbb{T}$, then the grading formulas for $\mathbb{G}/e$ imply the grading formulas for $\mathbb{G}$.

If an edge $e$ in $\mathbb{G}$ is either a bridge or a separating loop in $\mathbb{G}$, then Proposition \ref{prop:bridge} states the corresponding chord in $C(\mathbb{G},\mathbb{T})$ does not intersect any other chord. 
Suppose that every edge in $\mathbb{G}$ is either a bridge or a separating loop. The resolution tree $\mathcal{T}(\mathbb{G})$ has only one leaf $\mathbb{L}=\mathbb{G}$, and thus only one spanning quasi-tree $\mathbb{T}$. The quasi-tree $\mathbb{T}$ consists of all edges of $\mathbb{G}$ that are bridges. There are no intersections in the chord diagram $C(\mathbb{G},\mathbb{T})$, and hence every edge is active. Since the bridges of $\mathbb{G}$ are internal with respect to $\mathbb{T}$, it follows that $ia(\mathbb{T})=\text{Bridge}(\mathbb{L})=|V(\mathbb{G})|-1$ and $ea(\mathbb{T}) = \text{Loop}(\mathbb{L})$. The grading formulas follow from the above equations and the fact that the genus of $\mathbb{T}$ is zero.

Suppose that $e_n$ is an edge of $\mathbb{G}$ that is not a bridge or a separating loop of $\mathbb{G}$. Proposition \ref{prop:bridge} states that the chord $c_n$ corresponding to $e_n$ intersects another another chord in $C(\mathbb{G},\mathbb{T})$. Suppose that $e_n$ is not an edge of $\mathbb{T}$, and define $\mathbb{T}^\prime$ to be the spanning quasi-tree of $\mathbb{G}\setminus e_n$ with $E(\mathbb{T}^\prime)=E(\mathbb{T})$. Let $\mathbb{L}$ be the leaf of the resolution tree of $\mathbb{G}$ corresponding to $\mathbb{T}$, and let $\mathbb{L}^\prime$ be the leaf of the resolution tree of $\mathbb{G}\setminus e_n$ corresponding to $\mathbb{T}^\prime$. Note that $\mathbb{L}$ and $\mathbb{L}^\prime$ are the same ribbon graph. By induction on the number of edges of $\mathbb{G}$ that are not bridges or separating loops, we know that
\begin{eqnarray*}
i(\mathbb{T}^\prime) & = & \text{Loop}(\mathbb{L}^\prime) + \text{Con}(\mathbb{L}^\prime)\\
& = &
2g(\mathbb{T}^\prime)+ea(\mathbb{T}^\prime)-ia(\mathbb{T}^\prime)+V(\mathbb{G}\setminus e_n) -1~\text{and}\\
j(\mathbb{T}^\prime) & = & 2~\text{Loop}(\mathbb{L}^\prime) - \text{Bridge}(\mathbb{L}^\prime) + \text{Con}(\mathbb{L}^\prime)\\
& = & 2(g(\mathbb{T}^\prime) + ea(\mathbb{T}^\prime) - ia(\mathbb{T}^\prime))+V(\mathbb{G}\setminus e_n) -1.
\end{eqnarray*}

Proposition \ref{prop:chord_diagram} implies that $C(\mathbb{G}\setminus e_n,\mathbb{T}^\prime)$ can be obtained from $C(\mathbb{G},\mathbb{T})$ by deleting the chord $c_n$. Since $c_n$ is the chord with the largest label and it intersects at least one more chord, it is not active. Moreover, the chord $c_n$ does not cause any other chord to be inactive. Therefore, $ia(\mathbb{T}) = ia(\mathbb{T}^\prime)$ and $ea(\mathbb{T})=ea(\mathbb{T}^\prime)$. Also, we have that $g(\mathbb{T})=g(\mathbb{T}^\prime)$ and $|V(\mathbb{G})| = |V(\mathbb{G}\setminus e_n)|$. Let $P$ be the path in the resolution tree of $\mathbb{G}$ from the root to the leaf $\mathbb{L}$, and let $P^\prime$ be the path in the resolution tree of $\mathbb{G}\setminus e_n$ from the root to $\mathbb{L}^\prime$. Since $\mathbb{L}$ and $\mathbb{L}^\prime$ are isomorphic, it follows that $\text{Loop}(\mathbb{L})  = \text{Loop}(\mathbb{L}^\prime)$ and $\text{Bridge}(\mathbb{L})=\text{Bridge}(\mathbb{L}^\prime)$. Since $P$ can be obtained from $P^\prime$ by adding one edge that corresponds to a deletion, it follows that $\text{Con}(\mathbb{L})=\text{Con}(\mathbb{L}^\prime)$. Therefore, $i(\mathbb{T})= i(\mathbb{T}^\prime)$ and $j(\mathbb{T})=j(\mathbb{T}^\prime)$, and the result holds for $\mathbb{T}$.

Suppose that $e_n$ is an edge of $\mathbb{G}$ that is not a bridge or separating loop and that $e_n$ is an edge of $\mathbb{T}$. Then $\mathbb{T}/e_n$ is a spanning quasi-tree of $\mathbb{G}/e_n$. Let $\mathbb{L}$ be the leaf of the resolution tree of $\mathbb{G}$ corresponding to $\mathbb{T}$, and let $\mathbb{L}^\prime$ be the leaf of the resolution tree of $\mathbb{G}/e_n$ corresponding to $\mathbb{T}/e_n$. As in the previous case, the ribbon graphs $\mathbb{L}$ and $\mathbb{L}^\prime$ are isomorphic. By induction on the number of edges of $\mathbb{G}$ that are not bridges or separating loops, we know that
\begin{eqnarray*}
i(\mathbb{T}/e_n) & = & \text{Loop}(\mathbb{L}^\prime) + \text{Con}(\mathbb{L}^\prime)\\
& = &
2g(\mathbb{T}/e_n)+ea(\mathbb{T}/e_n)-ia(\mathbb{T}/e_n)+V(\mathbb{G}/e_n) -1~\text{and}\\
j(\mathbb{T}/e_n) & = & 2~\text{Loop}(\mathbb{L}^\prime) - \text{Bridge}(\mathbb{L}^\prime) + \text{Con}(\mathbb{L}^\prime)\\
& = & 2(g(\mathbb{T}/e_n) + ea(\mathbb{T}/e_n) - ia(\mathbb{T}/e_n))+V(\mathbb{G}/e_n) -1.
\end{eqnarray*}

Proposition \ref{prop:chord_diagram} implies that $C(\mathbb{G}/ e_n,\mathbb{T}/e_n)$ can be obtained from $C(\mathbb{G},\mathbb{T})$ by deleting the chord $c_n$. Since $c_n$ is the chord with the largest label and it intersects at least one more chord, it is not active. Moreover, the chord $c_n$ does not cause any other chord to be inactive. Therefore, $ia(\mathbb{T}) = ia(\mathbb{T}/e_n)$ and $ea(\mathbb{T})=ea(\mathbb{T}/e_n)$. If the edge $e_n$ is a loop, then it must be a nonseparating loop, and thus $g(\mathbb{T}) = g(\mathbb{T}/e_n)+1$ and $|V(\mathbb{G})| = |V(\mathbb{G}/e_n)| -1$. If the edge $e_n$ is not a loop, then $g(\mathbb{T}) = g(\mathbb{T}/e_n)$ and $|V(\mathbb{G})| = |V(\mathbb{G}/e_n)| +1$.
Let $P$ be the path in the resolution tree of $\mathbb{G}$ from the root to the leaf $\mathbb{L}$, and let $P^\prime$ be the path in the resolution tree of $\mathbb{G}/e_n$ from the root to $\mathbb{L}^\prime$. Since $\mathbb{L}$ and $\mathbb{L}^\prime$ are isomorphic, it follows that $\text{Loop}(\mathbb{L})  = \text{Loop}(\mathbb{L}^\prime)$ and $\text{Bridge}(\mathbb{L})=\text{Bridge}(\mathbb{L}^\prime)$. Since $P$ can be obtained from $P^\prime$ by adding one edge that corresponds to a contraction, it follows that $\text{Con}(\mathbb{L})=\text{Con}(\mathbb{L}^\prime)+1$. 

Either $e_n$ is a loop or $e_n$ is not a loop. Suppose $e_n$ is a (necessarily nonseparating) loop. Then
\begin{eqnarray*}
\text{Loop}(\mathbb{L}) + \text{Con}(\mathbb{L}) & = & \text{Loop}(\mathbb{L}^\prime) + \text{Con}(\mathbb{L}^\prime)  +1\\
& = & 2 [g(\mathbb{T}/e_n) + 1] + ea(\mathbb{T}/e_n) - ia(\mathbb{T}/e_n) + [|V(\mathbb{G}/e_n)| - 1] -1\\
& = & 2 g(\mathbb{T}) + ea(\mathbb{T}) - ia(\mathbb{T}) +|V(\mathbb{G})| - 1 
\end{eqnarray*}
and
\begin{eqnarray*}
2~\text{Loop}(\mathbb{L}) - \text{Bridge}(\mathbb{L}) + \text{Con}(\mathbb{L}) & = &
2 ~ \text{Loop}(\mathbb{L}^\prime) - \text{Bridge}(\mathbb{L}^\prime) + \text{Con}(\mathbb{L}^\prime) + 1\\
& = & 2([g(\mathbb{T}/e_n) +1] +ea(\mathbb{T}/e_n) - ia(\mathbb{T}/e_n)) + |V(\mathbb{G}/e_n)| -2\\
& = & 2(g(\mathbb{T}) + ea(\mathbb{T}) -ia(\mathbb{T}))+|V(\mathbb{G})| -1. 
\end{eqnarray*}
Suppose that $e_n$ is not a loop. Then
\begin{eqnarray*}
\text{Loop}(\mathbb{L}) + \text{Con}(\mathbb{L}) & = & \text{Loop}(\mathbb{L}^\prime) + \text{Con}(\mathbb{L}^\prime)  +1\\
& = & 2 g(\mathbb{T}/e_n) + ea(\mathbb{T}/e_n) - ia(\mathbb{T}/e_n) + [|V(\mathbb{G}/e_n)| + 1] -1\\
& = & 2 g(\mathbb{T}) + ea(\mathbb{T}) - ia(\mathbb{T}) +|V(\mathbb{G})| - 1 
\end{eqnarray*}
and
\begin{eqnarray*}
2~\text{Loop}(\mathbb{L}) - \text{Bridge}(\mathbb{L}) + \text{Con}(\mathbb{L}) & = &
2 ~ \text{Loop}(\mathbb{L}^\prime) - \text{Bridge}(\mathbb{L}^\prime) + \text{Con}(\mathbb{L}^\prime) + 1\\
& = & 2(g(\mathbb{T}/e_n) +ea(\mathbb{T}/e_n) - ia(\mathbb{T}/e_n)) + [|V(\mathbb{G}/e_n)| +1]-1\\
& = & 2(g(\mathbb{T}) + ea(\mathbb{T}) -ia(\mathbb{T}))+|V(\mathbb{G})| -1. 
\end{eqnarray*}
\end{proof}

\begin{corollary}
\label{cor:Jones}
Let $D$ be a diagram of the link $L$, and let $\mathbb{D}$ be the all-$A$ ribbon graph of $D$. The Jones polynomial of $L$ evaluated at $q^2$ can be expressed as
$$J_L(q^2) = \sum_{\mathbb{T}\in\mathcal{Q}(\mathbb{G})} (-1)^{i(\mathbb{T})-n_-}q^{j(\mathbb{T})+n_+-2n_-},$$
where $n_{\pm}$ is the number of positive and negative crossings in $D$ respectively.
\end{corollary}


\section{Applications of the quasi-tree model}
\label{section:applications}

\subsection{Genus and homological width}
Recall that the diagonal grading $\delta$ is defined as $\delta=j/2-i$ and that the homological width of $\widetilde{Kh}(\mathbb{G})$ is defined by $hw(\widetilde{Kh}(\mathbb{G}))=\delta_{\max}(\mathbb{G}) - \delta_{\min}(\mathbb{G}) + 1$, where 
\begin{eqnarray*}
\delta_{\max}(\mathbb{G}) & = & \max\{\delta~|~\widetilde{Kh}^{\delta}(\mathbb{G})\neq 0\}~\text{and}\\
\delta_{\min}(\mathbb{G}) & = & \min\{\delta~|~\widetilde{Kh}^{\delta}(\mathbb{G})\neq 0\}.
\end{eqnarray*}

If $\mathbb{D}$ is the all-$A$ ribbon graph of a link diagram $D$, then the surface $\Sigma_D$ obtained by capping off the boundary components of $\Sigma_{\mathbb{D}}$ with disks is called the Turaev surface of $D$. The {\em Turaev genus of $L$}, denoted $g_T(L)$, is defined as $g_T(L)=\min\{g(\Sigma_D)~|~D~\text{is a diagram of}~L\}.$ There are several lower bounds on Turaev genus coming from the different knot homology theories \cite{CKS:Quasi-trees, Lowrance:KnotFloerWidth, DasbachLowrance:KnotSignature}, and Theorem \ref{theorem:hw} is another result of this flavor. Further studies of Turaev genus may be found in \cite{DFKLS:GraphsOnSurfaces,Abe:TuraevGenusAdequateKnot, Lowrance:ThreeBraids}.
\begin{proof}[Proof of Theorem \ref{theorem:hw}] Let $\mathbb{G}$ be a ribbon graph whose edges are numbered $1,\dots,n$. The diagonal grading of a generator $\mathbb{T}$ in the spanning quasi-tree complex is given by
\begin{eqnarray*}
\delta(\mathbb{T}) & = & \frac{ j(\mathbb{T})}{2} - i(\mathbb{T})\\
& = & g(\mathbb{T}) + ea(\mathbb{T}) - ia(\mathbb{T}) +  \frac{1}{2}(|V(\mathbb{G})| - 1)\\
&  & -\left(2g(\mathbb{T}) + ea(\mathbb{T}) - ia(\mathbb{T}) + |V(\mathbb{G})| -1 \right)\\
& = & -g(\mathbb{T}) - \frac{1}{2}(|V(\mathbb{G})| -1).
\end{eqnarray*} 
While the homological and polynomial gradings each depend on the activities of the spanning quasi-tree (and thus on the labeling of the edges of $\mathbb{G}$), the diagonal grading does not. In fact, since $\frac{1}{2}(|V(\mathbb{G})|-1)$ does not depend on $\mathbb{T}$, the difference between the diagonal gradings of any two spanning quasi-trees depends only on the difference of their genera. The result follows from the fact that any ribbon graph $\mathbb{G}$ contains a spanning quasi-tree of genus zero and a spanning quasi-tree of genus $g(\mathbb{G})$.
\end{proof}

\subsection{Ribbon graphs with no loops}
Throughout this section, suppose $\mathbb{G}$ is a ribbon graph with no loops. If, in addition to having no loops, the dual ribbon graph $\mathbb{G}^*$ has no loops, then $\mathbb{G}$ is called {\em adequate}. If $\mathbb{D}$ is adequate and the all-$A$ ribbon graph of some link diagram of the link $L$, then $L$ is said to be adequate. Khovanov \cite{Khovanov:Patterns} proves that the summands of the Khovanov homology of an adequate link in the maximal and minimal polynomial gradings are each isomorphic to $\mathbb{Z}$. Abe \cite{Abe:TuraevGenusAdequateKnot} proves that if a link is adequate, then its Turaev genus is its homological width plus one. Theorem \ref{theorem:loopless} and Corollary \ref{cor:adequate} are ribbon graph analogs of Khovanov's result for adequate links, and Corollary \ref{cor:hwadequate} is the ribbon graph version of Abe's result.
\begin{lemma}
\label{lemma:int_inactive}
Let $\mathbb{G}$ be a ribbon graph with no loops, and let $\mathbb{T}$ be a spanning quasi-tree of $\mathbb{G}$ with $g(\mathbb{T})>0$. For any ordering of the edges of $\mathbb{G}$, there are at least $g(\mathbb{T})+1$ internally inactive edges with respect to $\mathbb{T}$.
\end{lemma}
\begin{proof}
Since $g(\mathbb{T})>0$ and $\mathbb{G}$ has no loops, $\mathbb{T}$ has a ribbon subgraph of one of the two following forms. The first type of ribbon subgraph contains at least two vertices $u$ and $v$ and three disjoint paths $P_1$, $P_2$ and $P_3$ from $u$ to $v$ such that the cyclic order of the paths at each both $u$ and $v$ is $(P_1,P_2,P_3)$ as in Figure \ref{fig:schem1}. If $P_1,P_2,$ and $P_3$ contain $k,l,$ and $m$ edges respectively, then the chord diagram $C(\mathbb{G},\mathbb{T})$ contains the configuration on the right of Figure \ref{fig:schem1}.
\begin{figure}[h]
$$\begin{tikzpicture}[thick,scale=.9]
\draw (2,2) circle (2cm);
\draw (2,4) node[above]{$P_1$};
\draw (2,0) node[below]{$P_3$};
\draw (2,2) node[below]{$P_2$};
\fill (4,2) circle (.2cm);
\draw (4.2,1.8) node[right]{$v$};
\fill (0,2) circle (.2cm);
\draw (-.2,1.8) node[left]{$u$};
\draw (0,2) arc (270:80:.5cm);
\draw (4,2) .. controls (2,2) and (.7, 2.5) .. (.4,2.8);
\def \circle{(10,2) circle (2cm)};
\draw \circle;
\begin{scope}[xshift = 10cm, yshift=2cm]
	\draw (50:2cm) -- (250:2cm);
	\draw (70:2cm) -- (230:2cm);
	\draw (55:1.8cm) node{.};
	\draw (60:1.8cm) node{.};
	\draw (65:1.8cm) node{.};
	\draw (235:1.8cm) node{.};
	\draw (240:1.8cm) node{.};
	\draw (245:1.8cm) node{.};
	\draw (110:2cm) -- (310:2cm);
	\draw (130:2cm) -- (290:2cm);
	\draw (115:1.8cm) node{.};
	\draw (120:1.8cm) node{.};
	\draw (125:1.8cm) node{.};
	\draw (295:1.8cm) node{.};
	\draw (300:1.8cm) node{.};
	\draw (305:1.8cm) node{.};
	\draw (170:2cm) -- (10:2cm);
	\draw(190:2cm) -- (350:2cm);
	\draw (-5:1.8cm) node{.};
	\draw (0:1.8cm) node{.};
	\draw (5:1.8cm) node{.};
	\draw (175:1.8cm) node{.};
	\draw (180:1.8cm) node{.};
	\draw (185:1.8cm) node{.};
\end{scope}
\draw (10.8,4) node[right]{$k$};
\draw (9.2,4) node[left]{$l$};
\draw (8,2) node[left]{$m$};
\end{tikzpicture}$$
\caption{{\bf Left:} A schematic depiction of the first type of a ribbon subgraph of $\mathbb{T}$. {\bf Right:} The corresponding configuration in the chord diagram $C(\mathbb{G},\mathbb{T})$.}
\label{fig:schem1}
\end{figure}
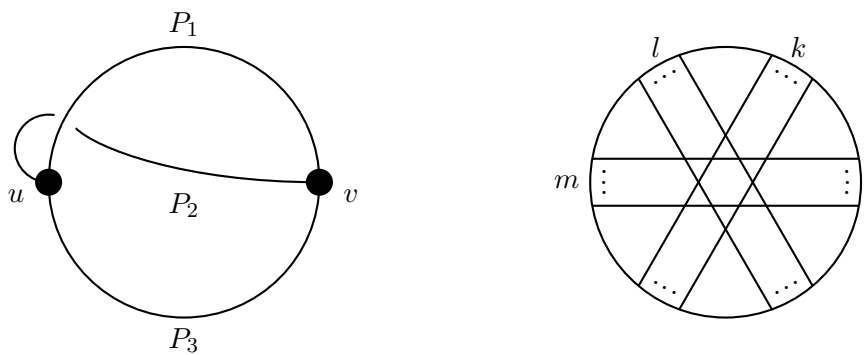

The second type of ribbon subgraph consists of a vertex $v$ and two disjoint cycles $C_1$ and $C_2$ containing $v$. The cyclic order of the cycles at the vertex $v$ is $(C_1,C_2,C_1,C_2)$ as depicted on the in Figure \ref{fig:schem2}. If $C_1$ contains $k$ edges and $C_2$ contains $l$ edges, then the chord diagram $C(\mathbb{G},\mathbb{T})$ contains the configuration on the right of Figure \ref{fig:schem2}. Since $\mathbb{G}$ has no loops, both $k$ and $l$ are at least two.
\begin{figure}[h]
$$\begin{tikzpicture}[scale=.9]
\draw (2,2) circle (1.5cm);
\draw (3.5,2) node[left]{$C_2$};
\fill (0.5,2) circle (.2cm);
\draw (.3,1.8) node[left]{$v$};
\begin{pgfonlayer}{background2}
	\draw (0.5,3.5) circle (1.5cm);
	\draw (-1,3.5) node[right]{$C_1$};
\end{pgfonlayer}
\begin{pgfonlayer}{background}
	\fill[white] (2,3.5) circle (.3cm);
\end{pgfonlayer}
\draw (10,2) circle (2cm);
\begin{scope}[xshift = 10cm, yshift = 2cm]
	\draw (10:2cm) -- (170:2cm);
	\draw (-10:2cm) -- (190:2cm);
	\draw (-5:1.8cm) node{.};
	\draw (0:1.8cm) node{.};
	\draw (5:1.8cm) node{.};
	\draw (175:1.8cm) node{.};
	\draw (180:1.8cm) node{.};
	\draw (185:1.8cm) node{.};
	\draw (80:2cm) -- (280:2cm);
	\draw (85:1.8cm) node{.};
	\draw (90:1.8cm) node{.};
	\draw (95:1.8cm) node{.};
	\draw (100:2cm) -- (260:2cm);
	\draw (265:1.8cm) node{.};
	\draw (270:1.8cm) node{.};
	\draw (275:1.8cm) node{.};
	\draw (90:2cm) node[above]{$k$};
	\draw (180:2cm) node[left]{$l$};
\end{scope}
\end{tikzpicture}$$
\caption{{\bf Left:} A schematic depiction of the second type of ribbon subgraph of $\mathbb{T}$. {\bf Right:} The corresponding configuration in the chord diagram $C(\mathbb{G},\mathbb{T})$. Note that $k\geq 2$ and $l\geq 2$.}
\label{fig:schem2}
\end{figure}
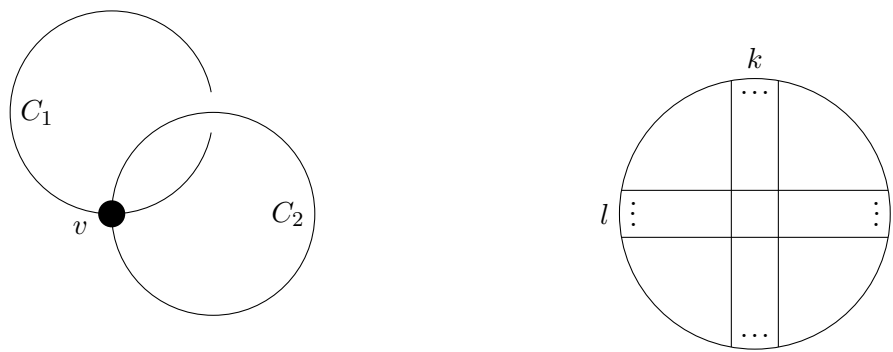

We will prove a slightly stronger result than stated via induction on the genus of $\mathbb{T}$. Let $C_{\text{int}}(\mathbb{G},\mathbb{T})$ be the chord diagram obtained by deleting all chords that are external with respect to $\mathbb{T}$. We show there are at least $g(\mathbb{T})+1$ internally inactive edges in $C_{\text{int}}(\mathbb{G},\mathbb{T})$. Since adding the external edges back can only possibly increase the number of internally inactive edges, the result follows. If $g(\mathbb{T})=1$, then $C_{\text{int}}(\mathbb{G},\mathbb{T})$ contains a configuration as in either Figure \ref{fig:schem1} or Figure \ref{fig:schem2}. In either case, there are at least two chords which must intersect lower numbered chords. Hence there are at least two internally inactive edges with respect to $\mathbb{T}$.

Suppose that $g(\mathbb{T})>1$. Since $g(\mathbb{T})>0$, there exists two chords $c_1$ and $c_2$ in $C(\mathbb{G},\mathbb{T})$ that are internal and intersect. Since they intersect, at least one of $c_1$ and $c_2$ is internally inactive with respect to $\mathbb{T}$. If $c_1$ and $c_2$ correspond to edges $e_1$ and $e_2$ respectively, there is a spanning quasi-tree $\mathbb{T}^\prime$ of $\mathbb{G}$ with $E(\mathbb{T}^\prime) = E(\mathbb{T})\setminus\{e_1,e_2\}$. The chords $c_1$ and $c_2$ intersect in $C(\mathbb{G},\mathbb{T}^\prime)$ except now they are external with respect to $\mathbb{T}^\prime$. 
\begin{figure}[h]
$$\begin{tikzpicture}[thick,>= triangle 45]
\begin{scope}[xshift=4cm]
	\draw (-2,2) node{\Large{$C(\mathbb{G},\mathbb{T}^\prime)$}};
	\draw (2,2) circle (2cm);
	\draw (2,4) node[above]{$c_1$};
	\draw (2,0) node[below]{$c_1$};
	\draw (2,4) -- (2,0);
	\draw (0,2) node[left]{$c_2$};
	\draw (4,2) node[right]{$c_2$};
	\draw (0,2) -- (4,2);
	\draw (1.3,3.9) node[above]{$w_1$};
	\draw (0.1,2.7) node[left]{$w_{k_1}$};
	\draw[loosely dotted, ultra thick] (.65,3.73) arc (128:142:2cm);
	\draw (1.3,0.1) node[below]{$x_{k_2}$};
	\draw (0.1,1.3) node[left]{$x_1$};
	\draw[loosely dotted, ultra thick] (.65,.27) arc (232:218:2cm);
	\draw (2.7,0.1) node[below]{$y_1$};
	\draw (3.9, 1.3) node[right]{$y_{k_3}$};
	\draw[loosely dotted, ultra thick] (3.35,.27) arc (308:322:2cm);
	\draw (3.9,2.7) node[right]{$z_1$};
	\draw (2.7,3.9) node[above]{$z_{k_4}$};
	\draw[loosely dotted, ultra thick] (3.35,3.73) arc (52:38:2cm);
\end{scope}
\draw[->,very thick] (4.2,0.2) -- (3.5,-.5);
\draw (3.967, -2.6327) arc (10:80:2cm);
\draw (0.033, -2.6327) arc (170:100:2cm);
\draw (0.033, -3.3473) arc (190:260:2cm);
\draw (3.967, -3.3473) arc (350:280:2cm);
\draw (2.3473,-1.033) -- (2.3473, -4.967);
\draw[xshift = -.6946cm] (2.3473,-1.033) -- (2.3473, -4.967);
\begin{pgfonlayer}{background}
	\fill[white] (2,-3) circle (2cm);
\end{pgfonlayer}
\begin{pgfonlayer}{background2}
\draw[thick] (2,-4.5) ellipse (3cm and 1.5cm);
\draw[thick] (2,-4.35) ellipse (3.4cm and 2.1cm);
\end{pgfonlayer}
\begin{scope}[yshift = -5cm]
	\draw (1.3,3.9) node[above]{$w_1$};
	\draw (0.1,2.7) node[left]{$w_{k_1}$};
	\draw[loosely dotted, ultra thick] (.65,3.73) arc (128:142:2cm);
	\draw (1.3,0.1) node[below]{$x_{k_2}$};
	\draw (0.1,1.3) node[left]{$x_1$};
	\draw[loosely dotted, ultra thick] (.65,.27) arc (232:218:2cm);
	\draw (2.7,0.1) node[below]{$y_1$};
	\draw (3.9, 1.3) node[right]{$y_{k_3}$};
	\draw[loosely dotted, ultra thick] (3.35,.27) arc (308:322:2cm);
	\draw (3.9,2.7) node[right]{$z_1$};
	\draw (2.7,3.9) node[above]{$z_{k_4}$};
	\draw[loosely dotted, ultra thick] (3.35,3.73) arc (52:38:2cm);
\end{scope}
\draw (1.7,-3) node[left]{$c_1$};
\draw (2.3,-3) node[right]{$c_1$};
\draw (2,-5.9) node[above]{$c_2$};
\draw (2,-6.4) node[below]{$c_2$};
\draw[->,very thick] (6,-3) -- (7,-3);
\begin{scope}[xshift=8cm,yshift=-5cm]
	\draw (2,2) circle (2cm);
	\draw (2,4) node[above]{$c_1$};
	\draw (2,0) node[below]{$c_1$};
	\draw (2,4) -- (2,0);
	\draw (0,2) node[left]{$c_2$};
	\draw (4,2) node[right]{$c_2$};
	\draw (0,2) -- (4,2);
	\draw (1.3,3.9) node[above]{$w_1$};
	\draw (0.1,2.7) node[left]{$w_{k_1}$};
	\draw[loosely dotted, ultra thick] (.65,3.73) arc (128:142:2cm);
	\draw (1.3,0.1) node[below]{$z_{k_2}$};
	\draw (0.1,1.3) node[left]{$z_1$};
	\draw[loosely dotted, ultra thick] (.65,.27) arc (232:218:2cm);
	\draw (2.7,0.1) node[below]{$y_1$};
	\draw (3.9, 1.3) node[right]{$y_{k_3}$};
	\draw[loosely dotted, ultra thick] (3.35,.27) arc (308:322:2cm);
	\draw (3.9,2.7) node[right]{$x_1$};
	\draw (2.7,3.9) node[above]{$x_{k_4}$};
	\draw[loosely dotted, ultra thick] (3.35,3.73) arc (52:38:2cm);
	\draw (2.8,4.5) node[above]{\Large{$C(\mathbb{G},\mathbb{T})$}};
\end{scope}
\end{tikzpicture}$$
\caption{Chords $c_1$ and $c_2$ are exterior in $C(\mathbb{G},\mathbb{T})$. Surgering bands along $c_1$ and $c_2$ results in a one component diagram, which can be smoothly deformed to give the chord diagram $C(\mathbb{G},\mathbb{T})$, where $c_1$ and $c_2$ are interior.}
\label{fig:chord_surgery}
\end{figure}
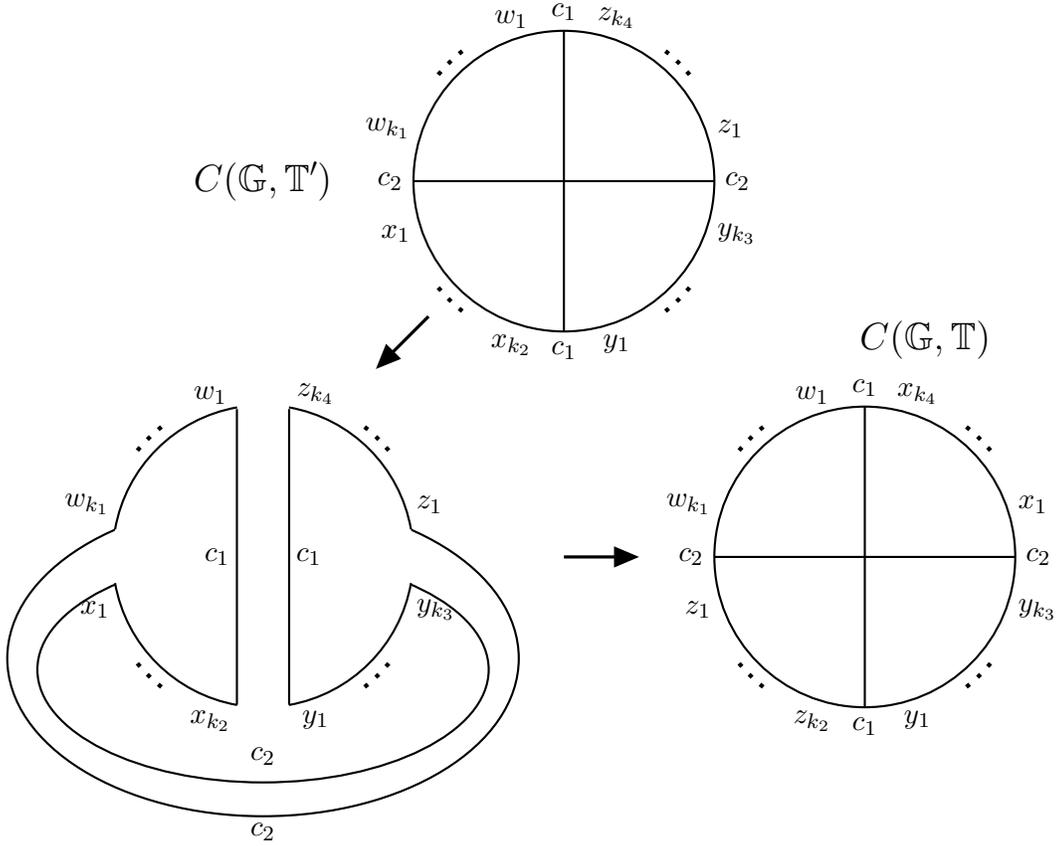

The endpoints of the chords $c_1$ and $c_2$ partition the circle in the chord diagram $C(\mathbb{G},\mathbb{T}^\prime)$ into four arcs. Label the endpoints of the other chords in $C(\mathbb{G},\mathbb{T}^\prime)$ by $w_{k_i},x_{k_i},y_{k_i},$ or $z_{k_i}$ as in Figure \ref{fig:chord_surgery}. Surgery may be performed along $c_1$ and $c_2$ in $C(\mathbb{G},\mathbb{T}^\prime)$ by replacing the chords with bands and labeling opposing midpoints of each band by the original label of the chord. The resulting diagram has one component that can be deformed into a round circle, producing the chord diagram $C(\mathbb{G},\mathbb{T})$. This process is depicted in Figure \ref{fig:chord_surgery}. 

Two chords different from $c_1$ and $c_2$ intersect in $C(\mathbb{G},\mathbb{T}^\prime)$ if and only if they also intersect in $C(\mathbb{G},\mathbb{T})$. By way of induction, there are at least $g(\mathbb{T}^\prime) + 1 = g(\mathbb{T})$ internally inactive edges in $C_{\text{int}}(\mathbb{G},\mathbb{T}^\prime)$. Those chords are also internally inactive in $C_{\text{int}}(\mathbb{G},\mathbb{T}).$ Since at least one of $c_1$ and $c_2$ must also be internally inactive in $C_{\text{int}}(\mathbb{G},\mathbb{T})$, it follows that there are at least $g(\mathbb{T})+1$ internally inactive edges in $C_{\text{int}}(\mathbb{G},\mathbb{T})$.
\end{proof}

\begin{proof}[Proof of Theorem \ref{theorem:loopless}]
Let $\mathbb{T}$ be a spanning quasi-tree of $\mathbb{G}$ such that $g(\mathbb{T})=0$. Hence the underlying graph $T$ of $\mathbb{T}$ is a spanning tree of the underlying graph $G$ of $\mathbb{G}$. We prove that there exists a labeling of the edges of $\mathbb{G}$ such that $i(\mathbb{T})=0$ and $j(\mathbb{T})=1-|V(\mathbb{G})|$ and that there is no other spanning quasi-tree $\mathbb{T}^\prime$ with $j(\mathbb{T}^\prime)\leq 1-|V(\mathbb{G})|$. The result follows from Theorem \ref{theorem:maintheorem2}.

Arbitrarily label the edges of $\mathbb{T}$ by $1,\dots,k$ and the edges of $\mathbb{G}$ not in $\mathbb{T}$ by $k+1,\dots,n$. Since $g(\mathbb{T})=0$, the chords corresponding to the edges of $\mathbb{T}$ do not intersect one another. Thus every edge in $\mathbb{T}$ is internally active, and $ia(\mathbb{T}) = |E(\mathbb{T})| = |V(\mathbb{G})|-1$. Let $c$ be a chord corresponding an edge $e$ in $\mathbb{G}$ but not in $\mathbb{T}$. Since $\mathbb{G}$ is loopless, the endpoints of the edge $e$ are two distinct vertices $u$ and $v$ in $V(\mathbb{G})$. Since $T$ is a spanning tree of $\mathbb{G}$, there is a unique path $\gamma$ from $u$ to $v$ in $T$. The chord $c$ intersects each of the chords corresponding to the edges of $\gamma$, and hence $c$ is not externally active. Therefore $ea(\mathbb{T}) = 0$, and thus $i(\mathbb{T}) = 0$ and $j(\mathbb{T})=1-|V(\mathbb{G})|$.

We prove that there does not exist a spanning quasi-tree $\mathbb{T}^\prime$ distinct from $\mathbb{T}$ such that $j(\mathbb{T}^\prime)\leq 1 - |V(\mathbb{G})|$ in two steps: first we show that no such genus zero quasi-tree can exist, and then we show it for quasi-trees with positive genus. 

Suppose that $\mathbb{T}^\prime$ is a quasi-tree which is distinct from $\mathbb{T}$ and such that $g(\mathbb{T}^\prime) = 0$. The least possible $j$-grading occurs when $ea(\mathbb{T}^\prime) = 0$ and $ia(\mathbb{T}^\prime) = |E(\mathbb{T})|$. Therefore, if $\mathbb{T}^\prime$ contains an internally inactive edge, then $j(\mathbb{T})>1-|V(\mathbb{G})|$.
Since $\mathbb{T}$ and $\mathbb{T}^\prime$ are distinct, there is a cycle $C$ in $\mathbb{G}$ whose edge set is a subset of $E(\mathbb{T})\cup E(\mathbb{T}^\prime)$. Also, there exist edges $e$ and $e^\prime$ in the edge set of $C$, such that $e\in E(\mathbb{T})$, $e\notin E(\mathbb{T}^\prime)$, $e^\prime\in E(\mathbb{T}^\prime)$, and $e^\prime\notin E(\mathbb{T})$. The chords $c$ and $c^\prime$  corresponding to the edges $e$ and $e^\prime$ intersect in $C(\mathbb{G},\mathbb{T}^\prime)$. Therefore $c^\prime$ is not internally active. 

Suppose that $\mathbb{T}^\prime$ is a spanning quasi-tree of $\mathbb{G}$ of positive genus. Lemma \ref{lemma:int_inactive} states that there are at least $g(\mathbb{T}^\prime)+1$ internally inactive edges in $C(\mathbb{G},\mathbb{T})$. An Euler characteristic argument implies that $2g(\mathbb{T}^\prime) - |E(\mathbb{T}^\prime)| = 1 - |V(\mathbb{G})|$. Thus
\begin{eqnarray*}
j(\mathbb{T}^\prime) & = & 2[g(\mathbb{T}^\prime) + ea(\mathbb{T}^\prime) - ia(\mathbb{T}^\prime)] +V(\mathbb{G}) -1\\
&\geq & 2[g(\mathbb{T}^\prime) + ea(\mathbb{T}^\prime) - (|E(\mathbb{T}^\prime)| - g(\mathbb{T}^\prime) -1)]+V(\mathbb{G}) -1\\
& = & 2[2g(\mathbb{T}^\prime) - |E(\mathbb{T}^\prime)| +ea(\mathbb{T}^\prime) +1] + V - 1\\
& = & 2[1-|V(\mathbb{G})|+ea(\mathbb{T}^\prime) +1] + V - 1\\
& = & 3- |V(\mathbb{G})| + 2~ea(\mathbb{T}^\prime).
\end{eqnarray*}
Therefore $j(\mathbb{T}^\prime)> 1 - |V(\mathbb{G})|$, and the result follows.
\end{proof}
For another result about the form of Khovanov homology in its extremal gradings see Todd \cite{Todd:KhovanovTwistNumber}.

\begin{proof}[Proof of Corollary \ref{cor:adequate}]
Since $\mathbb{G}$ is adequate, neither $\mathbb{G}$ nor $\mathbb{G}^*$ have loops. Corollary \ref{cor:duality} implies that $j_{\max}(\mathbb{G}) = |E(\mathbb{G})| - j_{\min}(\mathbb{G}^*)$. Since $j_{\min}(\mathbb{G}^*) = 1 - |V(\mathbb{G}^*)| = 1 - |F(\mathbb{G})|$, it follows that $j_{\max}(\mathbb{G}) =  |E(\mathbb{G})| + |F(\mathbb{G})| -1$. The result follows from Theorem \ref{theorem:loopless} and Corollary \ref{cor:duality}.
\end{proof}

Abe \cite{Abe:TuraevGenusAdequateKnot} proves that if a link is adequate, then the upper bound on its homological width coming from Khovanov homology is tight. We generalize that result to ribbon graphs in the following corollary.
\begin{corollary}
\label{cor:hwadequate}
Let $\mathbb{G}$ be an adequate ribbon graph. Then
$$hw(\widetilde{Kh}(\mathbb{G})) = g(\mathbb{G}) + 1.$$
\end{corollary}
\begin{proof}
Corollary \ref{cor:adequate} implies that $\widetilde{Kh}(\mathbb{G})$ is nontrivial in bigradings $(0, 1 - |V(\mathbb{G})|)$ and $(|E(\mathbb{G})|,|E(\mathbb{G})| + |F(\mathbb{G})| - 1)$. Hence it is nontrivial in diagonal grading  
\begin{eqnarray*}
\delta_1 & = & \frac{1 - |V(\mathbb{G})|}{2}~\text{and}\\ 
\delta_2 & = & \frac{|F(\mathbb{G})| - |E(\mathbb{G})| - 1}{2}.
\end{eqnarray*}
Since $2 g(\mathbb{G}) = 2 - |V(\mathbb{G})| + |E(\mathbb{G})| - |F(\mathbb{G})|$, it follows that
\begin{eqnarray*}
\delta_1 - \delta_2 & = & \frac{2 -|V(\mathbb{G})| + |E(\mathbb{G})| + |F(\mathbb{G})|}{2} \\
& = & g(\mathbb{G}).\\
\end{eqnarray*}
Therefore $hw(\widetilde{Kh}(\mathbb{G}))\geq g(\mathbb{G})+1$. Since Theorem \ref{theorem:hw} states that $hw(\widetilde{Kh}(\mathbb{G}))\leq g(\mathbb{G}) + 1$, the desired equality follows.
\end{proof}

\subsection{Bounds on gradings}
Let $\mathbb{G}$ be a ribbon graph possibly with loops, and let $i_{\min}(\mathbb{G})$, $i_{\max}(\mathbb{G})$, $j_{\min}(\mathbb{G})$, and $j_{\max}(\mathbb{G})$ be the minimum and maximum homological and polynomial gradings of $\widetilde{Kh}(\mathbb{G})$ respectively. From the cube of resolutions complex for $\widetilde{Kh}(\mathbb{G})$ it is clear that $i_{\min}(\mathbb{G})\geq 0$ and $i_{\max}(\mathbb{G})\leq |E(\mathbb{G})|$, and thus $i_{\max}(\mathbb{G}) - i_{\min}(\mathbb{G}) \leq |E(\mathbb{G})|.$ The relationship between the maximum and minimum polynomial gradings is captured in the following proposition.
\begin{proposition}
\label{prop:polygrading}
Let $\mathbb{G}$ be a ribbon graph, and let $j_{\min}(\mathbb{G})$ and $j_{\max}(\mathbb{G})$ be the minimum and maximum polynomial gradings where $\widetilde{Kh}(\mathbb{G})$ is nontrivial. Then
$$j_{\max}(\mathbb{G}) - j_{\min}(\mathbb{G}) \leq 2(|E(\mathbb{G})| - g(\mathbb{G})).$$
Moreover, if $\mathbb{G}$ is adequate, then 
$$j_{\max}(\mathbb{G}) - j_{\min}(\mathbb{G}) = 2(|E(\mathbb{G})| - g(\mathbb{G})).$$
\end{proposition}
\begin{proof}
Equation \ref{eqn:j-grading} implies that the lowest possible polynomial grading occurs when $g(\mathbb{T}) = 0$ and $ia(\mathbb{T}) = |E(\mathbb{T})|$. In this case, $j_{\min}(\mathbb{G}) = 1-|V(\mathbb{G})|$. Note that there is no guarantee that such a $\mathbb{T}$ exists, and thus in general, we have that $j_{\min}(\mathbb{G}) \geq 1 - |V(\mathbb{G})|$. Similarly to the proof of Corollary \ref{cor:adequate}, if $j_{\min}(\mathbb{G}) \geq 1 - |V(\mathbb{G})|$, then Corollary \ref{cor:duality} implies that $j_{\max}(\mathbb{G}) \leq |E(\mathbb{G})| + |F(\mathbb{G})| -1$.

Hence, 
\begin{eqnarray*}
j_{\max}(\mathbb{G}) - j_{\min}(\mathbb{G}) & \leq & |E(\mathbb{G})| + |F(\mathbb{G})| + |V(\mathbb{G})| -2\\
& = & 2|E(\mathbb{G})| + (|V(\mathbb{G})| - |E(\mathbb{G})| + |F(\mathbb{G})| - 2 )\\
& = & 2|E(\mathbb{G})|-2g(\mathbb{G})
\end{eqnarray*}
Corollary \ref{cor:adequate} implies that equality is achieved when $\mathbb{G}$ is adequate.
\end{proof}

Theorem \ref{theorem:hw} and Proposition \ref{prop:polygrading} impose strong restrictions on the possible bigradings where $\widetilde{Kh}(\mathbb{G})$ can be nontrivial.  The shaded area in Figure \ref{fig:bounds} indicates where $\widetilde{Kh}(\mathbb{G})$ can be nontrivial.
\begin{figure}[h]
$$\begin{tikzpicture}[thick, >= triangle 45]
\draw (1,1) rectangle (7,7);
\draw[->] (1,7) -- (1,8);
\draw (1,8) node[above]{$j$-grading};
\draw[->] (7,1) -- (8,1);
\draw (8,1) node[right]{$i$-grading};
\begin{scope}[xshift = -1cm]
	\draw (0,1) -- (0,3.5);
	\draw (-.5,1) -- (.5,1);
	\draw (0,4.5) -- (0,7);
	\draw (-.5,7) -- (.5,7);
	\draw (0,4) node{\Large{$2(|E(\mathbb{G})|-g(\mathbb{G}))$}};
\end{scope}
\draw (1,0) -- (3,0);
\draw (5,0) -- (7,0);
\draw (1,-.5) -- (1,.5);
\draw (7,-.5) -- (7,.5);
\draw (4,0) node{\Large{$|E(\mathbb{G})|$}};
\draw (1,2.5) -- (5.5,7);
\draw (2.5,1) -- (7,5.5);
\draw (4,4) node{\Large{$g(\mathbb{G})$}};
\draw (2.5,4) -- (3.4,4);
\draw (5.5,4) -- (4.6,4);
\begin{pgfonlayer}{background}
\fill[black!30!] (1,1) -- (1,2.5) -- (5.5,7) -- (7,7) -- (7, 5.5) -- (2.5,1) -- (1,1);
\end{pgfonlayer}
\end{tikzpicture}$$
\caption{For a ribbon graph $\mathbb{G}$, the reduced Khovanov homology $\widetilde{Kh}(\mathbb{G})$ can only be nontrivial within the shaded region.}
\label{fig:bounds}
\end{figure}
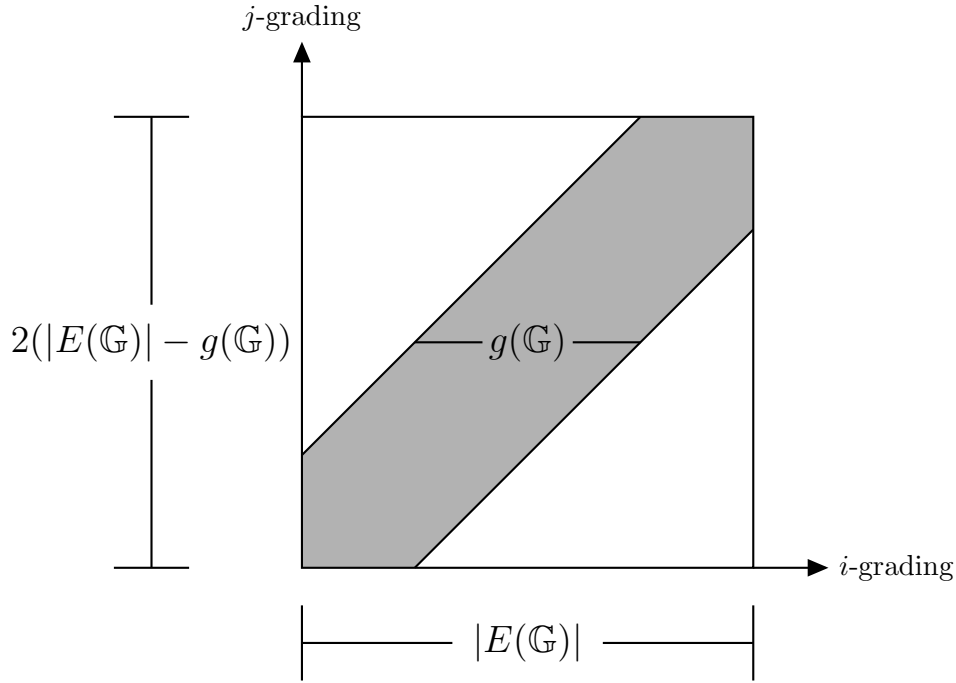

\section{An example}
\label{section:example}

In this section, we compute $Kh(\mathbb{G})$ and $\widetilde{Kh}(\mathbb{G})$ for a ribbon graph $\mathbb{G}$. We first compute $Kh(\mathbb{G})$ via the cube of resolutions complex. Although $\widetilde{Kh}(\mathbb{G})$ can also be computed from the cube of resolutions complex, we instead compute $\widetilde{Kh}(\mathbb{G})$ via the spanning quasi-tree model, which is possible because the differential in the spanning quasi-tree complex is zero for grading reasons.
\begin{figure}[h]
$$\begin{tikzpicture}[scale=1.5,thick]
\fill (0,0) circle (.2cm);
\draw (0.3536, 0.3536) circle (.5cm);
\draw (.7,.8) node[right]{$2$};
\draw (0,0) arc (-90:35:.5cm);
\draw (0,0) arc (270:55:.5cm);
\draw (-.5,.5) node[left]{$1$};
\draw (0,0) arc (180:95:.5cm);
\draw(0,0) arc (-180:30:.5cm);
\draw (0.55,0.49) arc (85:55:.5cm);
\draw (.9,-.3) node[right]{$3$};
\end{tikzpicture}$$
\caption{We compute the Khovanov homology $Kh(\mathbb{G})$ and reduced Khovanov homology $\widetilde{Kh}(\mathbb{G})$ for the depicted ribbon graph $\mathbb{G}$.}
\label{fig:nonlinkribbon}
\end{figure}
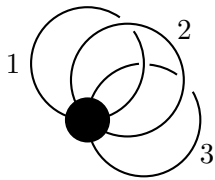

The ribbon graph $\mathbb{G}$ in Figure \ref{fig:nonlinkribbon} is not the all-$A$ ribbon graph of any classical link diagram. In order for a ribbon graph to be the all-$A$ ribbon graph of some classical link diagram, there must be a realization of $\mathbb{G}$ in the form of the middle diagram of Figure \ref{fig:ribbonconstruct} where the vertices of $\mathbb{G}$ are simple closed curves in the plane and the edges of $\mathbb{G}$ are nonintersecting curves with endpoints on the simple closed curves. The ribbon graph in Figure \ref{fig:nonlinkribbon} has no such depiction. However, this ribbon graph is represented arrow presentation in Figure \ref{fig:virtualribbon} and hence is the all-$A$ ribbon graph of the virtual link diagram in Figure \ref{fig:virtualribbon}.

\subsection{Computing $Kh(\mathbb{G})$}
Since $V(\mathbb{G}(I)) = V$ or $V^{\otimes 2}$ for each $I\in \mathcal{V}(3)$, each of our edge maps will either be $m$ or $\Delta$ (without tensoring with the identity on other factors of $V$). Before we begin the computation, set the notation for $CKh^{1,*}(\mathbb{G})$ and $CKh^{2,*}(\mathbb{G})$ by
\begin{eqnarray*}
CKh^{1,*}(\mathbb{G}) & = & V(\mathbb{G}(1,0,0))\oplus V(\mathbb{G}(0,1,0)) \oplus V(\mathbb{G}(0,0,1))\\
CKh^{2,*}(\mathbb{G}) & = & V(\mathbb{G}(1,1,0))\oplus V(\mathbb{G}(1,0,1)) \oplus V(\mathbb{G}(0,1,1)),
\end{eqnarray*}
so that the order of the terms in the direct sum are ordered from top to bottom as they appear in Figure \ref{fig:cube_example}. Note that any subgroup $CKh^{i,j}(\mathbb{G})$ of $CKh^{i,*}(\mathbb{G})$ will share this order of its summands.
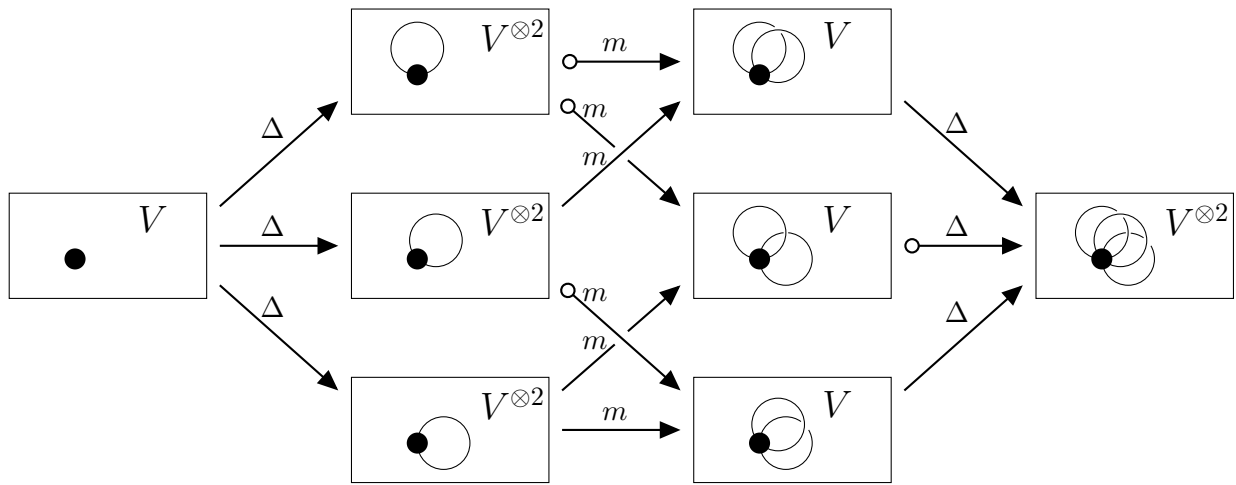
\begin{figure}[h]
$$\begin{tikzpicture}[scale=.35]
\draw (-1.5,-1.5) rectangle (6,2.5);
\begin{scope}[scale=2, xshift = .5cm]
	\fill (0,0) circle (.2cm);
\end{scope}
\draw (3,1.5) node[right]{\Large{$V$}};
\begin{scope}[xshift = 13cm]
\draw (-1.5,-1.5) rectangle (6,2.5);
\begin{scope}[scale=2, xshift = .5cm]
	\fill (0,0) circle (.2cm);
	\draw (0.3536, 0.3536) circle (.5cm);
\end{scope}
\draw (3,1.5) node[right]{\Large{$V^{\otimes 2}$}};
\end{scope}
\begin{scope}[xshift=13cm, yshift=7cm]
\draw (-1.5,-1.5) rectangle (6,2.5);
\begin{scope}[scale=2, xshift = .5cm]
	\fill (0,0) circle (.2cm);
	\draw (0,.5) circle (.5cm);
\end{scope}
\draw (3,1.5) node[right]{\Large{$V^{\otimes 2}$}};
\end{scope}
\begin{scope}[xshift=13cm, yshift=-7cm]
\draw (-1.5,-1.5) rectangle (6,2.5);
\begin{scope}[scale=2, xshift = .5cm]
	\fill (0,0) circle (.2cm);
	\draw (.5,0) circle (.5cm);
\end{scope}
\draw (3,1.5) node[right]{\Large{$V^{\otimes{2}}$}};
\end{scope}
\begin{scope}[xshift=26cm]
\draw (-1.5,-1.5) rectangle (6,2.5);
\begin{scope}[scale=2, xshift = .5cm]
	\fill (0,0) circle (.2cm);
	\draw (0,.5) circle (.5cm);
	
	\draw (0,0) arc (180:95:.5cm);
	\draw(0,0) arc (-180:85:.5cm);
\end{scope}
\draw (3,1.5) node[right]{\Large{$V$}};
\end{scope}
\begin{scope}[xshift=26cm, yshift=7cm]
\draw (-1.5,-1.5) rectangle (6,2.5);
\begin{scope}[scale=2, xshift = .5cm]
	\fill (0,0) circle (.2cm);
	\draw (0.3536, 0.3536) circle (.5cm);
	\draw (0,0) arc (-90:35:.5cm);
	\draw (0,0) arc (270:55:.5cm);
\end{scope}
\draw (3,1.5) node[right]{\Large{$V$}};
\end{scope}
\begin{scope}[xshift=26cm, yshift=-7cm]
\draw (-1.5,-1.5) rectangle (6,2.5);
\begin{scope}[scale=2, xshift = .5cm]
	\fill (0,0) circle (.2cm);
	\draw (0.3536, 0.3536) circle (.5cm);
	\draw (0,0) arc (180:55:.5cm);
	\draw(0,0) arc (-180:30:.5cm);
\end{scope}
\draw (3,1.5) node[right]{\Large{$V$}};
\end{scope}
\begin{scope}[xshift=39cm]
\draw (-1.5,-1.5) rectangle (6,2.5);
\begin{scope}[scale=2, xshift = .5cm]
	\fill (0,0) circle (.2cm);
	\draw (0.3536, 0.3536) circle (.5cm);
	\draw (0,0) arc (-90:35:.5cm);
	\draw (0,0) arc (270:55:.5cm);
	\draw (0,0) arc (180:95:.5cm);
	\draw(0,0) arc (-180:30:.5cm);
	\draw (0.55,0.49) arc (85:55:.5cm);
\end{scope}
\draw (3,1.5) node[right]{\Large{$V^{\otimes 2}$}};
\end{scope}
\begin{scope}[>=triangle 45]
	\draw[->,thick] (6.5,.5) -- (11,.5);
	\draw (8.5,.5) node[above]{$\Delta$};
	\draw[->,thick] (6.5,2) -- (11, 6);
	\draw (8.5, 4.2) node[above]{$\Delta$};
	\draw[->,thick] (6.5,-1) -- (11,-5);
	\draw (8.5, -2.6) node[above]{$\Delta$};
	\begin{scope}[xshift = 13cm]
		\draw[->,thick] (6.5,2) -- (11,6);
		\draw (7.75,3.2) node[above]{$m$};
		\draw[o->,thick](6.5,-1)--(11,-5);
		\draw (7.75,-2) node[above]{$m$};
	\end{scope}
	\begin{scope}[xshift=13cm,yshift = 7cm]
		\draw[o->,thick](6.5,.5)--(11,.5);
		\draw (8.5,.5) node[above]{$m$};
		\draw[o-,thick] (6.5,-1) -- (8.5,-2.75);
		\draw (7.75,-2) node[above]{$m$};
		\draw[->,thick] (9,-3.25)--(11,-5);
	\end{scope}
	\begin{scope}[xshift=13cm, yshift=-7cm]
		\draw[->,thick](6.5,.5) -- (11,.5);
		\draw (8.5,.5) node[above]{$m$};
		\draw[thick] (6.5,2) -- (8.5,3.75);
		\draw (7.75,3.2) node[above]{$m$};
		\draw[->,thick](9,4.25) -- (11,6);
	\end{scope}
	\begin{scope}[xshift = 26cm, yshift=7cm]
		\draw[->,thick](6.5,-1) -- (11,-5);
		\draw (8.5, -2.6) node[above]{$\Delta$};
	\end{scope}
	\begin{scope}[xshift = 26cm]
		\draw[o->,thick] (6.5,.5) -- (11,.5);
		\draw (8.5,.5) node[above]{$\Delta$};
	\end{scope}
	\begin{scope}[xshift = 26cm, yshift = -7cm]
		\draw[->,thick] (6.5,2) -- (11, 6);
		\draw (8.5, 4.2) node[above]{$\Delta$};
	\end{scope}
\end{scope}
\end{tikzpicture}$$
\caption{The hypercube complex for the ribbon graph $\mathbb{G}$ in Figure \ref{fig:nonlinkribbon}.}
\label{fig:cube_example}
\end{figure}

We proceed by computing the complexes for each nontrivial polynomial grading. The $\mathbb{Z}$-module $CKh^{0,-1}(\mathbb{G})$ has basis $\{v_-\}$ and the $\mathbb{Z}$-module $CKh^{1,-1}(\mathbb{G})$ has basis $\{(v_-\otimes v_-,0,0),(0,v_-\otimes v_-, 0), (0,v_-\otimes v_-)\}$. For all other values of $i$, the $\mathbb{Z}$-module $CKh^{i,j}(\mathbb{G})$ is trivial, and so our complex for $j=-1$ is 
$$0\to \mathbb{Z} \xrightarrow{d^{0,-1}} \mathbb{Z}^3 \to 0.$$
Since $d^{0,-1}(v_-) = \Delta\oplus\Delta\oplus\Delta(v_-) = (v_-\otimes v_-,v_-\otimes v_-,v_-\otimes v_-)$, it follows that $Kh^{0,-1}(\mathbb{G}) = 0$ and $Kh^{1,-1}(\mathbb{G}) = \mathbb{Z}^2$.

The $\mathbb{Z}$-module $CKh^{0,1}(\mathbb{G})$ has basis $\{v_+\}$, the $\mathbb{Z}$-module $CKh^{1,1}(\mathbb{G})$ has basis $\{(v_{\pm}\otimes v_{\mp},0,0), (0, v_{\pm}\otimes v_{\mp},0), (0,0,v_{\pm}\otimes v_{\mp})\},$ the $\mathbb{Z}$-module $CKh^{2,1}(\mathbb{G})$ has basis $\{(v_-,0,0),(0,v_-,0),(0,0,v_-)\}$, and the $\mathbb{Z}$-module $CKh^{3,1}(\mathbb{G})$ has basis $\{v_-\}$. Thus our complex for $j=1$ is 
$$0\to \mathbb{Z} \xrightarrow{d^{0,1}} \mathbb{Z}^6 \xrightarrow{d^{1,1}} \mathbb{Z}^3\xrightarrow{d^{2,1}} \mathbb{Z}\to 0.$$
The kernel of $d^{0,1}$ is trivial, and thus $Kh^{0,1}(\mathbb{G})=0$. The image of $d^{0,1}$ has basis $\{(v_+\otimes v_- + v_-\otimes v_+,v_+\otimes v_- + v_-\otimes v_+,v_+\otimes v_- + v_-\otimes v_+)\}$, while the kernel of $d^{1,1}$ has basis $\{(v_+\otimes v_- + v_-\otimes v_+,v_+\otimes v_- + v_-\otimes v_+,v_+\otimes v_- + v_-\otimes v_+), (v_+\otimes v_- - v_-\otimes v_+,0,0),(0,v_+\otimes v_- - v_-\otimes v_+,0),(0,0,v_+\otimes v_- - v_-\otimes v_+)\}$. Therefore $Kh^{1,1}(\mathbb{G}) = \mathbb{Z}^3$. It can also be shown that the kernel of $d^{2,1}$ is equal to the image of $d^{1,1}$, and hence $Kh^{2,1}(\mathbb{G}) = 0$ and that the kernel of $d^{3,1}$ is equal to the image of $d^{2,1}$, and hence $Kh^{3,1}(\mathbb{G})=0$.

The $\mathbb{Z}$-module $CKh^{1,3}(\mathbb{G})$ has basis $\{(v_+\otimes v_+,0,0),(0,v_+\otimes v_+,0),(0,0,v_+\otimes v_+)\}$, the $\mathbb{Z}$-module $CKh^{2,3}(\mathbb{G})$ has basis $\{(v_+,0,0)(0,v_+,0),(0,0,v_+)\}$, and the $\mathbb{Z}$-module $CKh^{3,3}(\mathbb{G})$ has basis $\{v_+\otimes v_-,v_-\otimes v_+\}$.  Thus the complex for $j=3$ is
$$0\to \mathbb{Z}^3\xrightarrow{d^{1,3}}\mathbb{Z}^3\xrightarrow{d^{2,3}}\mathbb{Z}^2\to 0.$$
The kernel of $d^{1,3}$ has basis $\{(v_+\otimes v_+,v_+\otimes v_+,v_+\otimes v_+)\}$, and hence $Kh^{1,3}(\mathbb{G})=\mathbb{Z}$. Both the kernel of $d^{2,3}$ and the image of $d^{1,3}$ have basis $\{(v_+,v_+,0),(v_+,0,-v_+)\}$, and hence $Kh^{2,3}(\mathbb{G})=0$. The kernel of $d^{3,3}$ is all of $CKh^{3,3}(\mathbb{G})$ and the image of $d^{2,3}$ has basis $\{v_+\otimes v_-+v_-\otimes v_+\}$. Therefore, $Kh^{3,3}(\mathbb{G}) = \mathbb{Z}$. 

The $\mathbb{Z}$-module $CKh^{3,5}(\mathbb{G})$ has basis $\{v_+\otimes v_+\}$, and for all $i\neq 3$, the $\mathbb{Z}$-module $CKh^{i,5}(\mathbb{G})$ is trivial. Therefore $Kh^{3,5}(\mathbb{G})=\mathbb{Z}$. 

Table \ref{table:khovanov} shows the results of our calculation. The rows correspond to polynomial grading, and the columns correspond to the homological grading. If an entry is left blank, then that summand is trivial. 
\begin{table}[h]
\begin{tabular}{|r || l | l | l |}
\hline
$j\backslash i$ & $1$ & $2$ & $3$\\
\hline
\hline
$5$ & & & $\mathbb{Z}$\\
\hline
$3$ & $\mathbb{Z}$ & & $\mathbb{Z}$\\
\hline
$1$ & $\mathbb{Z}^3$ & & \\
\hline
$-1$ & $\mathbb{Z}^2$ & & \\
\hline
\end{tabular}
\caption{The Khovanov homology $Kh(\mathbb{G})$ of the ribbon graph $\mathbb{G}$ from Figure \ref{fig:nonlinkribbon}.}
\label{table:khovanov}
\end{table}

The Khovanov homology of this ribbon graph does not contain any torsion. Asaeda and Przytycki \cite{AsaedaPrzytycki:Thickness} proved a conjecture of Shumakovitch which states the all nonsplit alternating links except the unknot, the Hopf link, and connected sums of Hopf links contain torsion of order two. In the language of ribbon graphs, Asaeda and Przytycki's theorem says that all connected, genus zero ribbon graphs without loops or vertices of degree one except the graph consisting of a single vertex and arbitrary $1$-sums of the graph with two vertices and two parallel edges between them contains two torsion. Recently, Shumakovitch \cite{Shumakovitch:Torsion} proved that nonsplit alternating links do not have torsion of odd order, and conjectured that Asaeda's and Prztycki's result should extend to non-alternating links as well. Our example indicates that examining the torsion for non-planar ribbon graphs could be an interesting question.

\subsection{Computing $\widetilde{Kh}(\mathbb{G})$}
Instead of marking a point on the boundary of each $\Sigma_{\mathbb{G}(I)}$ and computing the subcomplex of the cube of resolutions complex, we examine the spanning quasi-tree complex and see that the differential must be zero. Let $\mathbb{T}_1$, $\mathbb{T}_2$, $\mathbb{T}_3$, and $\mathbb{T}_4$ be the four spanning quasi-trees depicted in Figure \ref{fig:exampleqtrees}.
\begin{figure}[h]
$$\begin{tikzpicture}[scale=1]
\fill (0,0) circle (.2cm);
\draw (0,-.5) node[below]{\Large{$\mathbb{T}_1$}};
\begin{scope}[xshift = 2.5cm]
	\fill (0,0) circle (.2cm);
	\draw (0.3536, 0.3536) circle (.5cm);
	\draw (0,0) arc (-90:35:.5cm);
	\draw (0,0) arc (270:55:.5cm);
	\draw (0,-.5) node[below]{\Large{$\mathbb{T}_2$}};
\end{scope}

\begin{scope}[xshift = 5cm]
	\fill (0,0) circle (.2cm);
	\draw (0,.5) circle (.5cm);
	
	\draw (0,0) arc (180:95:.5cm);
	\draw(0,0) arc (-180:85:.5cm);
	\draw (0,-.5) node[below]{\Large{$\mathbb{T}_3$}};
\end{scope}

\begin{scope}[xshift = 7.5cm]
	\fill (0,0) circle (.2cm);
	\draw (0.3536, 0.3536) circle (.5cm);
	\draw (0,0) arc (180:55:.5cm);
	\draw(0,0) arc (-180:30:.5cm);
	\draw (0,-.5) node[below]{\Large{$\mathbb{T}_4$}};
\end{scope}
\end{tikzpicture}$$
\caption{The four spanning quasi-trees of $\mathbb{G}$.}
\label{fig:exampleqtrees}
\end{figure}
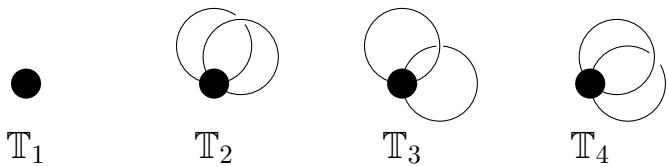
After constructing chord diagrams from each spanning quasi-tree (as in Subsection \ref{subsec:activities}), one can compute the number of internally active and externally active edges in $\mathbb{T}_k$ for each spanning quasi-tree using the edge ordering given in Figure \ref{fig:nonlinkribbon}. From this information, the homological and polynomial grading of each spanning quasi-tree can be computed. Table \ref{table:grading} shows the genus, activity information, and gradings of each of the four spanning quasi-trees.
\begin{table}[h]
\begin{tabular}{| l | c | c | c | c | c |}
\hline
$k$ & $g(\mathbb{T}_k)$ & $ia(\mathbb{T}_k)$ & $ea(\mathbb{T}_k)$ & $i(\mathbb{T}_k)$ & $j(\mathbb{T}_k)$\\
\hline
\hline
$1$ & $0$ & $0$ & $1$ & $1$ & $2$\\
\hline
$2$ & $1$ & $1$ & $0$ & $1$ & $0$\\
\hline
$3$ & $1$ & $1$ & $0$ & $1$ & $0$\\
\hline
$4$ & $1$ & $0$ & $1$ & $3$ & $4$\\
\hline
\end{tabular}
\caption{The activities and bigrading for each of the four spanning quasi-trees of $\mathbb{G}$.}
\label{table:grading}
\end{table}

As Table \ref{table:grading} indicates, there does not exist two spanning quasi-tree $\mathbb{T}_k$ and $\mathbb{T}_l$ with $j(\mathbb{T}_k) = j(\mathbb{T}_l)$ and $i(\mathbb{T}_k) = i(\mathbb{T}_l) - 1$. Therefore, the differential of the reduced Khovanov complex (for any choice of basepoint) must be zero, and the reduced Khovanov homology of $\mathbb{G}$ has a summand of $\mathbb{Z}$ corresponding to each spanning tree. Table \ref{table:reduced} shows the reduced Khovanov homology of $\mathbb{G}$.
\begin{table}[h]
\begin{tabular}{|r || l | l | l |}
\hline
$j\backslash i$ & $1$ & $2$ & $3$\\
\hline
\hline
$4$ & & & $\mathbb{Z}$\\
\hline
$2$ & $\mathbb{Z}$ & & \\
\hline
$0$ & $\mathbb{Z}^2$ & & \\
\hline
\end{tabular}
\caption{The reduced Khovanov homology $\widetilde{Kh}(\mathbb{G})$ of the ribbon graph $\mathbb{G}$ from Figure \ref{fig:nonlinkribbon}.}
\label{table:reduced}
\end{table}

\newpage
\bibliography{RibbonGraphHomology}
\bibliographystyle {amsalpha}

\end{document}